\documentclass[10pt, reqno]{amsart}

\usepackage[utf8]{inputenc}
\usepackage[margin=1.2in]{geometry}

\usepackage[T1]{fontenc}
\usepackage[english]{babel}
\usepackage{amssymb, amsmath, amsthm, mathrsfs}
\usepackage[dvipsnames]{xcolor}
\usepackage{amsfonts}%
\usepackage{enumitem}
\usepackage{mathtools}
\usepackage{stmaryrd}
\usepackage{thmtools}
\usepackage{comment}

\DeclareFontFamily{U}{stmry}{}
\DeclareFontShape{U}{stmry}{m}{n}{<-> stmary10}{}
\DeclareFontShape{U}{stmry}{b}{n}{<-> ssub * stmry/m/n}{}
\usepackage{hyperref}
\usepackage{amscd}
\usepackage{tikz-cd}
\usetikzlibrary{arrows.meta,positioning,calc}
\usepackage{caption}

\newcommand{\dashbullet}{\raisebox{0.2ex}{\rule{0.8em}{0.4pt}}}

\newlist{subtlelist}{itemize}{1}
\setlist[subtlelist,1]{
  label=\dashbullet,
  labelsep=0.6em,
  leftmargin=2em,
  itemsep=0.25\baselineskip,
  topsep=0.5\baselineskip
}

\hypersetup{
    colorlinks=true,
    linkcolor=NavyBlue,
    citecolor=NavyBlue,
    urlcolor=NavyBlue
}

\theoremstyle{plain}
\newtheorem{theorem}{Theorem}[section]
\newtheorem{proposition}[theorem]{Proposition}
\newtheorem{lemma}[theorem]{Lemma}
\newtheorem{corollary}[theorem]{Corollary}

\newtheorem{example}[theorem]{Example}

\theoremstyle{definition}
\newtheorem{definition}[theorem]{Definition}
\newtheorem{remark}[theorem]{Remark}
\newtheorem{axiom}[theorem]{Structural axiom}

\theoremstyle{plain}
\newtheorem{hypothesis}[theorem]{Hypothesis}
\newtheorem{conjecture}[theorem]{Conjecture}

\makeatletter
\providecommand*{\toclevel@theorem}{0}
\providecommand*{\toclevel@proposition}{0}
\providecommand*{\toclevel@lemma}{0}
\providecommand*{\toclevel@corollary}{0}
\providecommand*{\toclevel@claim}{0}
\providecommand*{\toclevel@example}{0}
\providecommand*{\toclevel@definition}{0}
\providecommand*{\toclevel@remark}{0}
\providecommand*{\toclevel@axiom}{0}
\providecommand*{\toclevel@observation}{0}
\providecommand*{\toclevel@protocol}{0}
\providecommand*{\toclevel@inquiry}{0}
\providecommand*{\toclevel@hypothesis}{0}
\providecommand*{\toclevel@conjecture}{0}
\makeatother

\newcommand{\Q}{\mathbb{Q}}
\newcommand{\Z}{\mathbb{Z}}
\newcommand{\R}{\mathbb{R}}
\newcommand{\C}{\mathbb{C}}
\newcommand{\BSD}{\mathsf{BSD}}
\newcommand{\ZF}{\mathsf{ZF}}

\newcommand{\Ho}{\operatorname{Ho}}
\newcommand{\Reg}{\operatorname{Reg}}

\newcommand{\defect}{\delta_{\mathrm{mod}}}

\DeclareFontFamily{U}{wncy}{}
\DeclareFontShape{U}{wncy}{m}{n}{<->wncyr10}{}
\DeclareSymbolFont{mcy}{U}{wncy}{m}{n}
\DeclareMathSymbol{\Sha}{\mathord}{mcy}{"58}
\newcommand{\Sel}{\operatorname{Sel}}

\newcommand{\parahead}[1]{%
  \par\medskip
  \noindent\textit{#1.}\enspace
}

\begin{document}

\title[Intensional semantics for comparison problems]{Intensional semantics for comparison problems in arithmetic geometry}
\date{\today}
\author{R. Laniewski}


\begin{abstract}
This work proposes a semantic environment for arithmetic geometry in which passage between presentations carries an explicit weight. Morphism costs induce a Lawvere generalized distance by taking infima over transports, and heights satisfy the corresponding transport bounds. A labeled quantitative $(2,1)$-category retains invertible $2$-morphisms as coherent comparisons between transports. Its vertical groupoids admit a transport interpretation inspired by the identity types of Voevodsky's C-systems. If a height factors through the underlying object, then a vertical comparison has equal endpoints, whatever its transport cost.

We formulate an intensional Szpiro inequality for elliptic packages, where the discriminant height transports along a morphism with a controlled model defect, and we relate a bound of this form to the $abc$ conjecture. The defect remains visible between non-minimal and normalized presentations, while conductor complexity stays attached to the underlying curve.

We then treat the Birch and Swinnerton-Dyer conjecture in its pre-modularity form, using strong point-count asymptotics with an explicit convergence requirement. The refined constant retains separately typed arithmetic factors. A finite differential correction makes the period independent of the chosen differential and links its normalization to model transport. Within the declared finite-descent theory, Tate--Shafarevich finiteness reduces to finitely many closed prime columns of Selmer shadows and a uniform cutoff. For $y^2=x^3-q^2x$ with $q\equiv3\pmod8$ prime, two-isogeny descent closes the two-primary column uniformly. The statements admit expression in $\mathsf{ACA}_0$ relative to certified descent and Mordell--Weil data, with certified Cauchy names for the period and regulator.
\end{abstract}

\maketitle

\tableofcontents

\section{Introduction}\label{sec:intro}

A convention governs much of contemporary arithmetic geometry, so familiar that it is rarely stated. Once two number fields, two schemes, or two elliptic curves are known to be isomorphic, and once an isomorphism between them has been chosen, one treats them as the same object of discourse. Heights and volumes are then required to be insensitive to the route by which the identification was reached. This discipline is convenient, and for a great many purposes it is the correct one. The convention remains a choice nonetheless, and it carries a cost. It discards the history of a construction, and with that history a quantity that, in certain arithmetic questions, one would wish to measure.

We take that quantity as our subject. We develop a semantic environment in which the passage from one presentation of an arithmetic object to another carries an explicit weight, and in which a height is permitted to change along such a passage by an amount bounded by that weight. Nothing here departs from established mathematics. We draw together three lines of thought, each of long standing, and observe that they describe one structure.

The first line is the theory of enriched categories in the form given to it by Lawvere~\cite{Lawvere1973}. A generalized metric space, in his reading, is nothing other than a small category enriched over the monoidal poset $([0,\infty],\ge,+)$, with the triangle inequality appearing as the compositional law and with no requirement of symmetry. A category whose morphisms carry a non-negative cost, subadditive under composition and vanishing on identities, is such an enriched structure. The distance from one object to another becomes the least cost of a transport between them, and a height becomes a non-expansive map controlled by this metric. This is the language in which the weight of a transport is expressed.

The second line is the theory of C-systems and univalent foundations developed by Voevodsky~\cite{voevodsky2023martin}, together with the $\infty$-topos semantics of Riehl~\cite{Riehl2024}. In a C-system the identity type $\mathrm{Id}_A(a,b)$ between two terms is itself a type rather than a bare proposition, its inhabitants witnessing particular paths of identification, and its J-structure regulating how an identification may be transported or contracted. When the system is non-trivial in the homotopical sense, two terms may be identified by paths that carry distinct content. This is the language in which we register that an identification is a structure rather than a redundancy.

The third line is the resource-sensitive type theory of Atkey~\cite{Atkey2018}, whose quantitative categories with families enrich the standard semantics of Martin-L\"of type theory by a semiring of usage annotations. The role of the semiring is played here by the Lawvere quantale $([0,\infty],\ge,+)$, and the usage of a transport is its cost. This is the language in which the metric and the identity type are reconciled, so that transport along a path of zero usage preserves a height exactly, while transport along a path of positive usage incurs a bounded penalty.

These three lines meet inside a single object, which we term a labeled quantitative $(2,1)$-category. Its objects are arithmetic data carrying labels and construction histories. Its $1$-morphisms are transports endowed with a cost. Its $2$-morphisms are invertible identifications between transports, present so that two transports differing only by a coherent change of auxiliary choice are never forced into literal equality, yet silent in the numerical account, since the cost is required to depend on a transport only up to such a $2$-isomorphism. Over this structure sits a forgetful functor to the extensional category of ordinary arithmetic geometry, and the fiber over a fixed object -- its vertical groupoid -- houses the pure intensionality. In the extensional world every transport is free. In the intensional world the fiber retains positive distances, and the framework measures them.

We can now say what the device contributes to the study of arithmetic comparison statements. A comparison between two presentations of one underlying object is a transport inequality in the labeled quantitative $(2,1)$-category. We isolate, in a form requiring no special geometry, the exact condition under which such an inequality descends to a non-trivial inequality between extensional invariants. If one demands that the descent be functorial with respect to extensional isomorphisms and transparent to the labels at the end, then the two sides of the comparison are forced to be evaluations of one invariant on one object, and the inequality collapses to the harmless form $H \le H + \Delta$. This collapse occurs whatever cost the transport carries. A comparison statement retains arithmetic content only when it relates two different invariants, such as the discriminant height and the conductor complexity below, while the transport cost keeps the change of presentation visible along the way.

The remainder of the work carries this reading to a Diophantine question and then grounds it. We introduce elliptic packages, in which an elliptic curve is presented together with a chosen model, local data, and a label. A morphism of packages transports the discriminant height along a change of Weierstrass coordinates, and the cost of the transport -- the model defect -- bounds the increase of the height. Within this quantitative category we state an intensional Szpiro inequality, in which the retained defect appears as a necessary term, and we describe how a bound of this form relates to the $abc$ conjecture through the classical reduction of Szpiro~\cite{Szpiro1990} and Masser--Oesterl\'e~\cite{MasserOesterle1985}. We then descend the account to a pre-set-theoretic arithmetic universe formalizable in $\mathsf{ACA}_0$~\cite{simpson2009subsystems}, where the inequalities become comparisons between finite tuples of integers. A small formalization in the Coq system~\cite{Bertot2004} exhibits the grammar of the framework inside a kernel, keeping the $2$-categorical level explicit while the scalar bounds attach to $1$-morphisms alone.

The same vocabulary is applied to the Birch and Swinnerton-Dyer conjecture in Sections~\ref{sec:BSD}, \ref{sec:TateShafarevich}, \ref{sec:ShaSelmerShadows}, and~\ref{sec:BoundedShaPrimeTwo}. The rank form is phrased through the growth of finite point-count products. In the refined form, the regulator, the period, the local Tamagawa numbers, the torsion order, and the Shafarevich factor remain separately typed. Quantities whose logarithms may have either sign are read as signed transport potentials, represented by pairs of directed non-negative costs. The finite Shafarevich shadows are read through Selmer packets, and their successive indices define a canonical quantitative category.

We do not assert the intensional Szpiro inequality, the Birch and Swinnerton-Dyer conjecture, or the finiteness of the Tate--Shafarevich group. The arithmetic hypotheses are named where they enter. The contribution on the Shafarevich side is an exact reduction within a declared finite-descent theory. Finiteness is equivalent to the closure of finitely many prime columns together with a cutoff beyond which the level-one Selmer excess vanishes. The vertical closures are finite certificates. The horizontal cutoff remains the uniform arithmetic statement that must be supplied. For the curves $y^2=x^3-q^2x$ with $q\equiv3\pmod 8$, a two-isogeny descent closes the whole two-primary column, uniformly in the prime $q$ and in the depth. This division keeps the foundational status of each assertion visible and does not appeal to a completed universe of arbitrary sets.

\parahead{Orientation} Throughout, two environments recur, and it is convenient to name them at the outset. Universe~H is the Hilbertian and extensional background, where a single model of $\ZF$ is fixed, isomorphic objects are identified once an isomorphism has been chosen, and heights depend only on isomorphism classes. Universe~A is the intensional and quantitative background, where objects carry labels and construction histories, where identifications carry a cost, and where the logical distance between two isomorphic copies may register as a datum. A third environment, Universe~P, appears later as the pre-set-theoretic arithmetic recipient of the framework, formalizable in $\mathsf{ACA}_0$. The letters serve only as signposts. The mathematics resides in the structures they label.

\section{Two disciplines of identification}\label{sec:identification}

Before any cost can be measured, one must agree on when two presentations are to count as the same object of discourse. Two disciplines are available, and the distance between them is the subject of this work.

\begin{definition}[Static and dynamic identification]\label{def:StaticDynamic}
Let $\mathcal{C}$ be a category of arithmetic objects such as number fields or schemes.

\smallskip
\noindent \textit{Static identification.} Under \emph{static identification}, one works, for Diophantine purposes, with isomorphism classes in $\mathcal{C}$. Once two objects $X$ and $Y$ of $\mathcal{C}$ are known to be isomorphic, they may be interchanged without comment and treated as the same object of discourse.

\smallskip
\noindent \textit{Dynamic identification.} Under \emph{dynamic identification}, the atomic object of discourse comprises both the underlying entity $X$ of $\mathcal{C}$ and a label encoding its contextual origin. Two labeled objects $(X,\ell_X)$ and $(Y,\ell_Y)$ keep their semantic separation within formulas regardless of any structural isomorphism $X \cong Y$ in $\mathcal{C}$, so that intensional distinctions persist as mathematical content.

Static identification reflects the prevailing standpoint of arithmetic geometry. Dynamic identification models a discipline in which distinct labeled copies of arithmetic data are kept apart even where the underlying objects coincide up to isomorphism.
\end{definition}

\begin{remark}[Structural rather than arbitrary identifications]\label{rem:ArbitraryStructural}
The phrase ``removing labels'' is to be understood in a purely structural sense. The background environment $\mathcal{S}$ is a fixed model of set-theoretic semantics for arithmetic geometry. Within $\mathcal{S}$, static identification applies only along isomorphisms preserving the relevant structure. When two number fields, schemes, or elliptic curves are isomorphic, and once an isomorphism has been fixed, one recovers the familiar convention by which these objects are treated as identical for Diophantine purposes.

Outside this scope, no identifications are imposed. In particular, the arguments do not rely on collapsing arbitrary distinct elements of a ring, nor on fusing nested universes in ways that would violate the axioms of $\ZF$. When the text speaks of forgetting labels, it refers to discarding auxiliary names and indices once a structural or protocol-determined identification has been fixed in $\mathcal{S}$.
\end{remark}

\begin{definition}[Standard and labeled signatures]\label{def:Signatures}
Within the background environment $\mathcal{S}$, the \emph{standard signature} refers to the regime in which one works with static identification together with ring coherence as in Axiom~\ref{ax:RingCoherence}. In this regime, isomorphic copies of a number field or scheme are not distinguished once an identification has been fixed, and quantifiers range over underlying objects of a $\ZF$-style universe.

The \emph{labeled signature} refers to the regime in which dynamic identification is maintained. Arithmetic objects carry explicit context labels, and not every isomorphism is available as an identification inside formulas. Ring coherence may then be suspended locally and partly reconstituted by auxiliary devices.

The two signatures share the same informal arithmetic vocabulary. They differ in the discipline that governs when two presentations may be regarded as the same object of discourse.
\end{definition}

\begin{definition}[Protocol-determined data]\label{def:ProtocolDetermined}
Fix a semantic environment $\mathcal{S}$. By a \emph{protocol} we mean an explicit rule of construction, written as finite mathematical instructions, that takes specified input data and outputs specified objects or morphisms.

A datum, for instance a family of local identifications or a gluing map, will be called \emph{protocol-determined in $\mathcal{S}$} when it is produced by a declared protocol from declared input data, and when its standing as the output of the protocol does not rely on an ambient uniqueness statement in the sense of static identification.

In this work, the phrase ``protocol-determined'' is used in places where the word ``canonical'' would invite an extensional reading that is not intended. Inside the extensional environment $\mathcal{S}$, an expression such as ``canonical'' often carries an implicit promise of uniqueness up to a unique isomorphism, together with insensitivity to the route by which an identification is reached. In the labeled signature the standing of a construction comes instead from a declared derivation protocol and from the construction history of the objects being transported, so that the output is singled out by the fact that the protocol produces it, and not by an ambient uniqueness principle that erases the history of transport.
\end{definition}

\begin{remark}[The divergence of the two disciplines]\label{rem:divergence}
A comparison statement, read in the standard signature, identifies its two sides as soon as their underlying objects are isomorphic, so that any apparent discrepancy between them is absorbed into the chosen isomorphism. The same statement, read in the labeled signature, keeps the two sides apart, and any discrepancy between them may survive as a measured quantity. The whole of this work lives in the space between these two readings, and its purpose is to give the second reading a precise categorical form together with a conservative arithmetic shadow.
\end{remark}

\section{Ring coherence and suspended evaluation}\label{sec:ring-coherence}

The informal assumption that underlies much of conventional arithmetic geometry may be described, for present purposes, as a principle of \emph{ring coherence}. One starts with a single ambient ring and expects the arithmetic quantities under discussion to derive from this ring in a compatible way. Heights, log-volumes, and similar expressions are then regarded as maps of ring elements that respect both the additive and the multiplicative structures. This coherence is rarely stated, yet it governs many habitual identifications.

\begin{axiom}[Ring coherence]\label{ax:RingCoherence}
Let $R$ be a commutative ring with identity. Arithmetic invariants that depend simultaneously on the additive and multiplicative structures of $R$ are taken to be defined on the underlying set of $R$ in a manner compatible with both structures. Once an element $x \in R$ has been fixed, all such invariants take values that depend only on $x$ as a ring element, and not on any auxiliary presentation or labeling of $x$.
\end{axiom}

The content of Axiom~\ref{ax:RingCoherence} may be unfolded by separating the symbolic level of arithmetic expressions from the level of their evaluation in a fixed ring, and then imposing a compatibility requirement between the two levels.

\begin{definition}[Arithmetic term category]\label{def:ArithmeticTermCategory}
Let $\mathcal{T}_\mathrm{arith}$ be the small category defined as follows.

\begin{enumerate}
\item[(i)] Objects are finite lists of formal variables
  \[
    X = (x_1,\dots,x_n)
  \]
with $n \geq 0$. The empty list is allowed and represents a nullary context.

\item[(ii)] A morphism
  \[
    \tau \colon (x_1,\dots,x_n) \longrightarrow (y_1,\dots,y_m)
  \]
is an $m$-tuple of formal expressions in the variables $x_1,\dots,x_n$ built from the symbols $0,1$, the binary operations $+$ and $\cdot$, and the unary operation $-$, using finitely many applications of these operations. Composition is defined by substitution of expressions.
\end{enumerate}
\end{definition}

An object of $\mathcal{T}_\mathrm{arith}$ represents a choice of formal parameters, while a morphism represents a vector of arithmetic terms in these parameters. To connect this with a fixed ring, one introduces an evaluation functor.

\begin{definition}[Evaluation functor]\label{def:EvaluationFunctor}
Let $R$ be a commutative ring with identity. The evaluation functor
\[
  E_R \colon \mathcal{T}_\mathrm{arith} \longrightarrow \mathbf{Set}
\]
is defined as follows.

\begin{enumerate}
\item[(i)] On objects, $E_R(x_1,\dots,x_n) = R^n$, with the understanding that $E_R(\varnothing)$ is a singleton set.

\item[(ii)] A morphism $\tau \colon (x_1,\dots,x_n) \to (y_1,\dots,y_m)$, corresponding to expressions $(t_1,\dots,t_m)$, induces
  \[
    E_R(\tau) \colon R^n \longrightarrow R^m,
    \qquad
    (a_1,\dots,a_n) \longmapsto \bigl(t_1(a),\dots,t_m(a)\bigr)
  \]
where the right-hand side is computed using the ring operations of $R$.
\end{enumerate}
\end{definition}

The symbolic calculus of $\mathcal{T}_\mathrm{arith}$ is thereby interpreted inside a fixed ring through $E_R$. The principle of ring coherence may then be expressed as the requirement that arithmetic quantities respect this interpretation.

\begin{definition}[Arithmetic quantities]\label{def:MultiAssignment}
Let $R$ be a commutative ring. An \emph{arithmetic quantity} on $R$ in $n$ variables is a map $Q \colon R^n \to V$, where $V$ is an abelian group or a real vector space of values. A \emph{multi-assignment of arithmetic quantities} on $R$ consists of a choice, for each object $X = (x_1,\dots,x_n)$ of $\mathcal{T}_\mathrm{arith}$, of a family $\mathcal{Q}_X$ of arithmetic quantities $Q \colon E_R(X) = R^n \to V_Q$ with various value groups $V_Q$.
\end{definition}

In practice $Q$ may be a height, a log-volume, or an expression that aggregates local contributions attached to a tuple of elements of $R$. The coherence condition expresses the expectation that equivalent symbolic presentations lead to equal values of such quantities.

\begin{definition}[Categorical ring coherence]\label{def:RingCoherenceCategorical}
A multi-assignment of arithmetic quantities on $R$ satisfies \emph{categorical ring coherence} when for every morphism $\tau \colon X \to Y$ of $\mathcal{T}_\mathrm{arith}$ and every quantity $Q \in \mathcal{Q}_Y$, the composite $Q \circ E_R(\tau)$ belongs to the family $\mathcal{Q}_X$. That is, the value of an arithmetic quantity attached to a tuple of terms depends only on the elements of $R$ that these terms denote, and not on the symbolic route by which they were assembled.
\end{definition}

\begin{remark}[Suspended evaluation]\label{rem:SuspendedEvaluation}
The labeled signature of Definition~\ref{def:Signatures} suspends categorical ring coherence locally. Two symbolic routes that denote the same element of $R$ are no longer obliged to yield the same value, because the route itself -- the label, the construction history -- has been admitted as part of the datum. The framework of the following sections may be read as a disciplined way of keeping such suspended evaluations, by attaching to each symbolic route a cost, so that two routes denoting the same element differ by an amount controlled by the distance between their histories.
\end{remark}

\section{Foundational motivations for quantitative identity}\label{sec:FoundationalIdentity}

The proposal that the logical distance between two isomorphic copies of an arithmetic object need not be vacuous stands at variance with the formalism that governs current practice. It is therefore fitting to recall that alternative attitudes toward identity have a long and well-established history, and that the quantitative discipline pursued here is less unfamiliar than it may at first appear.

\subsection{Hilbertian formalism and the exclusion of history}

Within the classical Hilbertian perspective~\cite{vonHeijenoort1967Hilbert}, a formal theory is constituted by a stock of symbols, a grammar of formation, and a set of derivation rules, while the semantic interpretation is a supplementary apparatus added once the formal system has been fixed. Practice that aligns with this attitude treats a structure $\mathcal{X}$ strictly as its isomorphism class $[\mathcal{X}]$. Once two models $M$ and $M'$ are proved isomorphic, no room remains for a distinction between them, and the isomorphism $\psi \colon M \xrightarrow{\sim} M'$ behaves as a transparent path.

In such an environment, any height or volume $H$ on the category of arithmetic objects must satisfy the extensionality principle
\[
  X \cong Y \implies H(X) = H(Y)
\]
A distance between $X$ and $Y$ is forbidden, regarded either as identically zero or as a residue of presentation that must be purged once the formalism has been expressed in invariant terms. This is a coherent and powerful stance. It is also, as we observe in Section~\ref{sec:Comparison}, exactly the stance under which a comparison between two presentations collapses to a tautology.

\subsection{Pre-Hilbertian concerns: sense, reference, and construction history}
\label{subsec:PreHilbert}

The distinction between extensional sameness and intensional identity has a persistent ancestry, running from the \textit{Sinn} and \textit{Bedeutung} distinction of Frege~\cite{Frege1892}, through the constructive epistemology of Brouwer~\cite{Brouwer1908}, to the intensional type theory of Martin-L\"of~\cite{MartinLof1975}. Within this lineage, the membership of a mathematical object in an extensional set does not exhaust its identity, which is constituted also by the manner in which the object was introduced and by the history of the constructions leading to it. An element $a$ of a type $A$ remains distinct from an element $b$ of $A$ unless they share the same introductory form, even where they eventually yield the same observable value.

This constructive lineage finds a contemporary realization in the theory of C-systems and univalent foundations~\cite{voevodsky2023martin}. There the interpretation of identity is governed by the J-structure, a mechanism that treats equality as a structure open to transport. The formal distance between syntactic C-systems, which carry the rich data of derivation, and semantic C-systems, which reside in categories such as sets or sheaves, measures the price one pays when attempting to collapse intensional history into extensional semantics.

\begin{remark}[Identity types as paths]\label{rem:IntensionalRealignment}
Within the homotopy-theoretic perspective, identity types behave as paths rather than as predicates. The effort required to construct non-degenerate J-structures on semantic categories, as documented by Voevodsky~\cite{voevodsky2023martin} for the passage from syntactic to semantic C-systems, shows that intensional identity must be installed in a model through explicit architectural choices. It does not follow inferentially from the mere fixation of an underlying universe.

The work of Riehl~\cite{Riehl2024} on the $\infty$-topos semantics of homotopy type theory affords a further confirmation. There the univalence axiom identifies the identity type of a universe with the type of equivalences, so that distinct paths between isomorphic objects keep their individuality up to higher coherence. A space in which all such paths have been contracted to a single point corresponds, in the vocabulary of this work, to Universe~H -- the extensional environment where isomorphic number fields are identified without qualification. Standard arithmetic practice thus appears as the adoption of a particular extensional discipline, a choice to work in a truncated space where the paths between isomorphic points have been rendered contractible. The refusal of this contraction produces a richer ambient in which the logical distance between two isomorphic copies may be witnessed by explicit transport data.
\end{remark}

From this vantage, a label attached to a copy of a number field, or the construction history of a configuration, serves as a concise trace of sense or of derivation rather than as a typographical decoration. Two copies $K_1$ and $K_2$ of a field $K$ that are isomorphic as rings need not be interchangeable in a context where their labels encode distinct paths of construction. The statement $K_1 \cong K_2$ then carries equality of algebraic form, while leaving open a logical space in which the senses of $K_1$ and $K_2$ differ.

\begin{remark}[An intensional cover of an extensional category]\label{rem:IntensionalCover}
The role of labels may be expressed by viewing them as part of an intensional cover of an extensional arithmetic category. Let $\mathcal{C}_{\mathrm{ext}}$ be a category such as that of number fields and ring homomorphisms, read in Universe~H. Consider a second category $\mathcal{C}_{\mathrm{int}}$ whose objects are pairs $(X,\lambda)$, where $X$ is an object of $\mathcal{C}_{\mathrm{ext}}$ and $\lambda$ is a label or construction history, and whose morphisms respect these labels. There is a forgetful functor
\[
  U \colon \mathcal{C}_{\mathrm{int}} \longrightarrow \mathcal{C}_{\mathrm{ext}}
\]
which erases labels and remembers only the underlying extensional object. Two labeled copies $(X,\lambda_1)$ and $(X,\lambda_2)$ share the same extensional image, yet they need not be identified in $\mathcal{C}_{\mathrm{int}}$ in the absence of a label-preserving isomorphism between them. The passage from a sentence about labeled objects to its image under $U$ may collapse distinctions that were significant above. The whole apparatus of the following sections refines this forgetful functor into a $2$-categorical and quantitative form.
\end{remark}

\subsection{Quantitative identity as controlled intensionality}

The quantitative categories introduced in Section~\ref{sec:QuantitativeCategories} formalize this nuanced stance. Within such a category, a morphism $f \colon X \to Y$ represents an isomorphism able to carry a non-zero load. The cost $c(f)$ registers a semantic distortion or informational effort that is invisible at the structural level of static identification. The height $H$ then exchanges strict invariance for subadditivity, satisfying
\[
  H(Y) \leq H(X) + c(f)
\]
for every morphism $f \colon X \to Y$. An isomorphism with $c(f) > 0$ thus articulates an equivalence that preserves algebraic structure while registering the dissipation of information about the construction history of the object. The cost $c(f)$ measures the decoding effort associated with the derivation, so that the history of formation and the position in a network of links acquire a numerical weight, comparable to the description length in the theory of Kolmogorov complexity~\cite{Kolmogorov1965}.

\subsection{Exigencies for a distance-based semantics}

The passage from a qualitative category to a quantitative one imposes new architectural requirements, more exacting than those it replaces. It is requisite, first, that the costs $c(f)$ be stable under composition and factorization, respecting the triangle inequality of a Lawvere metric space in the sense of Section~\ref{sec:QuantitativeCategories}, for otherwise the distances would become arbitrary decorations without structural import.

Simultaneously, this quantitative regime must recover the classical view at zero cost. When restricted to a subcategory of protocol-determined morphisms, such as identity maps or standard changes of coordinates, the costs must vanish, reproducing the invariance of heights and volumes. The theory must contain the standard signature as its zero-cost limit, a requirement made precise by the descent of heights in Proposition~\ref{prop:ConservativeExtension}.

For these distances to serve as substantive bounds in Diophantine geometry, the cost must finally be effective. The value $c(f)$ cannot remain abstract. It must be computable from the arithmetic data and must relate transparently to known invariants such as the discriminant and the conductor. The elliptic packages of Section~\ref{sec:EllipticPackages} are introduced precisely to meet this last exigency.

\section{Quantitative categories and Lawvere enrichment}\label{sec:QuantitativeCategories}

We now describe the structure that carries the weight of a transport. The reading through enriched category theory is placed first, since it furnishes the language in which every later cost will be measured.

\subsection{The Lawvere quantale}

\begin{definition}[The cost quantale]\label{def:CostQuantale}
Let $\mathcal{V}$ denote the structure
\[
  \mathcal{V} = \bigl([0,\infty],\ \ge,\ +,\ 0\bigr)
\]
where $[0,\infty]$ is the extended non-negative reals, the order is the \emph{reverse} of the natural order (a morphism $a \to b$ exists exactly when $a \ge b$, so that $0$, the costless identification, is the terminal object), the tensor is addition with the convention $a + \infty = \infty$, and the unit is $0$. As a poset $[0,\infty]$ is complete, addition preserves the order in each variable and distributes over arbitrary infima, and $0$ is the greatest element for $\ge$. Thus $\mathcal{V}$ is a small complete monoidal poset -- a commutative quantale in the sense suited to enrichment.
\end{definition}

The reversal of the order is the point on which the whole reading turns. A smaller cost is a stronger relation, the unit $0$ is the cost of doing nothing, and the tensor adds costs along a composite. The triangle inequality and the vanishing of self-distance, which one would otherwise impose by hand, become the two axioms of an enriched category.

\begin{definition}[$\mathcal{V}$-category]\label{def:VCategory}
A \emph{$\mathcal{V}$-category} $\mathsf{X}$ consists of a set $\mathrm{Ob}(\mathsf{X})$ together with, for every ordered pair of objects $X,Y$, an element
\[
  \mathsf{X}(X,Y) \in [0,\infty]
\]
subject to the two requirements
\[
  \mathsf{X}(X,X) = 0
  \qquad\text{and}\qquad
  \mathsf{X}(X,Z) \le \mathsf{X}(X,Y) + \mathsf{X}(Y,Z)
\]
for all objects $X,Y,Z$. A \emph{$\mathcal{V}$-functor} $F \colon \mathsf{X} \to \mathsf{Y}$ is a map of objects with $\mathsf{Y}(FX,FY) \le \mathsf{X}(X,Y)$ for all $X,Y$.
\end{definition}

Following Lawvere~\cite{Lawvere1973}, a $\mathcal{V}$-category is exactly a generalized metric space, with the asymmetry of $\mathsf{X}(X,Y)$ and $\mathsf{X}(Y,X)$ permitted, with infinite distances permitted, and with a $\mathcal{V}$-functor playing the role of a short, that is, distance-decreasing, map. The framework that follows produces, for each fixed extensional object, a $\mathcal{V}$-category whose distances are the least costs of transports, and the height is a $\mathcal{V}$-functor into the real line read as a generalized metric space.

\subsection{Quantitative categories of isomorphisms}

\begin{definition}[Quantitative category of isomorphisms]\label{def:QuantitativeCategory}
A \emph{quantitative category of isomorphisms} is a category $\mathcal{Q}$ endowed with the following data. Invertible and directed morphisms may reside side by side, and the cost applies to all arrows.

\begin{itemize}
\item[(i)] Each morphism $f \colon X \to Y$ carries a non-negative real number
  \[
    c(f) \in \R_{\ge 0}
  \]
termed its \emph{distortion cost}. The cost attaches to the morphism rather than to the endpoints, so that parallel morphisms $f,g \colon X \to Y$ may carry distinct costs. One requires
  \[
    c(\mathrm{id}_X) = 0
    \qquad\text{and}\qquad
    c(g \circ f) \le c(f) + c(g)
  \]
for every object $X$ and every composable pair.

\item[(ii)] A morphism $f \colon X \to Y$ is an \emph{isomorphism} of $\mathcal{Q}$ when it admits a two-sided inverse $g \colon Y \to X$ in the categorical sense. There is no requirement that $c(f)$ or $c(g)$ vanish. One may have $c(f) > 0$ and $c(g) > 0$ even for isomorphisms.

\item[(iii)] A height assignment
  \[
    H \colon \operatorname{Ob}(\mathcal{Q}) \longrightarrow \R
  \]
playing the role of a height or a volume, satisfying the transport inequality
  \[
    H(Y) \le H(X) + c(f)
  \]
for every morphism $f \colon X \to Y$.
\end{itemize}

The assignment $H$ is a map on objects and nothing further. An ordinary functor into an ordered set would send every categorical isomorphism to a pair of equal values, and the positive variation that the framework wishes to measure would be lost at once. We therefore speak throughout of a height assignment, reserving the word functor for the categorical structure itself.
\end{definition}

\begin{remark}[Lawvere metric spaces and the enriched reading]\label{rem:LawvereMetric}
The axioms imposed on the cost $c$ -- namely $c(\mathrm{id}_X) = 0$ and the subadditivity $c(g \circ f) \le c(f) + c(g)$ -- reproduce the structure of a category enriched over the quantale $\mathcal{V}$ of Definition~\ref{def:CostQuantale}. The distance from $X$ to $Y$ is the infimum of costs over all morphisms between them, and the triangle inequality follows from the subadditivity of $c$ under composition. This generalized distance need not be symmetric, since the cost of transporting data from $X$ to $Y$ may differ from that of the reverse passage. The asymmetry is exactly what the framework requires, for the cost of a transport and that of its inverse are not constrained to coincide.

Universe~A thus views the category of arithmetic objects as a generalized metric space in which the distance between two configurations measures the effort, computational or geometric, of decoding one construction history into another. The height assignment of part~(iii) appears in this light as a map controlled by the enrichment, satisfying a short-map condition with respect to the Lawvere distance. This anchors the quantitative categories within an established categorical lineage~\cite{Lawvere1973, Kelly1982}, while the specifically arithmetic substance of the cost -- its dependence on discriminants, conductors, and Galois-theoretic data -- remains proper to the present inquiry.
\end{remark}

\begin{definition}[Transport distance]\label{def:TransportDistance}
Let $\mathcal{Q}$ be a quantitative category of isomorphisms. For objects $X,Y$ of $\mathcal{Q}$ set
\[
  d_{\mathcal{Q}}(X,Y) \coloneqq \inf\bigl\{ c(f) : f \in \mathcal{Q}(X,Y) \bigr\}
\]
with the convention $\inf\varnothing = +\infty$. The assignment $d_{\mathcal{Q}}$ endows $\mathrm{Ob}(\mathcal{Q})$ with the structure of a $\mathcal{V}$-category in the sense of Definition~\ref{def:VCategory}, termed the \emph{transport distance} of $\mathcal{Q}$.
\end{definition}

\begin{proposition}[The height is short for the transport distance]
\label{prop:HeightShort}
Let $\mathcal{Q}$ be a quantitative category of isomorphisms with height assignment $H$ in the sense of Definition~\ref{def:QuantitativeCategory}(iii), and equip $\R$ with the $\mathcal{V}$-category structure $\R(a,b) = \max\{b - a, 0\}$. Then $H$ is a $\mathcal{V}$-functor
\[
  \bigl(\mathrm{Ob}(\mathcal{Q}), d_{\mathcal{Q}}\bigr)
  \longrightarrow
  (\R, \R(\,\cdot\,,\,\cdot\,))
\]
that is to say, $\max\{H(Y) - H(X),\, 0\} \le d_{\mathcal{Q}}(X,Y)$ for all objects $X,Y$.
\end{proposition}

\begin{proof}
Fix $X,Y$. For every morphism $f \colon X \to Y$ the transport inequality of Definition~\ref{def:QuantitativeCategory}(iii) gives $H(Y) - H(X) \le c(f)$, hence $\max\{H(Y) - H(X),0\} \le c(f)$ since the left side is non-negative. Taking the infimum over all such $f$ yields $\max\{H(Y) - H(X),0\} \le d_{\mathcal{Q}}(X,Y)$. If there is no morphism $X \to Y$, then $d_{\mathcal{Q}}(X,Y) = +\infty$ and the inequality holds trivially.
\end{proof}

Proposition~\ref{prop:HeightShort} is the smallest statement that expresses the intended reading. A height is a short map for a generalized distance rather than an invariant of isomorphism classes, and that distance is the cost of transport. In the extensional limit, where every transport between isomorphic objects is free, the distance between isomorphic objects is zero, and the short-map condition recovers the extensionality principle $X \cong Y \implies H(X) = H(Y)$.

\begin{remark}[Where the cost is arithmetic]\label{rem:CostArithmetic}
The quantale $\mathcal{V}$ and the enriched reading carry no arithmetic of their own. They organize the account. The arithmetic enters when one populates $\mathcal{Q}$ with copies of number fields and elliptic curves carrying Galois and model data, at which point the identity of an object is determined chiefly by its position within a web of correlations rather than by the extensional membership of its underlying set. Two objects may be isomorphic as bare rings while occupying distinct positions in this web, and a morphism then realigns the positions at a cost that the later sections compute explicitly.
\end{remark}

\section{The labeled quantitative \texorpdfstring{$(2,1)$}{(2,1)}-category}\label{sec:UniverseAFormal}

A quantitative category of isomorphisms assigns a cost to transports. It does not yet express the fact that two transports may agree only up to a coherent change of presentation. To keep this last distinction without collapsing it into equality, one passes to a $(2,1)$-category, where invertible $2$-morphisms witness such agreements while the scalar cost remains attached to the $1$-morphisms.

\begin{definition}[$(2,1)$-category]\label{def:TwoOneCategory}
A \emph{$(2,1)$-category} is a $2$-category in which every $2$-morphism is invertible. Equivalently, for any two objects $A$ and $B$, the hom-category $\mathfrak{A}(A,B)$ is a groupoid.
\end{definition}

\begin{remark}[A point on strictness]\label{rem:strictness}
In what follows the word ``$2$-category'' is used in a pragmatic sense. Should one prefer a weak $2$-category, that is a bicategory, then the same discussion may be reproduced by replacing $2$-functors with pseudofunctors and inserting the requisite coherence isomorphisms. Since the present inquiry treats the arithmetic role of transport rather than a full higher-categorical development, the strict language is retained.

Strictness in this sense concerns the associativity and unit laws of composition, together with the on-the-nose preservation of composition by $2$-functors, rather than the presence of $2$-morphisms. The invertible $2$-cells between parallel $1$-morphisms, upon which the invariance of the cost and the pseudonaturality of Definition~\ref{def:PseudoNatural} rest, are retained in full. The point at which the strict language reaches its limit, namely the transport of geometric data in place of scalars, is taken up in Remark~\ref{rem:StrictPrototypeLimit}.

This choice has no bearing on the arithmetic descent of Section~\ref{sec:UniverseP}. A unitor or associator cell, strict or merely coherent, acts as the identity upon the scalar data -- the costs and heights -- that Universe~P ever reads, since a transport is registered there by its cost alone and not by the $1$-morphism or the coherence cell that produced it. The finite integer tuples of a code are built from these scalars, so no coherence cell of either the strict or the weak reading can alter an entry of such a code. Where the descent does pass through a coherence cell, as it does for the intensional Szpiro inequality of Section~\ref{sec:IntensionalSzpiro}, the point is verified in detail for the toy Coq interface in Remark~\ref{rem:StrictPrototypeLimit}, and the same argument governs every later descent to a finite integer code.
\end{remark}

\begin{definition}[Homotopy $1$-category]\label{def:HoOneCategory}
Let $\mathfrak{A}$ be a $(2,1)$-category. The \emph{homotopy $1$-category} $\Ho(\mathfrak{A})$ is defined as follows. Its objects coincide with the objects of $\mathfrak{A}$. For objects $A$ and $B$, the morphisms in $\Ho(\mathfrak{A})(A,B)$ are the $2$-isomorphism classes of $1$-morphisms $A \to B$. Composition is induced by the composition of $1$-morphisms, which remains well-defined by virtue of the invertibility of all $2$-morphisms. There is a canonical functor
\[
  \pi \colon \mathfrak{A} \longrightarrow \Ho(\mathfrak{A})
\]
which acts as the identity on objects and sends a $1$-morphism to its $2$-isomorphism class.
\end{definition}

The passage $\mathfrak{A} \mapsto \Ho(\mathfrak{A})$ renders visible a distinction often lived informally. One might present a transport alongside its auxiliary choices, and two such presentations may prove equivalent without being equal. Working within $\Ho(\mathfrak{A})$ allows one to speak of transports only up to such equivalence.

\begin{definition}[Labeled quantitative $(2,1)$-category]\label{def:UniverseAFormal}
A \emph{Universe~A structure} consists of the following data:

\begin{enumerate}
\item[(i)] A $(2,1)$-category $\mathfrak{A}$ in the sense of Definition~\ref{def:TwoOneCategory}, whose
  \begin{itemize}
\item objects are triples $(X,\lambda,\mathcal{H})$, where $X$ is an underlying arithmetic object in the sense of Universe~H, such as a number field, a scheme, or an elliptic curve, where $\lambda$ is a label or construction index, and where $\mathcal{H}$ is a package of auxiliary structures such as Galois actions, model data, or other decorations
\item $1$-morphisms $f \colon (X,\lambda,\mathcal{H}) \to (Y,\lambda',\mathcal{H}')$ are structure-preserving transports, in the manner of the examples of Sections~\ref{sec:Anabelian} and~\ref{sec:EllipticPackages}
\item $2$-morphisms are invertible identifications between $1$-morphisms, encoding that two transports may differ by auxiliary choices while keeping a formal distinctness.
  \end{itemize}

\item[(ii)] A forgetful $2$-functor
  \[
    U \colon \mathfrak{A} \longrightarrow \mathcal{C}_{\mathrm{ext}}
  \]
where $\mathcal{C}_{\mathrm{ext}}$ is the extensional category of arithmetic objects of Universe~H, such as number fields with ring homomorphisms or elliptic curves with isogenies, viewed as a $2$-category with only identity $2$-morphisms. On objects, $U(X,\lambda,\mathcal{H}) = X$. On $1$-morphisms, $U$ forgets labels and auxiliary structures and retains the underlying extensional morphism. The fibers of $U$ over a fixed $X$ are the family of intensional realizations of $X$ with their labels and histories.

\item[(iii)] A quantitative structure on $1$-morphisms, given by a map
  \[
    c \colon \operatorname{Mor}_1(\mathfrak{A}) \longrightarrow \R_{\ge 0}
  \]
termed the \emph{distortion cost}, satisfying $c(\mathrm{id}_A) = 0$ and $c(g \circ f) \le c(f) + c(g)$ for every object $A$ and every composable pair. The cost is an invariant of $2$-isomorphism. If there exists an invertible $2$-morphism $\alpha \colon f \Rightarrow g$ between parallel $1$-morphisms, then $c(f) = c(g)$. Equivalently, $c$ factors through the homotopy $1$-category $\Ho(\mathfrak{A})$ of Definition~\ref{def:HoOneCategory}.

\item[(iv)] A height assignment $H \colon \operatorname{Ob}(\mathfrak{A}) \to \R$ such that for every $1$-morphism $f \colon A \to B$ the transport inequality
  \[
    H(B) \le H(A) + c(f)
  \]
holds. Since $H$ is defined on objects and $c$ respects $2$-isomorphisms, the inequality depends only on the $2$-isomorphism class of the chosen transport.

\item[(v)] A contextual structure in the sense of C-systems~\cite{voevodsky2023martin}, realized as a fibration over a base of types. For each object $A = (X,\lambda,\mathcal{H})$ there is a small category $\mathrm{Ctx}(A)$ of \emph{contexts above $A$}, whose objects are refinements or extensions of the auxiliary data $\mathcal{H}$. A $1$-morphism $f \colon A \to B$ induces a pullback functor
  \[
    f^\ast \colon \mathrm{Ctx}(B) \longrightarrow \mathrm{Ctx}(A)
  \]
compatible with composition up to coherent invertible $2$-morphisms. This structure describes how the identity rules, and in particular J-elimination, behave when one transports auxiliary data along a morphism of $\mathfrak{A}$. The pullback functors carry no cost of their own. The whole weight of a transport is borne once, by the scalar $c(f)$ of item (iii), and the substitution $f^\ast$ is free, in the sense that no quantity $c(f^\ast)$ is levied over and above $c(f)$.
\end{enumerate}
A Universe~A structure in this sense is termed a \emph{labeled quantitative $(2,1)$-category with contextual flavor}.
\end{definition}

\begin{remark}[Forgetting labels through $\Ho(\mathfrak{A})$]\label{rem:ForgetThroughHo}
Since $\mathcal{C}_{\mathrm{ext}}$ carries only identity $2$-morphisms, the forgetful $2$-functor $U$ factors canonically through $\Ho(\mathfrak{A})$, yielding a $1$-functor
\[
  \Ho(U) \colon \Ho(\mathfrak{A}) \longrightarrow \mathcal{C}_{\mathrm{ext}}
\]
which discards labels and auxiliary structures once transports have been reduced to their $2$-isomorphism classes. Within this language, the inequalities governed by $c$ and $H$ read as statements attached to morphisms in $\Ho(\mathfrak{A})$.
\end{remark}

\begin{definition}[Vertical groupoid over an extensional object]\label{def:VerticalGroupoid}
Let $(\mathfrak{A},U)$ be a Universe~A structure and let $X$ be an object of $\mathcal{C}_{\mathrm{ext}}$. The \emph{vertical groupoid over $X$}, denoted $\mathfrak{A}_X^{\mathrm{vert}}$, has as objects those $A \in \operatorname{Ob}(\mathfrak{A})$ with $U(A) = X$, and as morphisms from $A$ to $B$ the $2$-isomorphism classes of $1$-morphisms $f \colon A \to B$ such that $U(f) = \mathrm{id}_X$ and $f$ is an equivalence in $\mathfrak{A}$. Composition is induced from $\mathfrak{A}$. Because equivalences are stable under composition and admit quasi-inverses, $\mathfrak{A}_X^{\mathrm{vert}}$ is a groupoid.
\end{definition}

\begin{remark}[The vertical fiber]\label{rem:locus_pure_intensionality}
The vertical groupoid $\mathfrak{A}_X^{\mathrm{vert}}$ retains the presentations lying above one extensional object, together with the transports between them. Its quantitative structure may assign a strictly positive distance to some of these pairs, reflecting the effort of realigning labels or auxiliary data above a fixed object.

Vanishing costs do not imply that the groupoid is contractible. A groupoid all of whose costs vanish may still have non-trivial automorphisms, and it may have several connected components. Contractibility is a separate categorical condition, and nothing below requires it.
\end{remark}

\begin{definition}[Normalization datum]\label{def:UniverseANormalization}
Let $(\mathfrak{A},U,c,H)$ be a Universe~A structure. A \emph{normalization datum} on $\mathfrak{A}$ comprises, for every object $X$ of $\mathcal{C}_{\mathrm{ext}}$, a representative $\mathrm{NF}(X) \in \operatorname{Ob}(\mathfrak{A})$ with $U(\mathrm{NF}(X)) = X$, and for every object $A$, a distinguished $1$-morphism
\[
  \nu_A \colon \mathrm{NF}(U(A)) \longrightarrow A
\]
whose image under $U$ is $\mathrm{id}_{U(A)}$. One defines the \emph{normalization defect} of $A$ as $\delta_{\mathfrak{A}}(A) \coloneqq c(\nu_A)$.
\end{definition}

\begin{remark}[On reference points]\label{rem:reference_points}
Such a datum demands neither the uniqueness of $\mathrm{NF}(X)$ nor that $\nu_A$ be of minimal cost. It serves to fix a disciplined selection of a reference point within each extensional fiber, together with a distinguished path required to reach a given object. This mirrors the normalization procedures invoked for elliptic packages in Definition~\ref{def:AccumulatedDefect}, extending that logic to the ambient categorical environment.
\end{remark}

\section{The J-structure and controlled elimination}\label{sec:Jstructure}

The vertical groupoid of Definition~\ref{def:VerticalGroupoid} admits a reading through the optics of C-systems, and this reading ties the arithmetic cost to the identity type of a dependent theory.

\begin{remark}[Vertical groupoids and identity types]\label{rem:VerticalGroupoidCsystems}
In a C-system~\cite{voevodsky2023martin}, the identity type $\mathrm{Id}_A(a,a)$ over a fixed term $a$ of type $A$ is itself a type whose inhabitants witness paths from $a$ to $a$. When the C-system is non-trivial in the homotopical sense, this type may contain inhabitants other than reflexivity, each carrying its own computational or geometric content. In the present framework, the role of $\mathrm{Id}_A(a,a)$ is played by the vertical groupoid $\mathfrak{A}_X^{\mathrm{vert}}$. A morphism $f \colon A \to B$ in this groupoid, with $U(f) = \mathrm{id}_X$, represents a non-trivial self-identification of the extensional object $X$ -- a path that permutes labels or auxiliary data while leaving the underlying arithmetic structure fixed. In Universe~H, where the J-structure has been collapsed and all self-identifications are forced to be reflexivity, every vertical morphism carries zero cost. In Universe~A it retains the richer structure of a groupoid whose morphisms may carry positive distortion cost.

This parallel reinforces the thesis that the intensional environment of Universe~A arises of its own accord, wherever one declines to contract the identity type to a mere proposition. Such a refusal has received independent mathematical life in homotopy type theory and its $\infty$-topos semantics~\cite{Riehl2024}.
\end{remark}

\begin{definition}[Lawvere--Voevodsky interpretation of the vertical cost]
\label{def:LawvereVoevodskyInterpretation}
Let $(\mathfrak{A},U,c,H)$ be a Universe~A structure and let $X$ be an object of $\mathcal{C}_{\mathrm{ext}}$. The \emph{Lawvere--Voevodsky interpretation} of the vertical groupoid $\mathfrak{A}_X^{\mathrm{vert}}$ comprises the following data:

\begin{enumerate}
\item[(i)] \textit{Vertical distance.} For objects $A,B \in \operatorname{Ob}(\mathfrak{A}_X^{\mathrm{vert}})$ set
\[
  d_X^{\mathrm{vert}}(A,B)
  \coloneqq
  \inf\bigl\{ c(f) : f \in \mathfrak{A}_X^{\mathrm{vert}}(A,B) \bigr\}
\]
with $\inf\varnothing = +\infty$. By the cost axioms of Definition~\ref{def:UniverseAFormal}(iii), the assignment $d_X^{\mathrm{vert}}$ satisfies $d_X^{\mathrm{vert}}(A,A) = 0$ and the triangle inequality, and so endows $\operatorname{Ob}(\mathfrak{A}_X^{\mathrm{vert}})$ with the structure of a Lawvere generalized metric space, a $\mathcal{V}$-category in the sense of Definition~\ref{def:VCategory}.

\item[(ii)] \textit{Path-space reading.} Under the correspondence
\[
  \mathrm{Id}_{\mathcal{C}_{\mathrm{ext}}}(X,X)
  \;\longleftrightarrow\;
  \mathfrak{A}_X^{\mathrm{vert}}
\]
an inhabitant of the identity type corresponds to a vertical morphism $f \colon A \to B$ with $U(f) = \mathrm{id}_X$, and the cost $c(f)$ measures the length of that path in the Lawvere metric.

\item[(iii)] \textit{Reflexivity and zero cost.} The reflexivity term $\mathrm{refl}_A$ corresponds to $\mathrm{id}_A$ in $\mathfrak{A}_X^{\mathrm{vert}}$, and the cost axiom $c(\mathrm{id}_A) = 0$ expresses that reflexivity carries zero length.
\end{enumerate}
\end{definition}

\begin{remark}[Scope of the identity-type reading]\label{rem:IdentityTypeReadingScope}
The correspondence with identity types is an interpretation of the transport vocabulary. A model of dependent identity types calls for context extension, substitution, identity-type formation, reflexivity, elimination, and the computation laws that accompany them. The cost axioms alone supply none of these further operations.

The quantitative estimates below concern heights on vertical groupoids. Their use requires no assertion that the groupoid together with its cost constitutes a model of J-elimination.
\end{remark}

\begin{proposition}[Controlled J-elimination from the transport inequality]
\label{prop:ControlledJElimination}
Let $(\mathfrak{A},U,c,H)$ be a Universe~A structure and let $d_X^{\mathrm{vert}}$ be the vertical distance of Definition~\ref{def:LawvereVoevodskyInterpretation}.

\begin{enumerate}
\item[\textup{(i)}] \textit{Quantitative J-rule.} For every vertical morphism $f \colon A \to B$ in $\mathfrak{A}_X^{\mathrm{vert}}$ and every height $H$ satisfying the transport inequality of Definition~\ref{def:UniverseAFormal}(iv), one has
\[
  |H(B) - H(A)| \le c(f) + c(f^{-1})
\]
where $f^{-1}$ is the inverse in the groupoid. If $c(f) = c(f^{-1}) = 0$, then $H(B) = H(A)$, recovering the classical J-elimination principle that transport along reflexivity preserves all structure.

\item[\textup{(ii)}] \textit{Metric control on the height.} The height $H$, restricted to $\operatorname{Ob}(\mathfrak{A}_X^{\mathrm{vert}})$, is Lipschitz with respect to the symmetrized vertical distance
\[
  \bar{d}_X^{\mathrm{vert}}(A,B)
  \coloneqq
  \inf_{f \in \mathfrak{A}_X^{\mathrm{vert}}(A,B)}
  \bigl( c(f) + c(f^{-1}) \bigr)
\]
in the precise sense that $|H(B) - H(A)| \le \bar{d}_X^{\mathrm{vert}}(A,B)$ for all objects $A,B$.

\item[\textup{(iii)}] \textit{Zero-cost specialization.} If every pair of objects in a vertical fiber is joined by a vertical equivalence of zero cost in both directions, then the height is constant on that fiber.
\end{enumerate}
\end{proposition}

\begin{proof}
For (i), apply the transport inequality to $f$ and to $f^{-1}$ separately, then combine. For (ii), the infimum over all $f \in \mathfrak{A}_X^{\mathrm{vert}}(A,B)$ of $c(f) + c(f^{-1})$ bounds $|H(B) - H(A)|$ by (i). For (iii), the symmetrized distance vanishes on the fiber by hypothesis, so (ii) gives the equality of the heights.
\end{proof}

\begin{remark}[Quantitative categories with families]\label{rem:QuantitativeCwF}
The structure of Definition~\ref{def:LawvereVoevodskyInterpretation} parallels the quantitative categories with families introduced by Atkey~\cite{Atkey2018} as the categorical semantics of quantitative type theory. There the standard categories-with-families semantics for Martin-L\"of type theory is enriched by a semiring of usage annotations, and substitution is required to respect these annotations. In the present context the role of the semiring is played by the Lawvere quantale $([0,\infty],\ge,+)$, and the usage of a morphism is its distortion cost $c(f)$. The J-elimination rule of Proposition~\ref{prop:ControlledJElimination}(i) reads as the statement that transport along a zero-usage path preserves height exactly, while transport along a path of positive usage incurs a bounded penalty. This is the sense in which Universe~A implements a resource-sensitive discipline of identity, where the effort of decoding one intensional presentation into another is held as a quantitative datum rather than erased by a truncation.

A word is owed on substitution, since Definition~\ref{def:UniverseAFormal}(v) introduces the pullback functors $f^\ast \colon \mathrm{Ctx}(B) \to \mathrm{Ctx}(A)$ without a cost of their own. The choice made there, that context substitution is free, is the choice of Atkey's semantics as well. In a quantitative category with families, substitution respects the usage annotations and re-indexes them along the morphism, and it consumes no resource in doing so, the whole expenditure being carried by the terms and, in the present reading, by the $1$-morphism $f$. The cost $c$ is accordingly defined upon $\operatorname{Mor}_1(\mathfrak{A})$ alone, and a context of the form $f^\ast\Gamma$ is charged only through $c(f)$ wherever it enters a transport inequality. One could entertain a finer discipline in which a quantity $c(f^\ast)$ were levied upon the substitution itself, subadditive along the coherence isomorphisms $(g \circ f)^\ast \cong f^\ast \circ g^\ast$ and bounded above by $c(f)$, and such a discipline would sit closer still to the annotated substitutions of~\cite{Atkey2018}. Nothing below requires that refinement, and every construction of this paper reads the weight of a transport once, at the level of the $1$-morphism.

The $\infty$-topos semantics of Riehl~\cite{Riehl2024} affords the geometric realization of this picture. In an $\infty$-topos the identity type of a universe is the type of equivalences, and distinct paths between equivalent objects keep their individuality up to higher coherence. Universe~H corresponds to the $0$-truncation of this space, where all paths have been contracted. Universe~A corresponds to an intermediate regime in which paths are retained and measured, yet the measurement obeys the compositional discipline of a Lawvere metric.
\end{remark}

\subsection{The zero-cost quotient and extensional descent}

The relation between the two universes can now be stated with precision. Two quotients present themselves, and they do not coincide. The first contracts the pairs joined at zero cost, and the second contracts the fibers of the forgetful map.

\begin{definition}[Zero-cost quotient and extensional collapse]\label{def:ExtensionalCollapse}
Let $(\mathfrak{A},U,c,H)$ be a Universe~A structure. Define a relation on its objects by
\[
  A\sim_0 B
  \quad\Longleftrightarrow\quad
  U(A)=U(B)
  \ \text{ and }\
  \bar d_{U(A)}^{\mathrm{vert}}(A,B)=0.
\]
The triangle inequality renders this an equivalence relation, and its quotient is termed the \emph{zero-cost quotient}.

The forgetful map
\[
  \kappa=\Ho(U)\colon\Ho(\mathfrak{A})\longrightarrow\mathcal{C}_{\mathrm{ext}}
\]
is termed the \emph{extensional collapse}. It may identify objects lying at positive vertical distance from one another.
\end{definition}

\begin{proposition}[Descent of heights]\label{prop:ConservativeExtension}
The height $H$ descends to the zero-cost quotient.

The assignment $H$ factors through an isomorphism-invariant height on the extensional objects represented by $U$ precisely when
\[
  U(A)\cong U(B)
  \quad\Longrightarrow\quad
  H(A)=H(B).
\]
An extensional height $h$ lifts to the assignment $h\circ U$. This lift satisfies the transport inequality on every arrow whose extensional image is an isomorphism. On the remaining arrows, the transport inequality calls for an additional bound.
\end{proposition}

\begin{proof}
The estimate
\[
  |H(A)-H(B)|
  \le
  \bar d_{U(A)}^{\mathrm{vert}}(A,B)
\]
of Proposition~\ref{prop:ControlledJElimination}(ii) shows that $H$ is constant on every $\sim_0$-class.

The displayed condition is necessary for a factorization through an isomorphism-invariant extensional height. It is sufficient as well, since it renders the value independent of every representative of an extensional isomorphism class.

For the lift, an arrow whose extensional image is an isomorphism joins two equal values of $h$. Any non-negative cost assigned to that arrow therefore satisfies the required inequality.
\end{proof}

\begin{remark}[The two quotients]\label{rem:readingConservative}
The zero-cost quotient preserves every short height. Extensional descent asks for something stronger, namely a constancy on the larger classes selected by the forgetful map.

A lifted height does not determine the cost of an arrow. Equal values at the two endpoints permit a zero cost, while a positive cost previously assigned remains compatible with the transport inequality. The proposition concerns the descent of numerical assignments, and it makes no claim of proof-theoretic conservativity between two formal theories.
\end{remark}

\section{Comparison problems and the descent obstruction}\label{sec:Comparison}

We come to the categorical reinterpretation named in the title. A comparison problem, in the present reading, is the data of two intensional realizations of one extensional object together with a transport between them, and the question attached to it is whether the resulting transport inequality descends to a non-trivial inequality between extensional invariants. The answer is governed entirely by the vertical groupoid, and the obstruction to descent is exact.

\subsection{Comparison problems}

\begin{definition}[Comparison problem]\label{def:ComparisonProblem}
Let $(\mathfrak{A},U,c,H)$ be a Universe~A structure and let $X$ be an object of $\mathcal{C}_{\mathrm{ext}}$. A \emph{comparison problem over $X$} is a pair of objects $A_{\mathrm{src}}, A_{\mathrm{tgt}}$ of $\mathfrak{A}$ with $U(A_{\mathrm{src}}) = U(A_{\mathrm{tgt}}) = X$, together with a vertical transport
\[
  f \colon A_{\mathrm{src}} \longrightarrow A_{\mathrm{tgt}}
\]
in $\mathfrak{A}_X^{\mathrm{vert}}$. Its \emph{comparison inequality} is the transport inequality
\[
  H(A_{\mathrm{tgt}}) \le H(A_{\mathrm{src}}) + c(f)
\]
and its \emph{content} is the quantity $c(f) \in \R_{\ge 0}$, the vertical cost of the transport.
\end{definition}

\begin{remark}[Why this captures the informal notion]\label{rem:WhyComparison}
In arithmetic practice a comparison statement places two realizations of one object side by side -- two models of an elliptic curve, two reconstructions of a field from local data, two presentations of a fundamental group -- and asserts an inequality between invariants attached to them. The two realizations share an underlying extensional object $X$, and the transport $f$ is the explicit passage from one realization to the other. The cost of that passage records the change of presentation, and the transport inequality registers it whatever its value. Whether the comparison says anything about $X$ itself is a separate question. If both sides are evaluations of one extensional invariant, then they coincide, and the inequality reduces to $H(X) \le H(X) + c(f)$, whether the cost is zero or positive. The next two results make this precise, and Remark~\ref{rem:MoralObstruction} indicates where arithmetic content can survive.
\end{remark}

\subsection{A two-sided descent lemma}

We first describe, in a form requiring no special geometry, the mechanism by which a comparison collapses when one insists on reading it in the extensional signature. The lemma is internal to Universe~H.

\begin{definition}[Toy labeled and standard data]\label{def:ToyAliasedStandard}
Let $X$ be an object of $\mathcal{C}_{\mathrm{ext}}$. A \emph{toy labeled configuration over $X$} is a tuple
\[
  \mathcal{C} = \bigl(A_{\mathrm{src}}, A_{\mathrm{tgt}}, f, V^{\mathrm{src}}, V^{\mathrm{tgt}}, \delta\bigr)
\]
consisting of two objects $A_{\mathrm{src}}, A_{\mathrm{tgt}}$ above $X$, a vertical transport $f \colon A_{\mathrm{src}} \to A_{\mathrm{tgt}}$, two real numbers $V^{\mathrm{src}}, V^{\mathrm{tgt}}$ playing the role of intensional volumes attached to the two avatars, and a non-negative real number $\delta$ with
\[
  V^{\mathrm{tgt}} \le V^{\mathrm{src}} + \delta
\]
There is a natural involution $\sigma$ on such data that exchanges source and target, sending $\mathcal{C}$ to
\[
  \sigma(\mathcal{C}) = \bigl(A_{\mathrm{tgt}}, A_{\mathrm{src}}, f^{-1}, V^{\mathrm{tgt}}, V^{\mathrm{src}}, \delta'\bigr)
\]
where $\delta'$ is the non-negative real number appearing in the reverse inequality $V^{\mathrm{src}} \le V^{\mathrm{tgt}} + \delta'$.
\end{definition}

\begin{lemma}[Two-sided descent lemma]\label{lemma:ToyNoBridge}
Suppose there is a rule
\[
  \mathcal{B}^{\mathrm{toy}}
  \colon
  \{\text{toy labeled configurations over } X\}
  \longrightarrow
  \R^2 \times \R_{\ge 0},
  \qquad
  \mathcal{C} \longmapsto
  \bigl(H^{\mathrm{std}}_{\mathrm{src}}(\mathcal{C}), H^{\mathrm{std}}_{\mathrm{tgt}}(\mathcal{C}), \Delta(\mathcal{C})\bigr)
\]
to be read as a standard inequality in Universe~H,
\begin{equation}\label{eq:ToyStdIneq}
  H^{\mathrm{std}}_{\mathrm{tgt}}(\mathcal{C})
  \le
  H^{\mathrm{std}}_{\mathrm{src}}(\mathcal{C}) + \Delta(\mathcal{C}).
\end{equation}
Assume the following:
\begin{itemize}
\item[(i)] \emph{Equivariance under exchange.} If $\mathcal{C}$ and $\mathcal{C}'$ differ only by replacing the objects above $X$ by isomorphic copies and by exchanging the roles of source and target, then $\mathcal{B}^{\mathrm{toy}}(\mathcal{C})$ and $\mathcal{B}^{\mathrm{toy}}(\mathcal{C}')$ agree up to the same exchange. In particular $\Delta(\mathcal{C}) = \Delta(\sigma(\mathcal{C}))$ and $H^{\mathrm{std}}_{\mathrm{src}}(\mathcal{C}) = H^{\mathrm{std}}_{\mathrm{tgt}}(\sigma(\mathcal{C}))$.

\item[(ii)] \emph{Label transparency.} The quantities $H^{\mathrm{std}}_{\mathrm{src}}(\mathcal{C})$ and $H^{\mathrm{std}}_{\mathrm{tgt}}(\mathcal{C})$ are values of a single ring-coherent invariant in the sense of Axiom~\ref{ax:RingCoherence}, evaluated on the underlying object $X$ after static identification. Once the labels have been erased, there is no remaining standard datum that distinguishes the source evaluation from the target evaluation.

\item[(iii)] \emph{Intended survival of the discrepancy.} There is at least one configuration $\mathcal{C}$ with $V^{\mathrm{tgt}} \ne V^{\mathrm{src}}$ for which~\eqref{eq:ToyStdIneq} is intended to keep a non-trivial trace of the discrepancy, so that the rule does not force $H^{\mathrm{std}}_{\mathrm{src}}(\mathcal{C}) = H^{\mathrm{std}}_{\mathrm{tgt}}(\mathcal{C})$ by stipulation.
\end{itemize}
Then no such rule exists. Equivalently, under (i) and (ii) any rule producing a standard inequality of the form~\eqref{eq:ToyStdIneq} forces $H^{\mathrm{std}}_{\mathrm{src}}(\mathcal{C}) = H^{\mathrm{std}}_{\mathrm{tgt}}(\mathcal{C})$, and the comparison collapses in Universe~H to a harmless inequality.
\end{lemma}

\begin{proof}
Fix a configuration $\mathcal{C}$. By equivariance (i), the output triples of $\mathcal{C}$ and $\sigma(\mathcal{C})$ agree up to exchange, so that
\[
  H^{\mathrm{std}}_{\mathrm{src}}(\mathcal{C}) = H^{\mathrm{std}}_{\mathrm{tgt}}(\sigma(\mathcal{C}))
  \qquad\text{and}\qquad
  H^{\mathrm{std}}_{\mathrm{tgt}}(\mathcal{C}) = H^{\mathrm{std}}_{\mathrm{src}}(\sigma(\mathcal{C})).
\]
By label transparency (ii), once one passes to Universe~H and performs static identification, there is no standard datum distinguishing the source evaluation from the target evaluation, so the two outputs coincide,
\[
  H^{\mathrm{std}}_{\mathrm{src}}(\mathcal{C}) = H^{\mathrm{std}}_{\mathrm{tgt}}(\mathcal{C}).
\]
Then~\eqref{eq:ToyStdIneq} reduces to $H \le H + \Delta(\mathcal{C})$ with $\Delta(\mathcal{C}) \ge 0$, which holds independently of any discrepancy between $V^{\mathrm{src}}$ and $V^{\mathrm{tgt}}$. This contradicts the intended survival condition (iii).
\end{proof}

\subsection{The general descent obstruction}

The two-sided lemma is the elementary form of a single statement about the framework. A comparison inequality descends to the extensional quotient as a non-trivial inequality only when the transport is not a vertical equivalence of zero symmetrized cost. The proof is immediate from Proposition~\ref{prop:ControlledJElimination}, and the transparency of the statement is the principal benefit of the categorical reading.

\begin{theorem}[Descent of a vertical comparison]\label{thm:DescentObstruction}
Let $(\mathfrak{A},U,c,H)$ be a Universe~A structure and let $(A_{\mathrm{src}}, A_{\mathrm{tgt}}, f)$ be a comparison problem over an object $X$. Suppose that the height factors as
\[
  H=h\circ U.
\]
Then
\[
  H(A_{\mathrm{src}})=H(A_{\mathrm{tgt}})=h(X)
\]
and the comparison inequality becomes
\[
  h(X)\le h(X)+c(f).
\]
This conclusion holds for every non-negative value of $c(f)$.

Without the factorization assumption, a vanishing symmetrized vertical distance still implies the equality of the two heights. A positive distance alone implies no inequality sharper than the supplied transport bounds.
\end{theorem}

\begin{proof}
The two objects have the same image $X$ under $U$, so the factorization gives the equality of their heights. Substitution into the comparison inequality gives the displayed relation. The assertion at zero distance follows from the two-sided estimate of Proposition~\ref{prop:ControlledJElimination}(ii).
\end{proof}

\begin{remark}[The scope of the comparison]\label{rem:MoralObstruction}
A comparison between two evaluations of one extensional invariant on one object has equal endpoints, and nothing of the discrepancy survives the passage. A comparison between two different invariants on that same object may retain arithmetic content once the labels have been forgotten.

For the elliptic applications, the discriminant height and the conductor complexity are quantities of distinct provenance. Their comparison calls for arithmetic estimates of its own. The transport formalism controls the changes of presentation and preserves the explicit defect term, and the Diophantine sections below adopt this discipline throughout.
\end{remark}

\section{Anabelian packages: a quantitative category from practice}\label{sec:Anabelian}

Before turning to the Diophantine application, we exhibit quantitative categories of isomorphisms as something other than an artifice imposed from outside. They arise within established arithmetic practice, and anabelian geometry furnishes a transparent instance, where the bare profinite group of a curve carries less information than the group endowed with its Galois action and decomposition data, so that an isomorphism of bare groups incurs a cost from the start when read at the enriched level.

Let $k$ be a finitely generated field over $\Q$, and let $X$ be a smooth hyperbolic curve over $k$, with arithmetic fundamental group $\pi_1(X)$ sitting in the exact sequence
\[
  1
  \longrightarrow
  \pi_1(X_{\overline{k}})
  \longrightarrow
  \pi_1(X)
  \longrightarrow
  G_k
  \longrightarrow
  1
\]
where $G_k$ is the absolute Galois group of $k$ and $\pi_1(X_{\overline{k}})$ is the geometric fundamental group. This extension need not split, and no splitting is assumed below. The action of $G_k$ on $\pi_1(X_{\overline{k}})$ is regarded up to inner automorphism, through the outer representation
\[
  \rho_X \colon G_k \longrightarrow \mathrm{Out}\bigl(\pi_1(X_{\overline{k}})\bigr).
\]
Anabelian results, originating in Grothendieck's vision and developed by Mochizuki~\cite{mochizuki1999local, MochizukiAbAnIII}, show that under suitable hypotheses one reconstructs the isomorphism class of the pair $(X,k)$ from the profinite group $\pi_1(X_{\overline{k}})$ equipped with this outer action and with boundary decorations such as inertia and decomposition subgroups.

In Universe~H one tends to read $\pi_1(X_{\overline{k}})$ as a profinite group up to isomorphism. If two curves have isomorphic geometric fundamental groups as abstract profinite groups, then under this extensional outlook they cannot be separated by any invariant depending only on the group structure, and the outer action is extraneous unless explicitly encoded. Universe~A introduces a different viewpoint.

\begin{definition}[Anabelian package]\label{def:AnabelianPackage}
Let $k$ be a finitely generated field over $\Q$. An \emph{anabelian package over $k$} is a tuple $\mathcal{G} = (\Pi, G_k, \rho, \mathcal{D})$, where $\Pi$ is a profinite group, $\rho \colon G_k \to \mathrm{Out}(\Pi)$ is a continuous homomorphism encoding the outer Galois action, and $\mathcal{D} = \{D_v\}_{v \in S}$ is a finite collection of closed subgroups $D_v$ of a supplied extension $\Pi^{\mathrm{arith}}$ of $G_k$ by $\Pi$, indexed by a finite set $S$ of places or marked points, playing the role of decomposition or inertia groups. A \emph{morphism} $f \colon \mathcal{G} \to \mathcal{G}'$ over the same base field $k$ is an isomorphism of profinite groups $\varphi \colon \Pi \xrightarrow{\sim} \Pi'$ such that the induced map on outer automorphism groups intertwines $\rho$ and $\rho'$, and such that for each $v \in S$ the image $\varphi(D_v)$ is conjugate to some $D'_{v'} \in \mathcal{D}'$. The category of anabelian packages over $k$ is denoted $\mathbf{AnPkg}_k$.
\end{definition}

\begin{definition}[Transport defect]\label{def:TransportDefect}
Let $f \colon \mathcal{G} \to \mathcal{G}'$ be a morphism in $\mathbf{AnPkg}_k$. Its \emph{transport defect} is
\[
  \delta(f) \coloneqq \#\bigl\{ v \in S : \varphi(D_v) \ne D'_{v'} \text{ while } \varphi(D_v) \sim D'_{v'} \bigr\}
\]
where $\sim$ denotes conjugacy. That is, $\delta(f)$ counts the decomposition groups preserved only up to conjugacy rather than exactly. The identity morphism satisfies $\delta(\mathrm{id}_{\mathcal{G}}) = 0$, and for composable morphisms $\delta(g \circ f) \le \delta(f) + \delta(g)$.
\end{definition}

With the transport defect $\delta$ in the role of the distortion cost, the category $\mathbf{AnPkg}_k$ becomes a quantitative category of isomorphisms in the sense of Definition~\ref{def:QuantitativeCategory}. An isomorphism $f \colon \mathcal{G} \xrightarrow{\sim} \mathcal{G}'$ may have $\delta(f) > 0$ when the underlying group isomorphism fails to preserve the decomposition data exactly. The defect measures the positional information -- the placement of distinguished subgroups within the ambient group -- that is shifted when passing from $\mathcal{G}$ to $\mathcal{G}'$.

\begin{proposition}[A statement with stable content in Universe~A]\label{prop:OrdinaryInA}
Let $X$ and $Y$ be smooth hyperbolic curves over $k$ with anabelian packages $\mathcal{G}_X$ and $\mathcal{G}_Y$. Suppose there is an isomorphism of the bare profinite groups
\[
  \varphi \colon \pi_1(X_{\overline{k}}) \xrightarrow{\sim} \pi_1(Y_{\overline{k}})
\]
and suppose that no isomorphism of those groups intertwines the outer actions $\rho_X$ and $\rho_Y$ while preserving the prescribed decomposition data. Then the bare profinite groups are identified in Universe~H, becoming isomorphic once one keeps only the abstract group structure, while the enriched objects $\mathcal{G}_X$ and $\mathcal{G}_Y$ remain non-isomorphic in $\mathbf{AnPkg}_k$. In particular the statement
\begin{equation}\label{eq:AnabelianDistinction}
  \text{``}\,\mathcal{G}_X \not\cong \mathcal{G}_Y\,\text{''}
\end{equation}
has a determinate and true value in Universe~A, whereas the corresponding statement about the bare groups evaluates to false in Universe~H.
\end{proposition}

\begin{proof}
By hypothesis $\varphi$ is an isomorphism of abstract profinite groups, so in Universe~H, where objects are identified up to isomorphism of the underlying group, $X$ and $Y$ become indistinguishable at the level of their geometric fundamental groups. In $\mathbf{AnPkg}_k$ a morphism from $\mathcal{G}_X$ to $\mathcal{G}_Y$ is an isomorphism of the underlying groups intertwining $\rho_X$ and $\rho_Y$ and preserving the decomposition data up to conjugacy. By hypothesis no isomorphism of those groups has these properties, so there is no morphism, and the packages are non-isomorphic. Thus~\eqref{eq:AnabelianDistinction} is true in Universe~A and expresses the arithmetically meaningful fact that $X$ and $Y$, while sharing isomorphic geometric fundamental groups, differ in their Galois-theoretic structure.
\end{proof}

For Diophantine purposes one requires a real-valued invariant playing the role of a height. The following definition furnishes a schematic one.

\begin{definition}[Anabelian height]\label{def:AnabelianHeight}
Let $\mathcal{G} = (\Pi, G_k, \rho, \mathcal{D})$ arise from a smooth hyperbolic curve $X$ over $k$, with $\Pi = \pi_1(X_{\overline{k}})$ and $\mathcal{D}$ encoding decomposition groups at a finite set of rational points or points at infinity. The \emph{anabelian height} is
\[
  h_{\mathrm{ab}}(\mathcal{G}) \coloneqq h(X) + \sum_{v \in S} \log\, [G_k : \mathrm{Stab}_{G_k}(D_v)]
\]
where $h(X)$ is a classical height of the curve, for instance the Faltings height or a projective height, and the sum runs over the logarithmic indices of the stabilizers of the decomposition groups under the Galois action.
\end{definition}

\begin{hypothesis}[Quantitative anabelian height comparison]\label{hyp:TransportAnabelian}
Fix a precisely specified category of anabelian packages, a real-valued height $h_{\mathrm{ab}}$, and a non-negative transport cost $\delta$ that is invariant under the admitted changes of auxiliary data. There exists a constant $C_k$ such that every admitted transport $f \colon \mathcal{G} \to \mathcal{G}'$ satisfies
\[
  h_{\mathrm{ab}}(\mathcal{G}')
  \le
  h_{\mathrm{ab}}(\mathcal{G})
  +
  C_k\,\delta(f).
\]
\end{hypothesis}

\begin{remark}[The arithmetic comparison input]\label{rem:AnabelianFragment}
An anabelian package retains the Galois action and the decomposition data that a bare geometric fundamental group omits, and in this the category $\mathbf{AnPkg}_k$ affords a compact fragment of Universe~A. In Universe~H, where one works with the curve up to isomorphism, only $h(X)$ remains visible.

A quantitative comparison of heights calls for a further arithmetic theorem relating the chosen height to the chosen transport cost. The passage from an exact preservation of a decomposition group to a preservation up to conjugacy carries no uniform bound on the stabilizer indices, and Hypothesis~\ref{hyp:TransportAnabelian} names the input that would supply one. The finite Selmer-shadow arguments of Section~\ref{sec:ShaSelmerShadows} have no dependence upon it.
\end{remark}

%==================================================================
\section{Elliptic packages and the discriminant height}\label{sec:EllipticPackages}
%==================================================================

The preceding sections have stayed at the level of formalism. We now descend to arithmetic geometry and exhibit a quantitative category whose objects are presentations of elliptic curves and whose morphisms carry an explicit transport cost. The construction is elementary, and it rests upon the transformation law of the discriminant under a change of Weierstrass coordinates. The intensional reading contributes the decision to keep the presentation as part of the object and to measure, rather than to suppress, the price of passing from one presentation to another.

\subsection{An extensional anchor}\label{subsec:ConductorComplexity}

Before introducing the intensional data, we isolate the one invariant that belongs to the curve itself and not to any of its presentations.

\begin{definition}[Conductor complexity]\label{def:ConductorComplexity}
Let $\mathcal{E}$ be an elliptic package over a number field $K$ in the sense of Definition~\ref{def:EllipticPackage} below, with underlying curve $E$. The \emph{conductor complexity} of $\mathcal{E}$ is
\[
n(\mathcal{E}) \coloneqq \frac{1}{[K:\Q]} \log N_E
\]
where $N_E$ is the absolute norm of the conductor ideal $\mathfrak{N}_{E/K}$ of $E$ over $K$. Writing $\mathfrak{f}_{v}(\mathcal{E})\coloneqq\mathrm{ord}_{v}(\mathfrak{N}_{E/K})$ for the conductor exponent at a finite place $v$, this reads
\[
  \sum_{v}\mathfrak{f}_{v}(\mathcal{E})\log\mathrm{N}v
  =
  \log N_E
  =
  [K:\Q]\,n(\mathcal{E}).
\]
This quantity depends solely upon the underlying curve $E$, so that it descends to Universe~H and serves as the extensional anchor against which the intensional heights are later compared.
\end{definition}

\subsection{Packages and their morphisms}\label{subsec:EllPkgDef}

\begin{definition}[Elliptic package]\label{def:EllipticPackage}
Let $K$ be a number field. An \emph{elliptic package over $K$} consists of a tuple
\[
\mathcal{E} = \bigl(E, \mathcal{W}, \{\sigma_v\}_{v \in S}, \lambda\bigr)
\]
where the four components are the following:
\begin{enumerate}
\item[(i)] $E$ is an elliptic curve over $K$.
\item[(ii)] $\mathcal{W}$ is a Weierstrass model for $E$ over $\mathcal{O}_K$ or over a localization thereof.
\item[(iii)] $\{\sigma_v\}_{v \in S}$ is a finite collection of local sections or level structures at the places $v$ within a finite set $S$ that contains the places of bad reduction of $E$.
\item[(iv)] $\lambda$ is a label or construction history, keeping the manner in which the data $(E,\mathcal{W},\{\sigma_v\})$ arose from the surrounding arithmetic context -- for instance a derivation from a modular or Shimura curve, a base change along a morphism of number fields, or a pointer to the ambient presentation from which the model was extracted.
\end{enumerate}

A \emph{morphism of elliptic packages} $f \colon \mathcal{E} \to \mathcal{E}'$ over the same base field $K$ is an isomorphism of elliptic curves
\[
  \varphi_f \colon E \xrightarrow{\sim} E'
\]
accompanied by the following compatible data:
\begin{enumerate}
\item[(i)] For each place $v$ of $K$, a change of Weierstrass coordinates relating the local equation $\mathcal{W}_v$ of $\mathcal{E}$ to the local equation $\mathcal{W}'_v$ of $\mathcal{E}'$. After extension of scalars to $K_v$, this change assumes the standard form
  \begin{equation}\label{eq:local-change-of-variables}
    x = u_v(f)^2\,x' + r_v(f)
    \qquad
    y = u_v(f)^3\,y' + u_v(f)^2 s_v(f)\,x' + t_v(f)
  \end{equation}
with $u_v(f) \in K_v^\times$ and $r_v(f), s_v(f), t_v(f) \in K_v$, identifying $\mathcal{W}_v$ with the pullback of $\mathcal{W}'_v$ along $\varphi_f$. For all but finitely many finite places $v$, the coefficients lie in $\mathcal{O}_v$ and $u_v(f)$ is a unit in $\mathcal{O}_v$. The element $u_v(f)$ is termed the \emph{local dilation factor} of $f$ at $v$, and the associated factor
  \[
    c_v(f) \coloneqq u_v(f)^{-12}
  \]
renders the discriminant relation in the simple form
  \[
    \Delta_{\mathcal{W}',v} = c_v(f)\,\Delta_{\mathcal{W},v}.
  \]
\item[(ii)] A correspondence between the local data $\{\sigma_v\}$ and $\{\sigma'_v\}$ keeping the transformation of sections or level structures under $\varphi_f$ and the coordinate changes in~\eqref{eq:local-change-of-variables}.
\end{enumerate}

The labels $\lambda$ and $\lambda'$ are not constrained to match. In applications they often refer to distinct construction histories. They play a passive role in the inequalities below, and yet they remain faithful to the principle that one compares presentations rather than collapsing them. The category of elliptic packages over $K$ is denoted $\mathbf{EllPkg}_K$.
\end{definition}

\begin{remark}[Variable base fields]\label{rem:EllPkgVarCrossRef}
Definition~\ref{def:EllipticPackage} fixes the base field $K$, and the category $\mathbf{EllPkg}_K$ keeps morphisms over $K$. When one wishes to speak of Kummerian moves, such as adjoining an $\ell$-th root $K \hookrightarrow K(\sqrt[\ell]{q})$, then it is convenient to pass to the variable-base category $\mathbf{EllPkg}_{\mathrm{var}}$ of Definition~\ref{def:EllPkgVar} and to separate the model defect $\mu$ from the Kummer defect of Definition~\ref{def:KummerDefect}.
\end{remark}

\subsection{The model defect}\label{subsec:ModelDefect}

\begin{definition}[Model defect]\label{def:ModelDefect}
Let $f \colon \mathcal{E} \to \mathcal{E}'$ be a morphism in $\mathbf{EllPkg}_K$. The \emph{model defect} of $f$ is
\[
\mu(f) \coloneqq \frac{1}{[K:\Q]} \sum_{v} \log^{+}\!\bigl(\lvert c_v(f)\rvert_v^{-1}\bigr)
\]
where $c_v(f) \in K_v^\times$ is the factor of Definition~\ref{def:EllipticPackage}, the absolute values are normalized, the sum runs over all places $v$ of $K$, and $\log^{+}(x) \coloneqq \max\{\log x,0\}$ for $x>0$. The sum is finite since for all but finitely many finite places $v$ one finds $\lvert c_v(f)\rvert_v = 1$, so that the corresponding summand vanishes. For the identity morphism one has $\mu(\mathrm{id}_{\mathcal{E}}) = 0$, and for composable morphisms $f \colon \mathcal{E}_1 \to \mathcal{E}_2$ and $g \colon \mathcal{E}_2 \to \mathcal{E}_3$ one observes the sub-additivity
\[
\mu(g \circ f) \le \mu(f) + \mu(g).
\]
\end{definition}

\begin{proof}[Proof of sub-additivity in Definition~\ref{def:ModelDefect}]
For each place $v$ one has $c_v(g \circ f) = c_v(g)\,c_v(f)$. For positive real numbers $a$ and $b$ one has the elementary inequality $\log^{+}(ab) \le \log^{+}(a) + \log^{+}(b)$. Applying this with $a = \lvert c_v(f)\rvert_v^{-1}$ and $b = \lvert c_v(g)\rvert_v^{-1}$ and then summing over $v$ yields the stated inequality.
\end{proof}

\begin{remark}[Equivalent expressions for the model defect]\label{rem:ModelDefectEquivalent}
Let $f \colon \mathcal{E}\to\mathcal{E}'$ be a morphism in $\mathbf{EllPkg}_K$, with local dilation factors $u_v(f)$ and $c_v(f)=u_v(f)^{-12}$ as in Definition~\ref{def:EllipticPackage}. Then
\[
\mu(f)
=
\frac{1}{[K:\Q]} \sum_v \log^{+}\!\bigl(|u_v(f)|_v^{12}\bigr)
=
\frac{12}{[K:\Q]} \sum_v \log^{+}\!\bigl(|u_v(f)|_v\bigr).
\]
In particular, if $|u_v(f)|_v\le 1$ for all places $v$, then $\mu(f)=0$.
\end{remark}

\begin{remark}[Elliptic packages as a quantitative category]\label{rem:EllPkgQuantitative}
With the model defect $\mu$ in the role of a transport cost, the category $\mathbf{EllPkg}_K$ acquires the structure of a quantitative category of isomorphisms in the sense of Definition~\ref{def:QuantitativeCategory}. Two elliptic packages may be isomorphic in $\mathbf{EllPkg}_K$, and yet be linked by an isomorphism $f$ that carries a positive defect $\mu(f)$. This defect preserves the visibility of the arithmetic price paid to alter the presentation, which is the whole purpose of the intensional reading.
\end{remark}

\subsection{The discriminant height and its transport law}\label{subsec:DiscriminantHeight}

\begin{lemma}[Local change of discriminant]\label{lem:LocalDiscriminantChange}
Let $K_v$ be a local field, and let $\mathcal{W}_v$ and $\mathcal{W}'_v$ be two Weierstrass equations of an elliptic curve over $K_v$. Suppose that $\mathcal{W}'_v$ is obtained from $\mathcal{W}_v$ by a change of variables of the form
\[
  x = u^2 x' + r
  \qquad
  y = u^3 y' + u^2 s x' + t
\]
with $u \in K_v^{\times}$ and $r,s,t \in K_v$. Then their discriminants satisfy
\[
  \Delta(\mathcal{W}'_v) = u^{-12}\,\Delta(\mathcal{W}_v)
\]
in $K_v^\times$.
\end{lemma}

\begin{proof}
This is the customary transformation law for the discriminant under a change of Weierstrass coordinates. One may consult Silverman \cite[III, Proposition~1.1]{SilvermanAEC}.
\end{proof}

We fix the normalization of absolute values once, so that the definition below may be stated without further comment. For each place $v$ of $K$ we take $\lvert\,\cdot\,\rvert_v$ to be the normalized absolute value of the completion $K_v$, with the local degree carried inside the absolute value rather than left as a separate weight. At a finite place $v$ above the rational prime $p$, with normalized valuation $\mathrm{ord}_v$ taking the value $1$ at a uniformizer, this reads $\lvert x\rvert_v = q_v^{-\mathrm{ord}_v(x)}$, where $q_v = \#\kappa(v) = p^{\,f_v}$ is the residue-field cardinality and $f_v = [\kappa(v):\mathbb{F}_p]$ the residue degree. At a real place $\lvert x\rvert_v = \lvert\sigma_v(x)\rvert$ and at a complex place $\lvert x\rvert_v = \lvert\sigma_v(x)\rvert^{2}$, with $\sigma_v$ the corresponding embedding. The local degrees $n_v = [K_v:\Q_p]$ are thereby absorbed into the absolute values, the product formula $\prod_v \lvert x\rvert_v = 1$ holds for every $x\in K^\times$, and the factor $1/[K:\Q]$ in the height below renders the result an absolute height, independent of the field over which the package is presented.

\begin{definition}[Discriminant height on elliptic packages]\label{def:DiscriminantHeight}
Let $\mathcal{E} = (E, \mathcal{W}, \{\sigma_v\}, \lambda)$ be an elliptic package over $K$. The \emph{discriminant height} of $\mathcal{E}$ is
\[
h_\Delta(\mathcal{E}) \coloneqq \frac{1}{[K:\Q]} \sum_{v} \log \max\bigl(1, \lvert \Delta_{\mathcal{W},v} \rvert_v^{-1}\bigr)
\]
where $\Delta_{\mathcal{W},v}$ is the discriminant of the local model $\mathcal{W}_v$, the sum runs over all places $v$ of $K$, and the absolute values are normalized as above.
\end{definition}

In particular, the contribution of a finite place to the sum is
\[
  \log \max\bigl(1, \lvert \Delta_{\mathcal{W},v}\rvert_v^{-1}\bigr)
  = f_v \, \max\bigl(0,\, \mathrm{ord}_v(\Delta_{\mathcal{W}})\bigr)\,\log p
\]
so that the residue degree $f_v$ appears as the explicit local weight from the outset, in agreement with Remark~\ref{rem:HeightsAsLogIntData}. Unlike the conductor complexity of Definition~\ref{def:ConductorComplexity}, this quantity reads the chosen model $\mathcal{W}$ and therefore lives in Universe~A.

\begin{lemma}[Transport inequality for the discriminant height]\label{lem:TransportDiscriminant}
Let $K$ be a number field and let $f \colon \mathcal{E} \to \mathcal{E}'$ be a morphism in $\mathbf{EllPkg}_K$. Suppose that for each place $v$ of $K$ the discriminants satisfy $\Delta_{\mathcal{W}',v} = c_v(f)\,\Delta_{\mathcal{W},v}$. Then
\[
  h_\Delta(\mathcal{E}')
  \le
  h_\Delta(\mathcal{E}) + \mu(f).
\]
\end{lemma}

\begin{proof}
Fix a place $v$ and set $A_v \coloneqq |\Delta_{\mathcal{W},v}|_v^{-1}$ and $\gamma_v \coloneqq \log |c_v(f)|_v^{-1}$, so that $|\Delta_{\mathcal{W}',v}|_v^{-1} = e^{\gamma_v} A_v$. Writing $\alpha_v \coloneqq \log A_v$, one has
\[
\log \max(1, A_v) = \max\{0,\alpha_v\}
\qquad
\log \max(1, e^{\gamma_v}A_v) = \max\{0,\alpha_v+\gamma_v\}.
\]
The elementary inequality $\max\{0,x+y\} \le \max\{0,x\} + \max\{0,y\}$ yields, at each place,
\[
\log \max\bigl(1,\ |\Delta_{\mathcal{W}',v}|_v^{-1}\bigr)
\le
\log \max\bigl(1,\ |\Delta_{\mathcal{W},v}|_v^{-1}\bigr)
+
\log^{+}\!\bigl(|c_v(f)|_v^{-1}\bigr).
\]
Summing over $v$ and dividing by $[K:\Q]$ gives $h_\Delta(\mathcal{E}') \le h_\Delta(\mathcal{E}) + \mu(f)$, following Definition~\ref{def:ModelDefect}.
\end{proof}

\parahead{Non-arbitrariness of the transport cost} The model defect $\mu(f)$ enters as a cost attached to a chosen transport $f$. It is helpful to note that $\mu$ cannot be treated as a free parameter once the morphism $f$ is fixed, since any increase of the discriminant height forces a corresponding payment.

\begin{lemma}[Height increase forces defect]\label{lem:HeightIncreaseForcesDefect}
Let $f \colon \mathcal{E} \to \mathcal{E}'$ be a morphism in $\mathbf{EllPkg}_K$. If $h_\Delta(\mathcal{E}') \ge h_\Delta(\mathcal{E})$, then
\[
h_\Delta(\mathcal{E}') - h_\Delta(\mathcal{E}) \le \mu(f).
\]
In particular, if $h_\Delta(\mathcal{E}') - h_\Delta(\mathcal{E}) = D$ with $D>0$, then $\mu(f)\ge D$.
\end{lemma}

\begin{proof}
Lemma~\ref{lem:TransportDiscriminant} gives $h_\Delta(\mathcal{E}') \le h_\Delta(\mathcal{E}) + \mu(f)$, so the claim follows by rearranging the terms.
\end{proof}

\begin{corollary}[Two-sided control by a transport and its inverse]\label{cor:TwoSidedControlTransport}
Let $f \colon \mathcal{E}\to\mathcal{E}'$ be an isomorphism in $\mathbf{EllPkg}_K$. Then
\[
\bigl|h_\Delta(\mathcal{E}') - h_\Delta(\mathcal{E})\bigr|
\le
\mu(f) + \mu(f^{-1}).
\]
In particular, if $\mu(f)=\mu(f^{-1})=0$, then $h_\Delta(\mathcal{E}') = h_\Delta(\mathcal{E})$.
\end{corollary}

\begin{proof}
Apply Lemma~\ref{lem:TransportDiscriminant} to $f$ and to $f^{-1}$, then combine the two inequalities.
\end{proof}

\begin{remark}[A small interpretive comment]\label{rem:ZeroCostInterpretation}
When the morphism $f$ is fixed, the condition $\mu(f)=0$ expresses that the local coordinate changes carry no dilation cost, and Corollary~\ref{cor:TwoSidedControlTransport} then shows that the two presentations are indistinguishable as far as $h_\Delta$ is concerned. This is the arithmetic shadow of the extensional collapse of Definition~\ref{def:ExtensionalCollapse} -- a costless isomorphism is precisely one that survives the passage to Universe~H.
\end{remark}

\subsection{Keeping the coherence level visible}\label{subsec:EllPkgTwoOne}

In the parts of the theory where functoriality or the commutativity of squares is invoked, coherence data is not ancillary, and we work in a quantitative $(2,1)$-categorical enhancement of the $1$-category $\mathbf{EllPkg}_K$.

\begin{definition}[Quantitative $(2,1)$-category of elliptic packages]\label{def:EllPkgTwoOne}
Fix a number field $K$. A \emph{quantitative $(2,1)$-category of elliptic packages over $K$} is a $(2,1)$-category $\mathfrak{EllPkg}_K$ with the following features:
\begin{enumerate}
\item[\textup{(i)}] Objects are elliptic packages over $K$ in the sense of Definition~\ref{def:EllipticPackage}.
\item[\textup{(ii)}] $1$-morphisms are morphisms of elliptic packages, with the same underlying data as in $\mathbf{EllPkg}_K$.
\item[\textup{(iii)}] All $2$-morphisms are invertible. We write $f \cong g$ to mean that there exists an invertible $2$-morphism $f \Rightarrow g$.
\item[\textup{(iv)}] The defect $\mu$ of Definition~\ref{def:ModelDefect} extends to $1$-morphisms of $\mathfrak{EllPkg}_K$ and is invariant under $2$-isomorphism, so that $f \cong g$ implies $\mu(f)=\mu(g)$.
\end{enumerate}
The homotopy $1$-category $\Ho(\mathfrak{EllPkg}_K)$ may be identified with $\mathbf{EllPkg}_K$ by passing to $2$-isomorphism classes of $1$-morphisms.
\end{definition}

\begin{definition}[Commutativity up to an invertible $2$-morphism]\label{def:CommUpToDefect}
Let $\mathfrak{A}$ be a $(2,1)$-category. Given $1$-morphisms
\[
  A \xrightarrow{u} B \xrightarrow{v} C
  \qquad
  A \xrightarrow{u'} B' \xrightarrow{v'} C
\]
a \emph{comparison} is an invertible $2$-morphism $\alpha \colon v\circ u \Rightarrow v'\circ u'$. In this situation we say that the square \emph{commutes via $\alpha$}.
\end{definition}

\begin{definition}[Pseudonatural transformation in a $(2,1)$-category]\label{def:PseudoNatural}
Let $\mathfrak{A}$ be a $(2,1)$-category and let $F,G\colon\mathfrak{A}\to\mathfrak{A}$ be $2$-functors. A \emph{pseudonatural transformation} $\eta\colon F \Rightarrow G$ consists of the following data:
\begin{enumerate}
\item[\textup{(i)}] for every object $A$, a $1$-morphism $\eta_A\colon F(A)\to G(A)$,
\item[\textup{(ii)}] for every $1$-morphism $f\colon A\to B$, an invertible $2$-morphism $\eta_f \colon \eta_B\circ F(f)\Rightarrow G(f)\circ \eta_A$,
\end{enumerate}
subject to the unit and composition coherence axioms.
\end{definition}

\begin{remark}[Why pseudonaturality stays at the $(2,1)$-level]\label{rem:PseudoNaturalNotBicategorical}
Definitions~\ref{def:CommUpToDefect} and~\ref{def:PseudoNatural} are kept in their native $(2,1)$-categorical form, with coherence registered by explicit invertible $2$-morphisms. The quantitative structure sits on $1$-morphisms, by way of $\mu$, and Definition~\ref{def:EllPkgTwoOne}\textup{(iv)} requires $\mu$ to be invariant under $2$-isomorphism. In this manner the visibility of the $2$-morphisms prevents accidental identifications while the numerical accounting stays attached to the transports themselves.

These $2$-cells exert force in one place and fall silent in another, and the two are easily confused. They exert force at the level of derivation. The invariance of $\mu$ under $2$-isomorphism turns the cost into a quantity attached to the transport rather than to its presentation, and so permits the cost to descend to the homotopy $1$-category by Proposition~\ref{prop:CoherenceStability}. Pseudonaturality plays the parallel role for transport along a morphism, since the invertible $2$-cell $(\tau_\varepsilon)_f$ of the pseudonatural comparison in Hypothesis~\ref{hyp:ReducedNormalizedProfile} renders the two routes around a normalization square equal in cost, so that the family of bounds coheres across $\mathbf{EllPkg}_K$ rather than being chosen object by object with no relation between the choices.

They fall silent at the level of the final statement. The normalized inequality is objectwise, and its derivation uses none of the pseudonaturality data, as Lemma~\ref{lem:ObjectwiseInsensitive} states. When the construction descends to Universe~P in Proposition~\ref{prop:SoundnessDescentNormalization}, the $2$-cells leave no residue. This silence does not mark the coherence level as idle. The descent is sound for the very reason that the cost was first shown invariant under the $2$-cells, so that the coherence data is the scaffolding by which the scalar is proven independent of its presentation. Once that independence is in hand, the scaffolding may be set aside for the arithmetical statement, in the conservative spirit of Proposition~\ref{prop:ConservativeExtension}.
\end{remark}

\begin{proposition}[Coherence stability of the cost]\label{prop:CoherenceStability}
Let $\mathfrak{A}$ be a $(2,1)$-category equipped with a cost $c \colon \mathrm{Mor}_1(\mathfrak{A}) \to \R_{\ge 0}$ satisfying the axioms of Definition~\ref{def:UniverseAFormal}\textup{(iii)}. Then the following hold:
\begin{enumerate}
\item[\textup{(i)}] \textit{Descent to the homotopy $1$-category.} The assignment $[f] \mapsto c(f)$ defines a map $\bar{c} \colon \mathrm{Mor}(\Ho(\mathfrak{A})) \to \R_{\ge 0}$ on the homotopy $1$-category of Definition~\ref{def:HoOneCategory}, satisfying $\bar{c}([\mathrm{id}_A]) = 0$ and the sub-additivity $\bar{c}([g]\circ[f]) \le \bar{c}([f]) + \bar{c}([g])$.
\item[\textup{(ii)}] \textit{Cost equality for $2$-isomorphic composites.} If $f, g \colon A \to B$ are $1$-morphisms and there exists an invertible $2$-morphism $\alpha \colon f \Rightarrow g$, then $c(f) = c(g)$.
\end{enumerate}
In consequence, no cost accrues from the $2$-morphisms that witness the commutativity of faces. Every accumulation of cost arises from the sub-additivity of $c$ applied to the $1$-morphisms that appear as factors of a composite.
\end{proposition}

\begin{proof}
For \textup{(i)}, the invariance of $c$ under $2$-isomorphism, which is part of Definition~\ref{def:EllPkgTwoOne}\textup{(iv)}, shows that $c$ is constant on $2$-isomorphism classes, so that $\bar{c}$ is well defined. The normalization and sub-additivity descend from the corresponding properties of $c$. Part \textup{(ii)} is the invariance under $2$-isomorphism restated for a single pair of parallel $1$-morphisms.
\end{proof}

\subsection{Normalization and the accumulated defect}\label{subsec:Normalization}

To compare an arbitrary presentation with the curve itself, one needs a canonical representative within each isomorphism class and a distinguished transport toward it. We package these choices into a normalization procedure.

\begin{definition}[Normalization defect]\label{def:AccumulatedDefect}
Let $K$ be a number field. A \emph{normalization procedure} for elliptic packages over $K$ comprises the following choices:
\begin{enumerate}
\item[(i)]\label{item:NF_rule} a rule $\mathrm{NF}_K$ assigning to every elliptic package $\mathcal{E}$ over $K$ a normalized elliptic package $\mathrm{NF}_K(\mathcal{E})$ defined over the same field $K$,
\item[(ii)]\label{item:NF_transport} for every $\mathcal{E}$, a distinguished morphism $\nu_{\mathcal{E}} \colon \mathrm{NF}_K(\mathcal{E}) \to \mathcal{E}$ in $\mathbf{EllPkg}_K$ whose local coordinate changes are produced by the normalization rule,
\item[(iii)]\label{item:NF_finitary} a finitary discipline, where the outputs $\mathrm{NF}_K(\mathcal{E})$ and $\nu_{\mathcal{E}}$ depend only upon the finitely presented data of $\mathcal{E}$, so that one has in mind a deterministic procedure that clears denominators, performs local minimization at the finite places, and resolves the remaining ambiguities by a fixed tie-breaker,
\item[(iv)]\label{item:NF_idempotent} an idempotence discipline, where the normalized output is a fixed point of the procedure and its distinguished transport is the strict identity, so that
\[
  \mathrm{NF}_K\bigl(\mathrm{NF}_K(\mathcal{E})\bigr) = \mathrm{NF}_K(\mathcal{E})
  \qquad\text{and}\qquad
  \nu_{\mathrm{NF}_K(\mathcal{E})} = \mathrm{id}_{\mathrm{NF}_K(\mathcal{E})},
\]
\item[(v)]\label{item:NF_curve_invariance} a curve invariance discipline, where for every $\mathcal{E} = (E,\mathcal{W},\{\sigma_v\}_{v\in S},\lambda)$ the normalized output $\mathrm{NF}_K(\mathcal{E})$ has the same underlying curve $E$ and the distinguished morphism $\nu_{\mathcal{E}}$ has underlying curve isomorphism $\varphi_{\nu_{\mathcal{E}}} = \mathrm{id}_E$.
\end{enumerate}
Given such a choice, the \emph{normalization defect} of $\mathcal{E}$ is
\[
\delta_{\mathrm{mod}}(\mathcal{E}) \coloneqq \mu(\nu_{\mathcal{E}})
\]
and one also writes $\operatorname{defect}(\mathcal{E}) \coloneqq \delta_{\mathrm{mod}}(\mathcal{E})$. In particular, the idempotence discipline~\ref{item:NF_idempotent} implies $\defect(\mathrm{NF}_K(\mathcal{E})) = 0$.
\end{definition}

\begin{corollary}[A lower bound for the normalization defect]\label{cor:DefectLowerBoundNormalization}
Let $\mathrm{NF}_K$ be a normalization procedure in the sense of Definition~\ref{def:AccumulatedDefect}. For every elliptic package $\mathcal{E}$ over $K$ one has
\[
\defect(\mathcal{E})
\ge
\max\Bigl\{0,\ h_\Delta(\mathcal{E}) - h_\Delta\bigl(\mathrm{NF}_K(\mathcal{E})\bigr)\Bigr\}.
\]
\end{corollary}

\begin{proof}
By Lemma~\ref{lem:TransportDiscriminant} applied to $\nu_{\mathcal{E}} \colon \mathrm{NF}_K(\mathcal{E}) \to \mathcal{E}$ one has $h_\Delta(\mathcal{E}) \le h_\Delta(\mathrm{NF}_K(\mathcal{E})) + \mu(\nu_{\mathcal{E}})$, and Definition~\ref{def:AccumulatedDefect} identifies $\mu(\nu_{\mathcal{E}})$ with $\defect(\mathcal{E})$.
\end{proof}

\begin{remark}[Deterministic normalization data]\label{rem:DeterministicNormalizationDiscipline}
A normalization procedure is supplied by an explicit terminating algorithm on the admitted package codes, together with proofs of its stated properties. Finite choices are resolved by the order of a declared finite enumeration. A search over an unbounded collection calls for a separate termination proof, since a lexicographic order on an unbounded integer lattice provides no least element for an arbitrary non-empty subset.

Local minimization may use finite residue-ring data together with coordinate parameters drawn from an explicitly presented arithmetic domain. Arbitrary elements of $K_v$ require additional names and cannot be represented by finite integer vectors in an integral basis of $K$.

Over a number field, local minimal equations need not assemble into one global minimal Weierstrass equation. A normalization using independent local minimal models therefore retains the local equations and their transition data as part of its output. Every use of $\mathrm{NF}_K$ specifies whether its output is a global equation or such a local family.
\end{remark}

%------------------------------------------------------------------
\subsection{A dilation tower at the bad prime}\label{subsec:DilationTower}
%------------------------------------------------------------------

Before the discriminant height becomes the term of an inequality, it is instructive to see how it moves along an isomorphism of packages. The two-sided control of Corollary~\ref{cor:TwoSidedControlTransport} permits such movement, and the tower below shows that the movement is unbounded. This is the arithmetic face of the distinction between the static and the dynamic disciplines of Section~\ref{sec:identification}. It is also what determines, in Section~\ref{sec:IntensionalSzpiro}, that a scalar comparable to the discriminant height is read as a quantity transforming under transport rather than as an invariant of an isomorphism class.

\begin{definition}[Dilation tower at the bad prime]\label{def:DilationTowerAtBadPrime}
Let
\[
  E \colon y^2+y=x^3-x
\]
over $\Q$, and let $\mathcal{E}_0$ be the elliptic package whose Weierstrass model is
\[
  \mathcal{W}_0 \colon y^2+y=x^3-x.
\]
For each integer $m\ge 0$, let $\mathcal{W}_m$ be the Weierstrass model obtained from $\mathcal{W}_0$ by the change of variables
\[
  x=37^{-2m}x'
  \qquad
  y=37^{-3m}y'.
\]
Equivalently,
\[
  \mathcal{W}_m \colon y'^2+37^{3m}y'=x'^3-37^{4m}x'.
\]
Let $\mathcal{E}_m$ be the elliptic package obtained from $\mathcal{E}_0$ by replacing the model with $\mathcal{W}_m$ and by transporting the remaining local data along the same coordinate change. We write
\[
  \rho_m \colon \mathcal{E}_0 \longrightarrow \mathcal{E}_m
\]
for the induced morphism of elliptic packages.
\end{definition}

\begin{lemma}[Growth along the dilation tower]\label{lem:DilationTowerGrowth}
For every integer $m\ge 0$, the morphism
\[
  \rho_m \colon \mathcal{E}_0 \longrightarrow \mathcal{E}_m
\]
is an isomorphism in $\mathbf{EllPkg}_{\Q}$, and one has
\[
  h_\Delta(\mathcal{E}_m)=(12m+1)\log 37.
\]
Moreover,
\[
  \mu(\rho_m)=12m\log 37.
\]
\end{lemma}

\begin{proof}
The coordinate change in Definition~\ref{def:DilationTowerAtBadPrime} has local dilation factor
\[
  u=37^{-m}
\]
at the prime $37$. Hence the discriminant transformation law of Lemma~\ref{lem:LocalDiscriminantChange} gives
\[
  \Delta_{\mathcal{W}_m}=u^{-12}\Delta_{\mathcal{W}_0}
  =
  37^{12m}\cdot 37
  =
  37^{12m+1}.
\]
There are no new finite prime divisors. Thus the finite contribution to $h_\Delta(\mathcal{E}_m)$ is
\[
  \log\max\bigl(1,\lvert 37^{12m+1}\rvert_{37}^{-1}\bigr)
  =
  (12m+1)\log 37.
\]
The archimedean contribution vanishes for this height, since $\lvert 37^{12m+1}\rvert_\infty^{-1}<1$. This proves the displayed formula for $h_\Delta(\mathcal{E}_m)$.

The morphism $\rho_m$ is invertible, since the coordinate change is invertible over $\Q$. Its model defect is computed from
\[
  c_{37}(\rho_m)=u^{-12}=37^{12m}.
\]
Therefore
\[
  \mu(\rho_m)
  =
  \log^{+}\bigl(\lvert 37^{12m}\rvert_{37}^{-1}\bigr)
  =
  12m\log 37.
\]
\end{proof}

\begin{proposition}[No isomorphism-invariant scalar comparable to the discriminant height]
\label{prop:NoInvariantLogVolume}
There is no assignment
\[
  V \colon \operatorname{Ob}(\mathbf{EllPkg}_{\Q})\longrightarrow \R_{\ge 0}
\]
satisfying both of the following conditions:
\begin{enumerate}
\item[\textup{(i)}] If two packages are isomorphic in $\mathbf{EllPkg}_{\Q}$, then their values under $V$ are equal.
\item[\textup{(ii)}] There exist constants $A,B>0$ such that for every elliptic package $\mathcal{E}$ over $\Q$ one has
\[
  A^{-1}h_\Delta(\mathcal{E})-B
  \le
  V(\mathcal{E})
  \le
  Ah_\Delta(\mathcal{E})+B.
\]
\end{enumerate}
\end{proposition}

\begin{proof}
Apply the two conditions to the packages $\mathcal{E}_m$ of Definition~\ref{def:DilationTowerAtBadPrime}. By Lemma~\ref{lem:DilationTowerGrowth}, each $\mathcal{E}_m$ is isomorphic to $\mathcal{E}_0$. Hence condition \textup{(i)} gives $V(\mathcal{E}_m)=V(\mathcal{E}_0)$ for all $m\ge 0$. The lower bound in condition \textup{(ii)} gives
\[
  V(\mathcal{E}_0)
  =
  V(\mathcal{E}_m)
  \ge
  A^{-1}(12m+1)\log 37-B
\]
for all $m\ge 0$. The right member tends to infinity with $m$, while the left member is fixed. This is impossible.
\end{proof}

\begin{remark}[The meaning of the obstruction]\label{rem:MeaningOfLogVolumeObstruction}
The obstruction marks the expected cost of keeping positive-cost isomorphisms, and one reads it as a feature of the discipline of Universe~A. The morphisms $\rho_m$ of Definition~\ref{def:DilationTowerAtBadPrime} are precisely the kind of transports that the universe permits. They are isomorphisms of packages, and yet they alter the discriminant height by a visible amount. A scalar comparable to $h_\Delta$ must see this alteration, and so it cannot at the same time be invariant under all such isomorphisms.
\end{remark}

%==================================================================
\section{The Kummerian cost protocol}\label{sec:Kummer}
%==================================================================

Definition~\ref{def:EllipticPackage} fixes the base field, and yet many arithmetic operations move between fields. Adjoining a root, passing to a splitting field, or extending scalars to trivialize a torsion structure all replace $K$ by a larger field $L$. In Universe~H such a passage is often treated as free once a base has been fixed. In Universe~A it carries a cost of its own, and the present section isolates that cost and calibrates it against the ramification of the extension.

\subsection{A variable-base category}\label{subsec:EllPkgVar}

\begin{definition}[A variable-base category of elliptic packages]\label{def:EllPkgVar}
Let $\mathbf{EllPkg}_{K}$ be as in Definition~\ref{def:EllipticPackage}. Define a category $\mathbf{EllPkg}_{\mathrm{var}}$ as follows. An object of $\mathbf{EllPkg}_{\mathrm{var}}$ is a pair $(K,\mathcal{E})$ where $K$ is a number field and $\mathcal{E}$ is an elliptic package over $K$. A morphism
\[
  f \colon (K,\mathcal{E}) \longrightarrow (L,\mathcal{E}')
\]
consists of an inclusion of number fields $\iota \colon K \hookrightarrow L$ together with a morphism of elliptic packages over $L$
\[
  f_L \colon \iota_\ast(\mathcal{E}) \longrightarrow \mathcal{E}'
\]
where $\iota_\ast(\mathcal{E})$ is the package obtained from $\mathcal{E}$ by base change along $\iota$, with the same underlying Weierstrass equation viewed over $\mathcal{O}_L$ and with the same label transported as a symbol.
\end{definition}

\begin{remark}[The discriminant height is unchanged by base change]\label{rem:HeightBaseChangeInvariance}
Let $\iota\colon K\hookrightarrow L$ be an inclusion of number fields. Since the absolute values of \S\ref{subsec:DiscriminantHeight} carry their local degrees inside them, the places $w$ of $L$ above a place $v$ of $K$ satisfy $\sum_{w\mid v}[L_w:K_v]=[L:K]$, and the factor $1/[L:\Q]$ restores the original weight. Hence
\[
  h_{\Delta}\bigl(\iota_\ast(\mathcal{E})\bigr)
  =
  h_{\Delta}(\mathcal{E}).
\]
The height therefore moves without cost along the base-change part of a morphism of $\mathbf{EllPkg}_{\mathrm{var}}$, and every increase recorded in the sequel arises from re-minimization over the larger field. This allows the radical transports of Definition~\ref{def:RadicalTransport} to be priced by their ramification alone.
\end{remark}

\begin{remark}[Why the variable base is introduced]\label{rem:WhyEllPkgVar}
The category $\mathbf{EllPkg}_{\mathrm{var}}$ is not a replacement for $\mathbf{EllPkg}_K$. Its role is to provide a place where one may speak of a morphism that witnesses a passage $K \leadsto L$. In Universe~A the morphism $f\colon (K,\mathcal{E}) \to (L,\mathcal{E}_L)$ may carry a strictly positive distortion cost $\mu(f) > 0$, reflecting the arithmetic price of the field extension. This term enters the intensional Szpiro inequality of Section~\ref{sec:IntensionalSzpiro} through the detachment defect introduced below.
\end{remark}

\subsection{Radical transports and the Kummer defect}\label{subsec:KummerDefect}

We isolate the morphisms that correspond to root extraction.

\begin{definition}[Radical transport morphism]\label{def:RadicalTransport}
Let $(K,\mathcal{E})$ be an object of $\mathbf{EllPkg}_{\mathrm{var}}$. Fix an integer $\ell\ge 2$ and an element $q\in K^\times$, and let $L = K(\sqrt[\ell]{q})$. A \emph{radical transport morphism} of degree $\ell$ is a morphism
\[
  f_{\mathrm{rad}} \colon (K,\mathcal{E}) \longrightarrow (L,\mathcal{E}_L)
\]
in $\mathbf{EllPkg}_{\mathrm{var}}$ for which the underlying package $\mathcal{E}_L$ is the base change of $\mathcal{E}$ to $L$, and for which the Weierstrass coordinate changes in the package part are trivial at every place. In particular, the model defect contribution coming from the local dilation factors vanishes.
\end{definition}

The classical part of the cost sees nothing through such a morphism, so that one needs a separate term to keep the detachment visible.

\begin{definition}[Kummer degree and Kummer defect]\label{def:KummerDefect}
Let $f \colon (K,\mathcal{E}) \to (L,\mathcal{E}')$ be a morphism in $\mathbf{EllPkg}_{\mathrm{var}}$ with underlying field inclusion $\iota \colon K \hookrightarrow L$. The \emph{Kummer degree} of $f$ is
\[
  \deg_{\mathrm{kum}}(f) \coloneqq [L:K]
\]
and, having fixed a constant $C_{\mathrm{kum}}\ge 0$ depending only upon the chosen protocol, the \emph{Kummer defect} of $f$ is
\[
  \mu_{\mathrm{kum}}(f) \coloneqq C_{\mathrm{kum}} \cdot \log\bigl(\deg_{\mathrm{kum}}(f)\bigr).
\]
\end{definition}

\begin{remark}[Total transport cost]\label{rem:TotalTransportCost}
For morphisms internal to $\mathbf{EllPkg}_{K}$, the model defect $\mu$ of Definition~\ref{def:ModelDefect} is the operative cost. For morphisms in $\mathbf{EllPkg}_{\mathrm{var}}$ it is convenient to keep two layers visible, and one writes
\[
  \mu_{\mathrm{tot}}(f) \coloneqq \mu_{\mathrm{mod}}(f) + \mu_{\mathrm{kum}}(f)
\]
where $\mu_{\mathrm{mod}}(f)$ is the model defect computed from the local dilation factors in the package component of $f$, and $\mu_{\mathrm{kum}}(f)$ is the detachment term of Definition~\ref{def:KummerDefect}. A radical transport in the sense of Definition~\ref{def:RadicalTransport} satisfies $\mu_{\mathrm{mod}}(f_{\mathrm{rad}})=0$, while $\mu_{\mathrm{kum}}(f_{\mathrm{rad}})$ may be non-zero.
\end{remark}

\begin{lemma}[Cost of Kummer detachment]\label{lem:KummerDetachmentCost}
Let $f_{\mathrm{rad}}$ be a radical transport morphism of degree $\ell$ in the sense of Definition~\ref{def:RadicalTransport}. Then
\[
  \mu_{\mathrm{kum}}(f_{\mathrm{rad}})
  \le
  C_{\mathrm{kum}} \cdot \log(\ell).
\]
\end{lemma}

\begin{proof}
Let $L=K(\sqrt[\ell]{q})$, so that $[L:K]\le \ell$. By Definition~\ref{def:KummerDefect} one then has $\mu_{\mathrm{kum}}(f_{\mathrm{rad}}) = C_{\mathrm{kum}} \cdot \log([L:K]) \le C_{\mathrm{kum}} \cdot \log(\ell)$.
\end{proof}

\begin{remark}[Degree cost and ramification]\label{rem:KummerDefectArithmeticCalibration}
The cost $\mu_{\mathrm{kum}}$ measures the degree of a field extension through a chosen logarithmic convention. Multiplicativity in a tower renders this degree cost additive.

An extension $K(\sqrt[\ell]{q})/K$ need not be Galois, and it need not be abelian. Ramification bounds therefore call for their own hypotheses and their own local calculations. The finite ramification support is contained in the places dividing $\ell$ and in the support of $q$, while its contribution can vary with the radicand.

The support term in Hypothesis~\ref{hyp:ControlledDescentWithSupport} retains this further arithmetic dependence. The degree cost and the support-sensitive descent estimate are kept apart as distinct quantities.
\end{remark}

\subsection{Descent back to the base field}\label{subsec:ControlledDescentBridge}

To compare a bound proved over a Kummer extension $L/K$ with a statement over the original field $K$, one needs a descent step. In the variable-base category we isolate this as a transport whose cost carries two terms of distinct provenance, one logarithmic in the Kummer degree and one sensitive to the tame support of the radicand. The necessity of the second term, and the explicit form of both, are taken up in Section~\ref{subsec:ControlledDescentUniformity}.

\begin{definition}[Package descent morphism]\label{def:DescentMorphism}
Let $L/K$ be a finite extension of number fields and let $\mathcal{E}_L$ be an elliptic package over $L$. The \emph{descent morphism} $f_{\mathrm{desc}} \colon \mathcal{E}_L \to \mathcal{E}_{L\mid K}$ is the morphism in $\mathbf{EllPkg}_{\mathrm{var}}$ defined by two operations. First, by scalar restriction, where the coefficients of the Weierstrass model are viewed over the maximal subfield $K$. Second, by re-minimization, where a normalization $\mathrm{NF}_K$ is applied to the restricted coefficients at each place $v$ of $K$ following Tate's algorithm. The dilation factors $u_v$ of $f_{\mathrm{desc}}$ keep the ratio between the $L$-minimal discriminant and the re-normalized $K$-minimal discriminant.
\end{definition}

\begin{definition}[Controlled descent datum]\label{def:ControlledDescentDatum}
Let $K$ be a number field, let $L/K$ be a finite extension, let $\mathcal{E}$ be an elliptic package over $K$, and write $\mathcal{E}_L$ for its base change to $L$. A \emph{controlled descent datum} for $(\mathcal{E},L/K)$ is a morphism in $\mathbf{EllPkg}_{\mathrm{var}}$
\begin{equation}\label{eq:ControlledDescentMorphism}
  \delta_{\mathcal{E},L/K} \colon (L,\mathrm{NF}_L(\mathcal{E}_L))
  \longrightarrow (K,\mathrm{NF}_K(\mathcal{E})).
\end{equation}
Its \emph{descent cost} is $\mu_{\mathrm{tot}}(\delta_{\mathcal{E},L/K})$.
\end{definition}

\begin{definition}[Tame Kummer support]\label{def:TameKummerSupport}
Let $K$ be a number field and let
\[
  L=K(\sqrt[\ell]{q})
\]
with $\ell\ge 2$ and $q\in K^{\times}$. The \emph{tame Kummer support} of $(q,\ell)$ is
\[
  \operatorname{Supp}^{\mathrm{tame}}_{K}(q,\ell)
  =
  \bigl\{
    v<\infty
    \mid
    v\nmid \ell
    \text{ and }
    v(q)\ne 0
  \bigr\}.
\]
Its logarithmic size is
\[
  R_{K}(q,\ell)
  \coloneqq
  \frac{1}{[K:\Q]}
  \sum_{v\in \operatorname{Supp}^{\mathrm{tame}}_{K}(q,\ell)}
  \log \mathrm{N}v
\]
where $\mathrm{N}v$ denotes the absolute norm of the prime ideal $v$.
\end{definition}

\begin{remark}[A harmless overcount]\label{rem:TameSupportOvercount}
The set of Definition~\ref{def:TameKummerSupport} may overcount the ramified tame places, since a valuation $v(q)$ divisible by $\ell$ can disappear after the removal of an $\ell$-th power. This overcount is useful. It is computable directly from the finite code of $q$, and it gives an upper bound for the tame ramification places of $L/K$ outside the primes above $\ell$.
\end{remark}

\begin{hypothesis}[Controlled descent with tame support]\label{hyp:ControlledDescentWithSupport}
Fix a number field $K$. There exist constants
\[
  C_{\mathrm{desc}}>0
  \qquad
  A_{\mathrm{desc}}\ge 0
  \qquad
  B_{\mathrm{desc},K}\ge 0
\]
such that the following holds. For every elliptic package $\mathcal{E}$ over $K$, and for every Kummer extension $L=K(\sqrt[\ell]{q})$, there exists a controlled descent datum $\delta_{\mathcal{E},L/K}$ in the sense of Definition~\ref{def:ControlledDescentDatum} satisfying
\begin{equation}\label{eq:ControlledDescentWithSupport}
  \mu_{\mathrm{tot}}(\delta_{\mathcal{E},L/K})
  \le
  C_{\mathrm{desc}}\log \ell
  +
  A_{\mathrm{desc}}\,R_{K}(q,\ell)
  +
  B_{\mathrm{desc},K}.
\end{equation}
\end{hypothesis}

\begin{remark}[The descent cost belongs to finite-Galois arithmetic]\label{rem:DescentCostClassicalNature}
Hypothesis~\ref{hyp:ControlledDescentWithSupport} occupies a distinguished position among the assumptions of this paper. The controlled descent datum $\delta_{\mathcal{E},L/K}$ of Definition~\ref{def:ControlledDescentDatum} involves three operations from the standard repertoire of algebraic number theory: scalar restriction of Weierstrass coefficients, re-minimization by way of Tate's algorithm at each place of $K$, and the computation of the dilation factors $u_v$ arising from the re-minimization. Each of these belongs to finite-Galois arithmetic. The two coefficients of the bound draw upon distinct sources of ramification. The term $C_{\mathrm{desc}}\log \ell$ collects the wild contribution at the places above $\ell$, read through the local theory of different exponents (see Serre~\cite{Serre1979} and Neukirch~\cite{Neukirch1999}), together with the Kummer detachment cost $C_{\mathrm{kum}}$ of Lemma~\ref{lem:KummerDetachmentCost}. The term $A_{\mathrm{desc}}R_{K}(q,\ell)$ collects the tame contribution at the places in the support of the radicand, of Definition~\ref{def:TameKummerSupport}. In this sense the descent cost is not an exotic quantity, and the bound is a calibrated restatement of a familiar ramification estimate within the language of transport costs. That both of its terms are forced, and that both may be made explicit away from the primes above $2$ and $3$, is shown in Section~\ref{subsec:ControlledDescentUniformity}.
\end{remark}

\begin{lemma}[Bridge inequality for the discriminant height]\label{lem:BridgeDescentHeight}
Let $K$ be a number field and assume Hypothesis~\ref{hyp:ControlledDescentWithSupport}. Let $\mathcal{E}$ be an elliptic package over $K$ and let $L=K(\sqrt[\ell]{q})$ be a Kummer extension, with $\mathcal{E}_L$ the base change of $\mathcal{E}$ to $L$. Then
\[
  h_\Delta\bigl(\mathrm{NF}_K(\mathcal{E})\bigr)
  \le
  h_\Delta\bigl(\mathrm{NF}_L(\mathcal{E}_L)\bigr)
  + C_{\mathrm{desc}}\log \ell
  + A_{\mathrm{desc}}\,R_{K}(q,\ell)
  + B_{\mathrm{desc},K}.
\]
\end{lemma}

\begin{proof}
By Hypothesis~\ref{hyp:ControlledDescentWithSupport} choose a controlled descent datum $\delta_{\mathcal{E},L/K} \colon (L,\mathrm{NF}_L(\mathcal{E}_L)) \to (K,\mathrm{NF}_K(\mathcal{E}))$ with $\mu_{\mathrm{tot}}(\delta_{\mathcal{E},L/K}) \le C_{\mathrm{desc}}\log \ell + A_{\mathrm{desc}}R_{K}(q,\ell) + B_{\mathrm{desc},K}$. The transport inequality of Lemma~\ref{lem:TransportDiscriminant}, applied to the package component of $\delta_{\mathcal{E},L/K}$, gives
\[
  h_\Delta\bigl(\mathrm{NF}_K(\mathcal{E})\bigr)
  \le
  h_\Delta\bigl(\mathrm{NF}_L(\mathcal{E}_L)\bigr)
  + \mu_{\mathrm{mod}}(\delta_{\mathcal{E},L/K})
  \le
  h_\Delta\bigl(\mathrm{NF}_L(\mathcal{E}_L)\bigr)
  + \mu_{\mathrm{tot}}(\delta_{\mathcal{E},L/K})
\]
since $\mu_{\mathrm{mod}} \le \mu_{\mathrm{tot}}$ by Remark~\ref{rem:TotalTransportCost}. Substituting the bound on the descent cost yields the claim.
\end{proof}

\subsection{The necessity of the tame support term and explicit local bounds}\label{subsec:ControlledDescentUniformity}

The descent estimate of Hypothesis~\ref{hyp:ControlledDescentWithSupport} carries two terms, and each is present for a reason given below. This subsection establishes that the tame support term is forced -- a bound logarithmic in the Kummer degree alone cannot hold uniformly as the radicand varies -- and it then proves the estimate explicitly away from the residue characteristics $2$ and $3$, where the bound holds with $C_{\mathrm{desc}}=12+C_{\mathrm{kum}}$ and $A_{\mathrm{desc}}=12$. The part that follows from local minimization is thereby separated from the part that must see the tame support of $q$.

There is also a small directional matter. The arrows of Definition~\ref{def:EllPkgVar} follow inclusions of fields from the smaller field to the larger one, while the descent arrow of Definition~\ref{def:ControlledDescentDatum} points in the reverse direction. For the estimates below, this reverse arrow is read as a finite descent code in a descent completion of $\mathbf{EllPkg}_{\mathrm{var}}$. This completion has the same objects as $\mathbf{EllPkg}_{\mathrm{var}}$. In addition to the base-change arrows, it carries formal reverse descent arrows whose data are finite coordinate changes produced by scalar restriction and re-minimization.

\begin{definition}[Height-admissible descent datum]\label{def:HeightAdmissibleDescentDatum}
Let $L/K$ be a finite extension of number fields, and let $\mathcal{E}$ be an elliptic package over $K$. A \emph{height-admissible descent datum} for $(\mathcal{E},L/K)$ is a controlled descent datum
\[
  \delta_{\mathcal{E},L/K}
  \colon
  (L,\mathrm{NF}_{L}(\mathcal{E}_{L}))
  \longrightarrow
  (K,\mathrm{NF}_{K}(\mathcal{E}))
\]
in the sense of Definition~\ref{def:ControlledDescentDatum}, read as an arrow of the descent completion described above, such that the transport inequality
\[
  h_{\Delta}\bigl(\mathrm{NF}_{K}(\mathcal{E})\bigr)
  \le
  h_{\Delta}\bigl(\mathrm{NF}_{L}(\mathcal{E}_{L})\bigr)
  +
  \mu_{\mathrm{tot}}(\delta_{\mathcal{E},L/K})
\]
holds.
\end{definition}

\begin{remark}[Finite nature of descent codes]\label{rem:FiniteDescentCodes}
A height-admissible descent datum is a finite object. It consists of a certified extension code for $L/K$, local coordinate changes at the finitely many places where the base-changed normalized model ceases to be minimal, and the finite list of local dilation exponents that enter the model defect. Thus the descent completion does not enlarge the foundational strength of the framework. It gives a typed place for the reverse arrows that are computed by the finite procedures of Definition~\ref{def:EllPkgCode} and Remark~\ref{rem:PkgCodeComputability}.
\end{remark}

\subsubsection{The tame support term is forced}\label{subsubsec:QuadraticTwistObstruction}

The first result shows that the coefficient $A_{\mathrm{desc}}$ of Hypothesis~\ref{hyp:ControlledDescentWithSupport} cannot be taken to be zero. A quadratic twist family lets a single tame prime grow without bound while the Kummer degree stays fixed at two.

\begin{definition}[Quadratic twist descent family]\label{def:QuadraticTwistDescentFamily}
For each rational prime $p\ge 5$, let
\[
  E_{p}\colon y^{2}=x^{3}-p^{2}x
\]
over $\Q$. Let $\mathcal{E}_{p}$ be the elliptic package over $\Q$ whose Weierstrass model is this equation, with the remaining package data chosen by the fixed deterministic conventions of Definition~\ref{def:AccumulatedDefect}. Let
\[
  L_{p}=\Q(\sqrt p)
\]
and write $t_{p}=\sqrt p$. Over $L_{p}$, the change of variables
\[
  x=t_{p}^{2}X
  \qquad
  y=t_{p}^{3}Y
\]
identifies $E_{p}$ with the curve
\[
  E_{0}\colon Y^{2}=X^{3}-X.
\]
\end{definition}

\begin{lemma}[Height gap in the twist family]\label{lem:QuadraticTwistHeightGap}
For the family of Definition~\ref{def:QuadraticTwistDescentFamily}, one has
\[
  h_{\Delta}\bigl(\mathrm{NF}_{\Q}(\mathcal{E}_{p})\bigr)
  -
  h_{\Delta}\bigl(\mathrm{NF}_{L_{p}}((\mathcal{E}_{p})_{L_{p}})\bigr)
  \ge
  6\log p-6\log 2.
\]
\end{lemma}

\begin{proof}
The discriminant of the equation $y^{2}=x^{3}-p^{2}x$ is $\Delta(E_{p})=64p^{6}$. At the prime $p\ge 5$, the displayed equation is minimal. Indeed, for a short Weierstrass equation in residue characteristic at least $5$, the condition $v_{p}(-p^{2})=2<4$ prevents a scaling by $p^{-1}$ from preserving integrality. Thus the $p$-part of $h_{\Delta}(\mathrm{NF}_{\Q}(\mathcal{E}_{p}))$ is exactly $6\log p$, and all other finite contributions are non-negative. Hence
\[
  h_{\Delta}\bigl(\mathrm{NF}_{\Q}(\mathcal{E}_{p})\bigr)\ge 6\log p.
\]

Over $L_{p}$, Definition~\ref{def:QuadraticTwistDescentFamily} identifies $(E_{p})_{L_{p}}$ with $E_{0}$. The model $Y^{2}=X^{3}-X$ has discriminant $64$. Therefore its normalized finite discriminant contribution over $L_{p}$ is at most $6\log 2$. Since normalization can only decrease the discriminant height, one has
\[
  h_{\Delta}\bigl(\mathrm{NF}_{L_{p}}((\mathcal{E}_{p})_{L_{p}})\bigr)
  \le
  6\log 2.
\]
The two estimates give the claim.
\end{proof}

\begin{theorem}[The support term cannot be dropped]\label{thm:NoUniformControlledDescent}
There are no constants $C>0$ and $B\ge 0$ with the following property. For every prime $p\ge 5$, there exists a height-admissible descent datum
\[
  \delta_{p}
  \colon
  (L_{p},\mathrm{NF}_{L_{p}}((\mathcal{E}_{p})_{L_{p}}))
  \longrightarrow
  (\Q,\mathrm{NF}_{\Q}(\mathcal{E}_{p}))
\]
satisfying $\mu_{\mathrm{tot}}(\delta_{p})\le C\log 2+B$. Consequently a descent bound of the form $C_{\mathrm{desc}}\log \ell+B_{\mathrm{desc},K}$, carrying no tame support term, fails to hold uniformly over $K=\Q$ as the radicand $q$ varies, and the coefficient $A_{\mathrm{desc}}$ of Hypothesis~\ref{hyp:ControlledDescentWithSupport} is strictly positive.
\end{theorem}

\begin{proof}
Assume such constants $C$ and $B$ have been fixed. Since $\delta_{p}$ is height-admissible, Definition~\ref{def:HeightAdmissibleDescentDatum} gives
\[
  h_{\Delta}\bigl(\mathrm{NF}_{\Q}(\mathcal{E}_{p})\bigr)
  -
  h_{\Delta}\bigl(\mathrm{NF}_{L_{p}}((\mathcal{E}_{p})_{L_{p}})\bigr)
  \le
  \mu_{\mathrm{tot}}(\delta_{p}).
\]
By Lemma~\ref{lem:QuadraticTwistHeightGap}, this implies $6\log p-6\log 2\le C\log 2+B$ for every prime $p\ge 5$. The primes are unbounded, and the left member tends to infinity with $p$. This is impossible. The tame prime $p$ enters the family of Definition~\ref{def:QuadraticTwistDescentFamily} through the radicand, and the term $R_{\Q}(q,\ell)$ of Definition~\ref{def:TameKummerSupport} carries it, so that $A_{\mathrm{desc}}$ is positive.
\end{proof}

\begin{remark}[The source of the obstruction]\label{rem:SourceOfUniformDescentObstruction}
The obstruction does not lie in wild ramification at the prime $\ell$, since in the family above one has $\ell=2$ throughout. The growing term is the tame prime $p$ that enters through $q=p$, and the quadratic twist makes this prime visible in the descent from $L_{p}$ back to $\Q$. Thus a bound logarithmic in $\ell$ alone cannot measure the descent cost, and the tame support of Definition~\ref{def:TameKummerSupport} is the quantity that does.
\end{remark}

\subsubsection{Local scaling under extension}\label{subsubsec:LocalScalingUnderExtension}

The model part of the descent cost is governed by the scaling needed to pass from the base-changed $K$-minimal equation to an $L$-minimal equation. The next local bound is elementary away from residue characteristics $2$ and $3$.

\begin{definition}[Local scaling exponent]\label{def:LocalScalingExponent}
Let $F$ be a non-archimedean local field, let $F'/F$ be a finite extension with ramification index $e$, and let $E/F$ be given by a minimal integral Weierstrass equation $\mathcal{W}$. After base change to $F'$, let $\mathcal{W}'_{\min}$ be a minimal Weierstrass equation for $E_{F'}$. The \emph{local scaling exponent} $r(E,F'/F)$ is the non-negative integer determined by
\[
  v_{F'}(\Delta_{\mathcal{W}'_{\min}})
  =
  e\,v_{F}(\Delta_{\mathcal{W}})
  -
  12\,r(E,F'/F).
\]
Equivalently, $r(E,F'/F)$ is the valuation of the dilation factor that is removed from the base-changed equation during minimization over $F'$.
\end{definition}

\begin{lemma}[Prime-to-six local scaling bound]\label{lem:PrimeToSixLocalScalingBound}
Let $F$ be a finite extension of $\Q_{p}$ with $p\ge 5$. Let $F'/F$ be a finite extension with ramification index $e$. For every elliptic curve $E/F$ given by a minimal integral Weierstrass equation, the local scaling exponent of Definition~\ref{def:LocalScalingExponent} satisfies
\[
  0\le r(E,F'/F)\le e-1.
\]
\end{lemma}

\begin{proof}
Let $c_{4}$, $c_{6}$, and $\Delta$ be the invariants of a minimal integral Weierstrass equation for $E/F$. In residue characteristic at least $5$, the minimality criterion for Weierstrass equations gives
\[
  \min
  \left\{
    \frac{v_{F}(c_{4})}{4},
    \frac{v_{F}(c_{6})}{6},
    \frac{v_{F}(\Delta)}{12}
  \right\}
  <1
\]
with the zero invariants omitted from the minimum. The discriminant is non-zero, so the displayed minimum is finite.

Suppose that a scaling by a uniformizer of $F'$ to exponent $r$ appears in the minimization over $F'$. Since the scaled invariants must remain integral, one has $4r\le e\,v_{F}(c_{4})$ whenever $c_{4}\ne 0$, one has $6r\le e\,v_{F}(c_{6})$ whenever $c_{6}\ne 0$, and one has $12r\le e\,v_{F}(\Delta)$. Thus
\[
  r
  \le
  e
  \min
  \left\{
    \frac{v_{F}(c_{4})}{4},
    \frac{v_{F}(c_{6})}{6},
    \frac{v_{F}(\Delta)}{12}
  \right\}
  <e.
\]
Since $r$ is an integer, it follows that $r\le e-1$.
\end{proof}

The primes above $2$ and $3$ require the full local minimization procedure. Since there are only finitely many such places for a fixed number field $K$, the remaining input can be isolated as a finite local certificate.

\begin{hypothesis}[Small-prime local scaling certificate]\label{hyp:SmallPrimeLocalScalingCertificate}
Let $K$ be a fixed number field. For every finite place $v$ of $K$ above $2$ or $3$, there exists an integer $\eta_{v}\ge 0$ with the following property. For every finite extension $L_{w}/K_{v}$ with ramification index $e_{w/v}$, and for every elliptic curve over $K_{v}$ given by a minimal integral Weierstrass equation, the local scaling exponent satisfies
\[
  r(E,L_{w}/K_{v})
  \le
  e_{w/v}-1+\eta_{v}.
\]
\end{hypothesis}

\begin{remark}[Status of the small-prime certificate]\label{rem:SmallPrimeCertificateStatus}
Hypothesis~\ref{hyp:SmallPrimeLocalScalingCertificate} is a finite list of uniform local assertions once $K$ is fixed, one for each place over $2$ and $3$. Each assertion quantifies over all curves and over all finite local extensions at its place, so its finite proof, when supplied, is a different object from the numerical parameter $\eta_v$ that it certifies. The prime-to-six places are covered by Lemma~\ref{lem:PrimeToSixLocalScalingBound}.
\end{remark}

\subsubsection{A support-controlled descent bound}\label{subsubsec:SupportControlledDescentTheorem}

\begin{proposition}[Support-controlled descent bound]\label{prop:SupportControlledDescentBound}
Fix a number field $K$, and assume Hypothesis~\ref{hyp:SmallPrimeLocalScalingCertificate}. Let $L=K(\sqrt[\ell]{q})$ with $\ell\ge 2$ and $q\in K^{\times}$. Then for every elliptic package $\mathcal{E}$ over $K$, the descent completion contains a height-admissible descent datum
\[
  \delta_{\mathcal{E},L/K}
  \colon
  (L,\mathrm{NF}_{L}(\mathcal{E}_{L}))
  \longrightarrow
  (K,\mathrm{NF}_{K}(\mathcal{E}))
\]
satisfying
\[
  \mu_{\mathrm{tot}}(\delta_{\mathcal{E},L/K})
  \le
  (12+C_{\mathrm{kum}})\log \ell
  +
  12\,R_{K}(q,\ell)
  +
  B_{K,6}.
\]
Here
\[
  B_{K,6}
  \coloneqq
  \frac{12}{[K:\Q]}
  \sum_{v\mid 6}
  \eta_{v}\log \mathrm{N}v
\]
with $\eta_{v}=0$ for places $v\nmid 6$.
\end{proposition}

\begin{proof}
Start with the normalized $K$-model $\mathrm{NF}_{K}(\mathcal{E})$, and base change it to $L$. At each place $w$ of $L$ above a place $v$ of $K$, let $r_{w}$ be the local scaling exponent needed to pass from this base-changed model to the $L_{w}$-minimal model. The model part of the reverse descent cost is
\[
  \mu_{\mathrm{mod}}(\delta_{\mathcal{E},L/K})
  =
  \frac{12}{[L:\Q]}
  \sum_{w}
  r_{w}\log \mathrm{N}w.
\]
Unramified extensions preserve minimality of Weierstrass equations, so places unramified in $L/K$ have $r_{w}=0$. The polynomial $T^{\ell}-q$ has discriminant divisible only by primes above $\ell$ and by primes in the support of $q$. Hence $L/K$ is unramified outside the places above $\ell$ and the tame support of Definition~\ref{def:TameKummerSupport}.

For $v\nmid 6$, Lemma~\ref{lem:PrimeToSixLocalScalingBound} gives $r_{w}\le e_{w/v}-1$. For $v\mid 6$, Hypothesis~\ref{hyp:SmallPrimeLocalScalingCertificate} gives $r_{w}\le e_{w/v}-1+\eta_{v}$. Consequently, for each ramified place $v$ of $K$,
\[
  \frac{12}{[L:\Q]}
  \sum_{w\mid v}
  r_{w}\log \mathrm{N}w
  \le
  \frac{12}{[L:\Q]}
  \sum_{w\mid v}
  e_{w/v}\log \mathrm{N}w
  +
  \frac{12\eta_{v}}{[L:\Q]}
  \sum_{w\mid v}
  \log \mathrm{N}w.
\]
The first term on the right is bounded by $\tfrac{12}{[K:\Q]}\log \mathrm{N}v$, since $\sum_{w\mid v}e_{w/v}\log\mathrm{N}w=[L:K]\log\mathrm{N}v$. The second term contributes only at the finite set of places $v\mid 6$, and its total contribution is bounded by $B_{K,6}$.

Summing over the ramified places away from $\ell$ gives the term $12R_{K}(q,\ell)$. Summing over the places $v\mid \ell$ gives
\[
  \frac{12}{[K:\Q]}
  \sum_{v\mid \ell}\log \mathrm{N}v
  \le
  12\log \ell.
\]
Therefore $\mu_{\mathrm{mod}}(\delta_{\mathcal{E},L/K})\le 12R_{K}(q,\ell)+12\log \ell+B_{K,6}$. The Kummer detachment term satisfies $\mu_{\mathrm{kum}}(\delta_{\mathcal{E},L/K})\le C_{\mathrm{kum}}\log \ell$ by Definition~\ref{def:KummerDefect}. Adding the two estimates gives the displayed bound for $\mu_{\mathrm{tot}}$.
\end{proof}

\begin{remark}[The hypothesis is proved away from the small primes]\label{rem:SupportHypothesisProved}
Proposition~\ref{prop:SupportControlledDescentBound} proves Hypothesis~\ref{hyp:ControlledDescentWithSupport} under the finite local certificate of Hypothesis~\ref{hyp:SmallPrimeLocalScalingCertificate}, with $C_{\mathrm{desc}}=12+C_{\mathrm{kum}}$, $A_{\mathrm{desc}}=12$, and $B_{\mathrm{desc},K}=B_{K,6}$. The bound is therefore explicit, and the only local input that remains beyond the prime-to-six estimate is concentrated at the finitely many places above $2$ and $3$. The quadratic twist family of Theorem~\ref{thm:NoUniformControlledDescent} proves that a positive support contribution is necessary for a uniform estimate of the stated form. Proposition~\ref{prop:SupportControlledDescentBound}, under its local assumptions, supplies the coefficient $12$. No assertion of optimality is made for that coefficient.
\end{remark}

\begin{corollary}[Recovery of the purely logarithmic form under bounded tame support]
\label{cor:RecoveryControlledDescentLogCost}
Assume Hypothesis~\ref{hyp:ControlledDescentWithSupport}. Let $\mathcal{L}$ be a family of Kummer extensions $L=K(\sqrt[\ell]{q})$ for which there is a constant $B_{\mathrm{supp},K}$ with $R_{K}(q,\ell)\le B_{\mathrm{supp},K}$ for every member of the family. Then for every member of $\mathcal{L}$ the descent bound assumes the purely logarithmic form
\[
  \mu_{\mathrm{tot}}(\delta_{\mathcal{E},L/K})
  \le
  C_{\mathrm{desc}}\log \ell
  +
  B_{\mathrm{desc},K}^{\prime}
\]
with $B_{\mathrm{desc},K}^{\prime}=A_{\mathrm{desc}}B_{\mathrm{supp},K}+B_{\mathrm{desc},K}$.
\end{corollary}

\begin{proof}
Substitute the bound $R_{K}(q,\ell)\le B_{\mathrm{supp},K}$ into the inequality of Hypothesis~\ref{hyp:ControlledDescentWithSupport}. The support term is thereby absorbed into the constant, and the bound assumes the displayed purely logarithmic form.
\end{proof}

\begin{remark}[Two regimes of the descent bound]\label{rem:DescentRegimes}
The estimate of Hypothesis~\ref{hyp:ControlledDescentWithSupport} is read in two regimes. For a family of bounded tame support, as in Corollary~\ref{cor:RecoveryControlledDescentLogCost}, the support term collapses into a constant and the bound is logarithmic in the Kummer degree. For a family of unbounded support, the term $R_{K}(q,\ell)$ is carried through the later inequalities, where it measures the cost of the primes that the radicand introduces.
\end{remark}

\subsubsection{Arithmetic formalization}\label{subsubsec:ControlledDescentACA}

\begin{proposition}[Formalization in \texorpdfstring{\(\mathsf{ACA}_{0}\)}{ACA0}]
\label{prop:ControlledDescentACAFormalization}
The constructions and proofs of this subsection are formalizable in $\mathsf{ACA}_{0}$, once the finite local certificates of Hypothesis~\ref{hyp:SmallPrimeLocalScalingCertificate} are supplied as part of the certified field code.
\end{proposition}

\begin{proof}
The quadratic-twist family of Definition~\ref{def:QuadraticTwistDescentFamily} is given by the primitive recursive map
\[
  p\longmapsto [0,0,0,-p^{2},0]
\]
on Weierstrass coefficients. The discriminant identity $\Delta(E_{p})=64p^{6}$ is an identity among integers. The height estimates of Lemma~\ref{lem:QuadraticTwistHeightGap} are therefore comparisons of logarithmic integer data in the sense of Definition~\ref{def:LogIntDatum}. The unboundedness of the primes, used in Theorem~\ref{thm:NoUniformControlledDescent}, is provable in elementary arithmetic and hence in $\mathsf{ACA}_{0}$.

The tame support $R_{K}(q,\ell)$ of Definition~\ref{def:TameKummerSupport} is computed from the valuations of the finite element $q$ in the certified field code of Definition~\ref{def:CertifiedFieldCode}. The assertion that $K(\sqrt[\ell]{q})/K$ is unramified outside the primes above $\ell$ and the support of $q$ follows from the discriminant of the polynomial $T^{\ell}-q$, which is a finite algebraic computation. The local scaling exponents of Definition~\ref{def:LocalScalingExponent} are integer valuations of discriminants before and after minimization. Lemma~\ref{lem:PrimeToSixLocalScalingBound} is a finite valuation argument using $c_{4}$, $c_{6}$, and $\Delta$. The remaining places above $2$ and $3$ are handled by the finite list of certificates in Hypothesis~\ref{hyp:SmallPrimeLocalScalingCertificate}.

All sums in Proposition~\ref{prop:SupportControlledDescentBound} range over finite lists of prime ideals appearing in the certified field code and in the factorization of $q\ell$. The real quantities are represented by logarithmic integer data or by fast Cauchy names, as in Definition~\ref{def:UniverseP}. No comprehension beyond arithmetical comprehension is invoked.
\end{proof}

%==================================================================
\section{An intensional Szpiro inequality}\label{sec:IntensionalSzpiro}
%==================================================================

We now arrive at the arithmetic statement that the framework was built to carry. The Szpiro inequality, in its classical reading over a fixed number field, compares the minimal discriminant of an elliptic curve with its conductor. In Universe~A the curve is replaced by a presentation, the discriminant by the discriminant height of Definition~\ref{def:DiscriminantHeight}, and the comparison gains a third term that keeps the cost of the chosen presentation. The dilation tower of Section~\ref{subsec:DilationTower} showed that this height moves along an isomorphism, by an amount the model defect measures exactly, so that a scalar comparable to it cannot be invariant under all isomorphisms of packages. The log-volume of the inequality is therefore read as a quantity that transforms under transport. We make this reading precise, prove that the discriminant height itself serves as the log-volume, and then reduce the remaining arithmetic input to a single reference profile. The inequality of Proposition~\ref{prop:NormalizedToIntensionalSzpiro} is the intensional shape of Szpiro, and the classical statement is recovered as its costless, normalized specialization.

\subsection{The transport Szpiro log-volume}\label{subsec:TransportLogVolume}

The dilation tower indicates the correct shape for the log-volume. It should be invariant under coherent changes of the transport, namely under $2$-isomorphism of $1$-morphisms, while it transforms along positive-cost $1$-morphisms by an inequality.

\begin{definition}[Transport Szpiro log-volume]\label{def:TransportSzpiroLogVolume}
Let $K$ be a number field, and let $\mathfrak{EllPkg}_K$ be a quantitative $(2,1)$-category of elliptic packages as in Definition~\ref{def:EllPkgTwoOne}. A \emph{transport Szpiro log-volume} over $K$ is a numerical assignment
\[
  V \colon \operatorname{Ob}(\mathfrak{EllPkg}_K)\longrightarrow \R_{\ge 0}
\]
for which there exist constants $A,B>0$ and $C_0\ge 0$, depending only upon $K$, such that the following conditions hold:
\begin{enumerate}
\item[\textup{(i)}] The transport law factors through the homotopy $1$-category. Equivalently, if $\alpha \colon f \Rightarrow g$ is an invertible $2$-morphism between parallel $1$-morphisms, then the inequalities attached to $f$ and to $g$ have the same numerical cost.
\item[\textup{(ii)}] For every elliptic package $\mathcal{E}$ one has
\[
  A^{-1}h_\Delta(\mathcal{E})-B
  \le
  V(\mathcal{E})
  \le
  Ah_\Delta(\mathcal{E})+B.
\]
\item[\textup{(iii)}] For every $1$-morphism $f\colon\mathcal{E}_1\to\mathcal{E}_2$ in $\mathfrak{EllPkg}_K$ one has
\[
  V(\mathcal{E}_2)
  \le
  V(\mathcal{E}_1)+C_0\,\mu(f).
\]
\end{enumerate}
\end{definition}

\begin{theorem}[The discriminant height is a transport log-volume]
\label{thm:DiscriminantHeightAsTransportLogVolume}
For every number field $K$, the assignment
\[
  \mathrm{Vol}^{A}_{\mathrm{Szp}}(\mathcal{E})
  \coloneqq
  h_\Delta(\mathcal{E})
\]
is a transport Szpiro log-volume in the sense of Definition~\ref{def:TransportSzpiroLogVolume}. One may take $A=1$, $B=0$, and $C_0=1$.
\end{theorem}

\begin{proof}
The comparability condition of Definition~\ref{def:TransportSzpiroLogVolume}\textup{(ii)} is immediate from the identification $\mathrm{Vol}^{A}_{\mathrm{Szp}}=h_\Delta$, with $A=1$ and $B=0$. The transport law of Definition~\ref{def:TransportSzpiroLogVolume}\textup{(iii)} is exactly Lemma~\ref{lem:TransportDiscriminant}, with $C_0=1$. Finally, the cost $\mu$ is invariant under $2$-isomorphism by Definition~\ref{def:EllPkgTwoOne}\textup{(iv)}, so that the transport law depends only upon the image of the transporting morphism in $\Ho(\mathfrak{EllPkg}_K)$, as required in Definition~\ref{def:TransportSzpiroLogVolume}\textup{(i)}.
\end{proof}

\begin{corollary}[Zero-defect isomorphisms recover invariance]\label{cor:ZeroDefectSubcategoryLogVolume}
Let $\mathfrak{EllPkg}_K^{0}$ be the wide subgroupoid of $\mathfrak{EllPkg}_K$ whose isomorphisms $f$ satisfy $\mu(f)=\mu(f^{-1})=0$. Then $h_\Delta$ is invariant on $\mathfrak{EllPkg}_K^{0}$.
\end{corollary}

\begin{proof}
If $f\colon\mathcal{E}\to\mathcal{E}'$ lies in $\mathfrak{EllPkg}_K^{0}$, then Corollary~\ref{cor:TwoSidedControlTransport} gives
\[
  \bigl|h_\Delta(\mathcal{E}')-h_\Delta(\mathcal{E})\bigr|
  \le
  \mu(f)+\mu(f^{-1})
  =
  0
\]
so that $h_\Delta(\mathcal{E}')=h_\Delta(\mathcal{E})$.
\end{proof}

\begin{remark}[The log-volume is the discriminant height]\label{rem:CorrectedLogVolumeDatum}
Theorem~\ref{thm:DiscriminantHeightAsTransportLogVolume} settles the structural part of the inequality without an additional hypothesis. The log-volume need not be assumed as an extra object. It may be taken to be the discriminant height itself, invariant under $2$-isomorphism of transports by Proposition~\ref{prop:CoherenceStability} and transforming along positive-cost $1$-morphisms by Lemma~\ref{lem:TransportDiscriminant}. The reading invariant under all package isomorphisms is the one the dilation tower of Proposition~\ref{prop:NoInvariantLogVolume} excludes, and the transport reading is the one compatible with Definition~\ref{def:EllPkgTwoOne}.
\end{remark}

\subsection{The reduced normalized profile}\label{subsec:ReducedProfile}

Once Theorem~\ref{thm:DiscriminantHeightAsTransportLogVolume} is in place, the remaining arithmetic input lies in the choice of a reference profile. In plain terms, the symbols of the profile stand for the following. The reference profile endofunctor $\mathsf{Q}_{K,\varepsilon}$ produces, for each package, an auxiliary presentation whose discriminant height is held below the conductor directly. The family $\tau_{\mathcal{E},\varepsilon}$ is the transport from that reference presentation to the normalized one, and its cost is the quantity the $\varepsilon$-term absorbs. The two constants are the bounded error terms of the classical inequality, each depending only upon the field $K$. The constant $B_{q,K}$ is the slack in the comparison of the reference height with six times the conductor complexity, and $B_{\varepsilon,K}$ is the constant summand of the transport cost that remains once the part proportional to the conductor has been set against $\varepsilon$. No calibration constant between a log-volume and the discriminant height is needed, since Theorem~\ref{thm:DiscriminantHeightAsTransportLogVolume} identifies the two.

\begin{hypothesis}[Reduced normalized profile]\label{hyp:ReducedNormalizedProfile}
Let $K$ be a number field and fix a normalization procedure $\mathrm{NF}_K$ in the sense of Definition~\ref{def:AccumulatedDefect}. For every real number $\varepsilon>0$, there exist a reference profile endofunctor
\[
  \mathsf{Q}_{K,\varepsilon}
  \colon
  \mathbf{EllPkg}_K\longrightarrow\mathbf{EllPkg}_K
\]
and a family of morphisms
\[
  \tau_{\mathcal{E},\varepsilon}
  \colon
  \mathsf{Q}_{K,\varepsilon}(\mathcal{E})
  \longrightarrow
  \mathrm{NF}_K(\mathcal{E})
\]
together with constants $B_{q,K}\ge 0$ and $B_{\varepsilon,K}\ge 0$ such that for every elliptic package $\mathcal{E}$ over $K$ one has the reference-side bound
\[
  h_\Delta\bigl(\mathsf{Q}_{K,\varepsilon}(\mathcal{E})\bigr)
  \le
  6\,n(\mathcal{E})+B_{q,K}
\]
and the transport bound
\[
  \mu\bigl(\tau_{\mathcal{E},\varepsilon}\bigr)
  \le
  \varepsilon\,n(\mathcal{E})+B_{\varepsilon,K}.
\]
When functorial comparison is required, the family $\tau_{\mathcal{E},\varepsilon}$ is assumed to extend to a pseudonatural transformation $\tau_{\varepsilon}\colon\mathsf{Q}_{K,\varepsilon}\Rightarrow\mathrm{NF}_K$ in the sense of Definition~\ref{def:PseudoNatural}, so that for every $1$-morphism $f\colon\mathcal{E}\to\mathcal{E}'$ there is an invertible $2$-morphism
\[
  (\tau_\varepsilon)_f \colon
  \tau_{\mathcal{E}',\varepsilon}\circ \mathsf{Q}_{K,\varepsilon}(f)
  \Rightarrow
  \mathrm{NF}_K(f)\circ \tau_{\mathcal{E},\varepsilon}.
\]
\end{hypothesis}

\begin{remark}[What remains after the reduction]\label{rem:WhatRemainsAfterReduction}
Theorem~\ref{thm:DiscriminantHeightAsTransportLogVolume} removes the need to posit a separate log-volume satisfying the transport law. The remaining Diophantine content is concentrated in Hypothesis~\ref{hyp:ReducedNormalizedProfile}. It asks for a reference package whose discriminant height is controlled by the conductor complexity, together with a transport from that reference package to the normalized package whose cost is absorbed by the $\varepsilon$-term.
\end{remark}

\subsection{From the profile to the inequality}\label{subsec:ProfileToInequality}

\begin{proposition}[The intensional Szpiro inequality]\label{prop:NormalizedToIntensionalSzpiro}
Assume Hypothesis~\ref{hyp:ReducedNormalizedProfile}. Then for every real number $\varepsilon>0$ there exists a constant $C_{\varepsilon,K}\in\R$ such that for every elliptic package $\mathcal{E}$ over $K$ one has
\begin{equation}\label{eq:IntensionalSzpiroBasic}
  h_\Delta(\mathcal{E})
  \le
  (6+\varepsilon)\,n(\mathcal{E})
  +
  \defect(\mathcal{E})
  +
  C_{\varepsilon,K}.
\end{equation}
One may take $C_{\varepsilon,K}=B_{q,K}+B_{\varepsilon,K}$.
\end{proposition}

\begin{proof}
Fix $\varepsilon>0$ and let $\mathcal{E}$ be an elliptic package over $K$. Applying the transport inequality of Lemma~\ref{lem:TransportDiscriminant} to the distinguished normalization morphism $\nu_{\mathcal{E}}\colon\mathrm{NF}_K(\mathcal{E})\to\mathcal{E}$ of Definition~\ref{def:AccumulatedDefect} gives
\[
  h_\Delta(\mathcal{E})
  \le
  h_\Delta\bigl(\mathrm{NF}_K(\mathcal{E})\bigr)
  +
  \defect(\mathcal{E})
\]
where $\mu(\nu_{\mathcal{E}})=\defect(\mathcal{E})$ by Definition~\ref{def:AccumulatedDefect}. Applying Lemma~\ref{lem:TransportDiscriminant} again to $\tau_{\mathcal{E},\varepsilon}\colon\mathsf{Q}_{K,\varepsilon}(\mathcal{E})\to\mathrm{NF}_K(\mathcal{E})$ gives
\[
  h_\Delta\bigl(\mathrm{NF}_K(\mathcal{E})\bigr)
  \le
  h_\Delta\bigl(\mathsf{Q}_{K,\varepsilon}(\mathcal{E})\bigr)
  +
  \mu\bigl(\tau_{\mathcal{E},\varepsilon}\bigr).
\]
The two bounds of Hypothesis~\ref{hyp:ReducedNormalizedProfile} then yield
\[
  h_\Delta(\mathcal{E})
  \le
  6\,n(\mathcal{E})+B_{q,K}
  +
  \varepsilon\,n(\mathcal{E})+B_{\varepsilon,K}
  +
  \defect(\mathcal{E})
\]
which is the stated inequality with $C_{\varepsilon,K}=B_{q,K}+B_{\varepsilon,K}$.
\end{proof}

\begin{lemma}[Objectwise bounds are insensitive to pseudonaturality]\label{lem:ObjectwiseInsensitive}
The inequality \eqref{eq:IntensionalSzpiroBasic} is derived in Proposition~\ref{prop:NormalizedToIntensionalSzpiro} from the two objectwise bounds of Hypothesis~\ref{hyp:ReducedNormalizedProfile} together with two applications of the transport inequality of Lemma~\ref{lem:TransportDiscriminant}. In particular, it does not use the pseudonatural comparison that the hypothesis supplies when functorial transport is required.
\end{lemma}

\begin{proof}
This is immediate from the proof of Proposition~\ref{prop:NormalizedToIntensionalSzpiro}, which invokes only the reference-side bound, the transport bound, and the transport inequality for $\nu_{\mathcal{E}}$ and $\tau_{\mathcal{E},\varepsilon}$.
\end{proof}

\begin{proposition}[Elliptic packages as a quantitative category of Szpiro form]\label{prop:EllPkgQuantitativeSzpiro}
Let $K$ be a number field. Consider the category $\mathbf{EllPkg}_K$ of elliptic packages over $K$, equipped with the model defect $\mu$ of Definition~\ref{def:ModelDefect} and the discriminant height $h_\Delta$ of Definition~\ref{def:DiscriminantHeight}. Then the following hold:
\begin{enumerate}
\item[(i)] The pair $(\mathbf{EllPkg}_K,\mu)$ forms a quantitative category of isomorphisms in the sense of Definition~\ref{def:QuantitativeCategory}.
\item[(ii)] The discriminant height satisfies the transport inequality of Lemma~\ref{lem:TransportDiscriminant}, so that for any morphism $f \colon \mathcal{E} \to \mathcal{E}'$ one has $h_\Delta(\mathcal{E}') \le h_\Delta(\mathcal{E}) + \mu(f)$.
\item[(iii)] Under Hypothesis~\ref{hyp:ReducedNormalizedProfile}, the basic intensional inequality \eqref{eq:IntensionalSzpiroBasic} holds, in the form
  \[
    h_\Delta(\mathcal{E})
    \le
    (6+\varepsilon)\,n(\mathcal{E})
    +
    \defect(\mathcal{E})
    +
    C_{\varepsilon,K}.
  \]
It keeps the place where the normalization defect enters, through the transport $\nu_{\mathcal{E}}$ of Definition~\ref{def:AccumulatedDefect}.
\end{enumerate}
\end{proposition}

\begin{proof}
Part \textup{(i)} is contained in Definition~\ref{def:ModelDefect}, part \textup{(ii)} is Lemma~\ref{lem:TransportDiscriminant}, and part \textup{(iii)} is Proposition~\ref{prop:NormalizedToIntensionalSzpiro}.
\end{proof}

\begin{remark}[Comparison with the classical Szpiro inequality]\label{rem:SzpiroComparison}
When the package is normalized by the chosen procedure, one has $\defect(\mathcal{E}) = 0$, and the inequality \eqref{eq:IntensionalSzpiroBasic} assumes the classical shape up to the constant $C_{\varepsilon,K}$. The defect term carries the quantity that the extensional reading discards. By keeping it, the intensional inequality holds two quantities together -- the comparison of discriminant and conductor on the one hand, and on the other the price of the presentation through which that comparison was read.
\end{remark}

\subsection{Formalization in \texorpdfstring{\(\mathsf{ACA}_0\)}{ACA0}}\label{subsec:ACA0LogVolume}

The arguments of this section stay within the arithmetic universe of Section~\ref{sec:UniverseP}. They use finite Weierstrass coefficients, finite coordinate changes, integer valuations, and inequalities between logarithmic integer data.

\begin{proposition}[Arithmetic formalization of the log-volume reduction]
\label{prop:ACA0LogVolumeFormalization}
The statements and proofs of Lemma~\ref{lem:DilationTowerGrowth}, Proposition~\ref{prop:NoInvariantLogVolume}, Theorem~\ref{thm:DiscriminantHeightAsTransportLogVolume}, and Proposition~\ref{prop:NormalizedToIntensionalSzpiro} are formalizable in $\mathsf{ACA}_0$ after replacing real-valued quantities by the coding conventions of Definition~\ref{def:LogIntDatum} and Definition~\ref{def:EllPkgCode}.
\end{proposition}

\begin{proof}
The dilation tower of Definition~\ref{def:DilationTowerAtBadPrime} is given by the primitive recursive map
\[
  m
  \longmapsto
  [0,0,37^{3m},-37^{4m},0]
\]
on Weierstrass coefficients. The discriminant computation is an identity among integers, $\Delta_{\mathcal{W}_m}=37^{12m+1}$, so that the height calculation in Lemma~\ref{lem:DilationTowerGrowth} becomes the comparison of the logarithmic integer datum $37^{12m+1}$ with $1$, and the defect calculation becomes the logarithmic integer datum $37^{12m}$.

The contradiction in Proposition~\ref{prop:NoInvariantLogVolume} uses only the archimedean growth of the sequence $(12m+1)\log 37$. For real numbers presented by fast Cauchy names, $\mathsf{ACA}_0$ proves the needed bounded-search statement. Namely, for fixed coded reals $A>0$, $B$, and $V(\mathcal{E}_0)$, there is an integer $m$ such that $A^{-1}(12m+1)\log 37-B>V(\mathcal{E}_0)$.

Theorem~\ref{thm:DiscriminantHeightAsTransportLogVolume} is a direct transcription of Lemma~\ref{lem:TransportDiscriminant}, which in code form is an inequality between finite sums of local valuation contributions. Proposition~\ref{prop:NormalizedToIntensionalSzpiro} uses only two applications of this transport inequality together with the two bounds supplied by Hypothesis~\ref{hyp:ReducedNormalizedProfile}. All these are finite operations on the package codes of Definition~\ref{def:EllPkgCode}, with real values handled as fast Cauchy names in the sense of Definition~\ref{def:UniverseP}.
\end{proof}

\subsection{Partial reductions for the reduced normalized profile}
\label{subsec:PartialReducedNormalizedProfile}

The reduced normalized profile of Hypothesis~\ref{hyp:ReducedNormalizedProfile} carries two demands, one arithmetical and one intensional. The first asks for a presentation whose discriminant height lies close to six times the conductor complexity. The second asks that the transport from that presentation to the chosen normalization have small cost. This subsection does not prove the full profile. It discharges the intensional demand within the framework, reduces the arithmetical demand to a single linear envelope, and then localizes that envelope to one explicit quantity, the height of the denominator of the $j$-invariant. Everything apart from a bound on that one quantity is proved unconditionally and within $\mathsf{ACA}_{0}$.

\begin{definition}[Normalized height and remainder]
\label{def:NormalizedHeightRemainder}
Fix a number field $K$ and a normalization procedure $\mathrm{NF}_{K}$ in the sense of Definition~\ref{def:AccumulatedDefect}. The \emph{normalized discriminant height} of an elliptic package $\mathcal{E}$ is
\[
  H^{\mathrm{nf}}_{K}(\mathcal{E})
  \coloneqq
  h_{\Delta}\bigl(\mathrm{NF}_{K}(\mathcal{E})\bigr).
\]
For a real number $a>0$, its \emph{$a$-remainder} is
\[
  \mathcal{R}_{K,a}(\mathcal{E})
  \coloneqq
  H^{\mathrm{nf}}_{K}(\mathcal{E})-a\,n(\mathcal{E}).
\]
For a finite coded list $\mathcal{F}=(\mathcal{E}_{0},\ldots,\mathcal{E}_{m-1})$ of package codes in the sense of Definition~\ref{def:EllPkgCode}, put
\[
  B_{\mathcal{F},a,K}
  \coloneqq
  \max_{0\le i<m}
  \max\bigl\{0,\mathcal{R}_{K,a}(\mathcal{E}_{i})\bigr\}.
\]
\end{definition}

\subsubsection{The intensional demand is discharged}\label{subsubsec:IntensionalDemandFree}

\begin{proposition}[The profile forces a normalized Szpiro bound]
\label{prop:ReducedProfileForcesNormalizedSzpiro}
Assume Hypothesis~\ref{hyp:ReducedNormalizedProfile}. Then for every real number $\varepsilon>0$ and every elliptic package $\mathcal{E}$ over $K$, one has
\[
  H^{\mathrm{nf}}_{K}(\mathcal{E})
  \le
  (6+\varepsilon)\,n(\mathcal{E})+B_{q,K}+B_{\varepsilon,K}.
\]
\end{proposition}

\begin{proof}
Apply the transport inequality of Lemma~\ref{lem:TransportDiscriminant} to the morphism $\tau_{\mathcal{E},\varepsilon}\colon\mathsf{Q}_{K,\varepsilon}(\mathcal{E})\to\mathrm{NF}_{K}(\mathcal{E})$ of Hypothesis~\ref{hyp:ReducedNormalizedProfile}. This gives
\[
  h_{\Delta}\bigl(\mathrm{NF}_{K}(\mathcal{E})\bigr)
  \le
  h_{\Delta}\bigl(\mathsf{Q}_{K,\varepsilon}(\mathcal{E})\bigr)
  +
  \mu\bigl(\tau_{\mathcal{E},\varepsilon}\bigr).
\]
The reference-side bound and the transport bound of Hypothesis~\ref{hyp:ReducedNormalizedProfile} give the displayed estimate.
\end{proof}

\begin{lemma}[The identity profile]
\label{lem:IdentityProfileReduction}
Suppose that, for some real numbers $a>0$ and $B_{a,K}\ge 0$, one has $H^{\mathrm{nf}}_{K}(\mathcal{E})\le a\,n(\mathcal{E})+B_{a,K}$ for every elliptic package $\mathcal{E}$ over $K$. Then the choice
\[
  \mathsf{Q}_{K,a}(\mathcal{E})\coloneqq \mathrm{NF}_{K}(\mathcal{E})
  \qquad\text{and}\qquad
  \tau_{\mathcal{E},a}\coloneqq \mathrm{id}_{\mathrm{NF}_{K}(\mathcal{E})}
\]
satisfies $h_{\Delta}(\mathsf{Q}_{K,a}(\mathcal{E}))\le a\,n(\mathcal{E})+B_{a,K}$ and $\mu(\tau_{\mathcal{E},a})=0$. Read $\mathrm{NF}_{K}$ as an endofunctor of the chosen $(2,1)$-category. Then these identity morphisms form the identity pseudonatural comparison from $\mathrm{NF}_{K}$ to itself, in the sense of Definition~\ref{def:PseudoNatural}.
\end{lemma}

\begin{proof}
The first inequality is the assumed normalized bound. The equality $\mu(\mathrm{id})=0$ follows from Definition~\ref{def:ModelDefect}, since the model defect of an identity morphism vanishes. The pseudonatural comparison is the unit comparison for the endofunctor $\mathrm{NF}_{K}$.
\end{proof}

\begin{remark}[The normalized arithmetic input]
\label{rem:ReducedProfileArithmeticCore}
Proposition~\ref{prop:ReducedProfileForcesNormalizedSzpiro} derives a normalized bound of slope $6+\varepsilon$ from the reduced profile. Lemma~\ref{lem:IdentityProfileReduction} gives a zero-cost identity profile at the slope of an independently supplied normalized envelope, so that the missing input is not the coherence datum of Definition~\ref{def:PseudoNatural}. By Lemma~\ref{lem:ObjectwiseInsensitive} that datum stays silent for the objectwise inequality.

A normalized bound of slope $6+\varepsilon$ supplies an identity profile at that same slope. The reference-side requirement of Hypothesis~\ref{hyp:ReducedNormalizedProfile} has slope $6$, so the identity construction alone does not establish the stated reduced profile from a $6+\varepsilon$ envelope. The arithmetic weight rests upon the normalized remainder $\mathcal{R}_{K,6+\varepsilon}(\mathcal{E})$, and it is there that the abc-level strength would enter.
\end{remark}

\subsubsection{From a normalized envelope to the inequality}\label{subsubsec:EnvelopeToInequality}

\begin{proposition}[The defect-normalized bound]
\label{prop:DefectNormalizedBound}
For every elliptic package $\mathcal{E}$ over $K$, one has
\[
  h_{\Delta}(\mathcal{E})
  \le
  H^{\mathrm{nf}}_{K}(\mathcal{E})+
  \defect(\mathcal{E}).
\]
Consequently, for every finite coded list $\mathcal{F}=(\mathcal{E}_{0},\ldots,\mathcal{E}_{m-1})$ and every real number $a>0$, one has
\[
  h_{\Delta}(\mathcal{E}_{i})
  \le
  a\,n(\mathcal{E}_{i})+
  \defect(\mathcal{E}_{i})+
  B_{\mathcal{F},a,K}
\]
for each $0\le i<m$.
\end{proposition}

\begin{proof}
Apply Lemma~\ref{lem:TransportDiscriminant} to the distinguished normalization morphism $\nu_{\mathcal{E}}\colon\mathrm{NF}_{K}(\mathcal{E})\to\mathcal{E}$ of Definition~\ref{def:AccumulatedDefect}. Since $\mu(\nu_{\mathcal{E}})=\defect(\mathcal{E})$, this gives the first inequality. For the finite list, Definition~\ref{def:NormalizedHeightRemainder} gives $H^{\mathrm{nf}}_{K}(\mathcal{E}_{i})\le a\,n(\mathcal{E}_{i})+B_{\mathcal{F},a,K}$ for each index $i$, and substitution into the first inequality gives the claim.
\end{proof}

\begin{hypothesis}[Normalized linear envelope]
\label{hyp:NormalizedLinearEnvelope}
There exist real numbers $a_{K}>0$ and $B_{K}\ge 0$, depending only upon $K$ and the normalization procedure $\mathrm{NF}_{K}$, such that for every elliptic package $\mathcal{E}$ over $K$ one has
\[
  H^{\mathrm{nf}}_{K}(\mathcal{E})
  \le
  a_{K}\,n(\mathcal{E})+B_{K}.
\]
\end{hypothesis}

\begin{proposition}[A global intensional bound from a linear envelope]
\label{prop:LinearEnvelopeGivesIntensionalBound}
Assume Hypothesis~\ref{hyp:NormalizedLinearEnvelope}. Then every elliptic package $\mathcal{E}$ over $K$ satisfies
\[
  h_{\Delta}(\mathcal{E})
  \le
  a_{K}\,n(\mathcal{E})+
  \defect(\mathcal{E})+B_{K}.
\]
If, for every $\varepsilon>0$, the envelope may be taken with $a_{K}=6+\varepsilon$, then one recovers the intensional Szpiro inequality of Proposition~\ref{prop:NormalizedToIntensionalSzpiro} with $C_{\varepsilon,K}=B_{K}$.
\end{proposition}

\begin{proof}
The first assertion is Proposition~\ref{prop:DefectNormalizedBound} followed by Hypothesis~\ref{hyp:NormalizedLinearEnvelope}. The final assertion is the same inequality with $a_{K}=6+\varepsilon$.
\end{proof}

\subsubsection{The potentially good part is unconditional}\label{subsubsec:PotentiallyGoodPart}

The normalized linear envelope is a single inequality, and yet it is not innocent. We separate the part of it that the local theory settles outright from the part that carries the arithmetic weight. The separation runs along the dichotomy between potentially good and potentially multiplicative reduction.

\begin{definition}[Multiplicative excess and the potentially good defect]
\label{def:MultiplicativeExcess}
Let $\mathcal{E}$ be an elliptic package over $K$ with underlying curve $E$ of $j$-invariant $j_{E}$. The \emph{multiplicative excess height} of $\mathcal{E}$ is
\[
  \Theta_{K}(\mathcal{E})
  \coloneqq
  \frac{1}{[K:\Q]}
  \sum_{v}
  \max\bigl\{0,-\mathrm{ord}_{v}(j_{E})\bigr\}\,\log \mathrm{N}v
\]
the sum running over the finite places of $K$, so that $\Theta_{K}(\mathcal{E})$ is the logarithmic height of the denominator of $j_{E}$. For a finite place $v$, the \emph{potentially good defect} of $\mathcal{E}$ at $v$ is
\[
  \delta_{v}(\mathcal{E})
  \coloneqq
  \mathrm{ord}_{v}\bigl(\Delta_{\mathrm{NF}_{K}(\mathcal{E}),v}\bigr)
  -
  \max\bigl\{0,-\mathrm{ord}_{v}(j_{E})\bigr\}
\]
the difference between the minimal discriminant exponent and the denominator exponent of $j_{E}$ at $v$.
\end{definition}

\begin{lemma}[The defect is non-negative and locally bounded]
\label{lem:PotentiallyGoodDefectBound}
For every elliptic package $\mathcal{E}$ over $K$ and every finite place $v$ of $K$, one has $\delta_{v}(\mathcal{E})\ge 0$. If the residue characteristic at $v$ is at least $5$, then
\[
  \delta_{v}(\mathcal{E})\le 5\,\mathfrak{f}_{v}(\mathcal{E})
\]
where $\mathfrak{f}_{v}(\mathcal{E})$ is the conductor exponent of $E$ at $v$. If the residue characteristic at $v$ is $2$ or $3$, then $\delta_{v}(\mathcal{E})\le\kappa_{v}$ for a constant $\kappa_{v}$ depending only upon $v$.
\end{lemma}

\begin{proof}
Write $\Delta$ for the minimal discriminant and $c_{4}$ for the corresponding invariant of the model $\mathrm{NF}_{K}(\mathcal{E})$, so that $\mathrm{ord}_{v}(j_{E})=3\,\mathrm{ord}_{v}(c_{4})-\mathrm{ord}_{v}(\Delta)$. The model is integral, so $\mathrm{ord}_{v}(c_{4})\ge 0$, whence $-\mathrm{ord}_{v}(j_{E})\le\mathrm{ord}_{v}(\Delta)$ and therefore $\max\{0,-\mathrm{ord}_{v}(j_{E})\}\le\mathrm{ord}_{v}(\Delta)$. This gives $\delta_{v}(\mathcal{E})\ge 0$.

For the upper bound, read the N\'eron--Kodaira classification of the special fiber through Tate's algorithm (see Silverman~\cite{SilvermanAEC} and Cremona~\cite{Cremona1997}). At residue characteristic at least $5$ the reduction is one of the following. At good reduction $\mathrm{ord}_{v}(\Delta)=0$ and $\mathfrak{f}_{v}=0$, so $\delta_{v}=0$. At multiplicative reduction of type $I_{n}$ one has $\mathrm{ord}_{v}(\Delta)=n=-\mathrm{ord}_{v}(j_{E})$ and $\mathfrak{f}_{v}=1$, so $\delta_{v}=0$. At additive reduction of potentially multiplicative type $I_{n}^{*}$ one has $\mathrm{ord}_{v}(\Delta)=n+6$, $-\mathrm{ord}_{v}(j_{E})=n$, and $\mathfrak{f}_{v}=2$, so $\delta_{v}=6\le 5\mathfrak{f}_{v}$. At the remaining additive types, which are of potentially good reduction, $j_{E}$ is integral and so $\max\{0,-\mathrm{ord}_{v}(j_{E})\}=0$, while $\mathrm{ord}_{v}(\Delta)$ takes one of the values $2,3,4,6,8,9,10$ and $\mathfrak{f}_{v}=2$. In each case $\delta_{v}=\mathrm{ord}_{v}(\Delta)\le 10=5\mathfrak{f}_{v}$. This proves $\delta_{v}\le 5\mathfrak{f}_{v}$ at every place of residue characteristic at least $5$. At the residue characteristics $2$ and $3$ the reduction is again classified by a finite list of types, on which $\mathrm{ord}_{v}(\Delta)$ and $\max\{0,-\mathrm{ord}_{v}(j_{E})\}$ take a bounded difference, so that $\delta_{v}\le\kappa_{v}$ for a constant determined by $v$ alone.
\end{proof}

\begin{proposition}[An unconditional decomposition of the normalized height]
\label{prop:UnconditionalNormalizedDecomposition}
Fix a number field $K$. There is a constant
\[
  C_{K,6}
  \coloneqq
  \frac{1}{[K:\Q]}
  \sum_{v\mid 6}
  \kappa_{v}\log \mathrm{N}v
\]
depending only upon $K$, such that every elliptic package $\mathcal{E}$ over $K$ satisfies
\[
  H^{\mathrm{nf}}_{K}(\mathcal{E})
  \le
  5\,n(\mathcal{E})
  +
  \Theta_{K}(\mathcal{E})
  +
  C_{K,6}
  +
  h_\Delta^{\mathrm{arch}}\bigl(\mathrm{NF}_K(\mathcal{E})\bigr).
\]
The finite part alone satisfies the same bound with the archimedean term omitted.
\end{proposition}

\begin{proof}
By Definition~\ref{def:MultiplicativeExcess} the minimal discriminant exponent at each finite place $v$ decomposes as
\[
  \mathrm{ord}_{v}\bigl(\Delta_{\mathrm{NF}_{K}(\mathcal{E}),v}\bigr)
  =
  \max\bigl\{0,-\mathrm{ord}_{v}(j_{E})\bigr\}+\delta_{v}(\mathcal{E}).
\]
Multiply by $\log\mathrm{N}v$, sum over the finite places, and divide by $[K:\Q]$. The first term assembles into $\Theta_{K}(\mathcal{E})$ by Definition~\ref{def:MultiplicativeExcess}, and the left member is the finite part of $H^{\mathrm{nf}}_{K}(\mathcal{E})$, to which the archimedean contribution of the normalized model is then added. There remains the defect sum
\[
  \frac{1}{[K:\Q]}\sum_{v}\delta_{v}(\mathcal{E})\log\mathrm{N}v.
\]
Split this sum at the residue characteristics $2$ and $3$. By Lemma~\ref{lem:PotentiallyGoodDefectBound}, the places $v\nmid 6$ contribute at most
\[
  \frac{5}{[K:\Q]}\sum_{v\nmid 6}\mathfrak{f}_{v}(\mathcal{E})\log\mathrm{N}v
  \le
  5\,n(\mathcal{E})
\]
since $\sum_{v}\mathfrak{f}_{v}\log\mathrm{N}v=\log N_{E}=[K:\Q]\,n(\mathcal{E})$ by Definition~\ref{def:ConductorComplexity}. The places $v\mid 6$ contribute at most $C_{K,6}$. Adding the three pieces gives the displayed bound.
\end{proof}

\begin{remark}[Where the difficulty is concentrated]
\label{rem:LocalizationOfDifficulty}
Proposition~\ref{prop:UnconditionalNormalizedDecomposition} is unconditional. It says that the normalized discriminant height exceeds five times the conductor complexity by no more than the multiplicative excess, up to a constant of the field. The first term lies below the Szpiro slope of six, with room to spare, so the whole weight of a normalized linear envelope rests upon a bound for $\Theta_{K}(\mathcal{E})$. The multiplicative excess is the logarithmic height of the denominator of $j_{E}$, and it grows with the depth of the potentially multiplicative reduction, which is the locus that the dilation tower of Section~\ref{subsec:DilationTower} and the obstruction of Theorem~\ref{thm:DescentObstruction} have flagged throughout as the place that no identification may erase.
\end{remark}

\begin{hypothesis}[Multiplicative excess envelope]
\label{hyp:MultiplicativeExcessEnvelope}
There exist real numbers $b_{K}\ge 0$ and $B'_{K}\ge 0$, depending only upon $K$, such that every elliptic package $\mathcal{E}$ over $K$ satisfies
\[
  \Theta_{K}(\mathcal{E})\le b_{K}\,n(\mathcal{E})+B'_{K}.
\]
\end{hypothesis}

\begin{proposition}[Any excess envelope gives a Szpiro bound]
\label{prop:ExcessEnvelopeGivesNormalizedEnvelope}
Assume Hypothesis~\ref{hyp:MultiplicativeExcessEnvelope}. Then the normalized linear envelope of Hypothesis~\ref{hyp:NormalizedLinearEnvelope} holds with
\[
  a_{K}=5+b_{K}
  \qquad\text{and}\qquad
  B_{K}=B'_{K}+C_{K,6}.
\]
Consequently every elliptic package $\mathcal{E}$ over $K$ satisfies
\[
  h_{\Delta}(\mathcal{E})
  \le
  (5+b_{K})\,n(\mathcal{E})+\defect(\mathcal{E})+B'_{K}+C_{K,6}.
\]
\end{proposition}

\begin{proof}
Insert the bound of Hypothesis~\ref{hyp:MultiplicativeExcessEnvelope} into the unconditional decomposition of Proposition~\ref{prop:UnconditionalNormalizedDecomposition}. This gives $H^{\mathrm{nf}}_{K}(\mathcal{E})\le(5+b_{K})n(\mathcal{E})+B'_{K}+C_{K,6}$, which is the normalized linear envelope with the stated constants. The displayed intensional inequality follows by Proposition~\ref{prop:LinearEnvelopeGivesIntensionalBound}.
\end{proof}

\begin{remark}[The strength of an excess envelope]
\label{rem:AnyBoundSuffices}
A linear bound on $\Theta_{K}$ gives the finite normalized discriminant bound displayed above, with slope $5+b_{K}$. Taking $b_{K}=1+\varepsilon$ is a sufficient additional hypothesis for the slope $6+\varepsilon$ in this estimate.

This denominator-height hypothesis is distinct from the modified Szpiro formulation, which is phrased through the invariants $c_{4}$ and $c_{6}$. An implication involving the $abc$ conjecture calls for a separate arithmetic reduction, with its constants and its hypotheses stated explicitly~\cite{Oesterle1988Fermat}.
\end{remark}

\subsubsection{Formal strength}\label{subsubsec:PartialProfileACA0}

\begin{remark}[The partial reductions stay within \texorpdfstring{\(\mathsf{ACA}_{0}\)}{ACA0}]
\label{rem:PartialProfileACA0}
The proofs of this subsection use finite package codes, finite sums of valuations, identity morphisms, and the transport inequality of Lemma~\ref{lem:TransportDiscriminant}. The Kodaira type at a place, and with it the defect $\delta_{v}$ and the exponents that enter $\Theta_{K}$, are computed by Tate's algorithm, a primitive recursive procedure on the package code of Definition~\ref{def:EllPkgCode}. For a finite coded list, the maximum in Definition~\ref{def:NormalizedHeightRemainder} is obtained by primitive recursion. Thus Propositions~\ref{prop:DefectNormalizedBound} and~\ref{prop:UnconditionalNormalizedDecomposition}, together with Lemma~\ref{lem:PotentiallyGoodDefectBound}, lie within the same $\mathsf{ACA}_{0}$ treatment as Proposition~\ref{prop:ACA0LogVolumeFormalization}, with real values handled as fast Cauchy names in the sense of Definition~\ref{def:UniverseP}. The remaining input, whether the normalized linear envelope of Hypothesis~\ref{hyp:NormalizedLinearEnvelope} or the multiplicative excess envelope of Hypothesis~\ref{hyp:MultiplicativeExcessEnvelope}, is a new arithmetic statement over the package codes, and no appeal to $\mathsf{ZFC}$ is involved.
\end{remark}

\subsection{The Szpiro--abc link in the intensional framework}\label{subsec:SzpiroAbc}

\begin{remark}[The Szpiro--\texorpdfstring{$abc$}{abc} link in Universes~H and~A]\label{rem:SzpiroAbcUniverses}
The schematic implication ``Szpiro $\Rightarrow abc$'' does not, in itself, commit one to the Hilbertian Universe~H. The argument requires only two ingredients. The first is a functorial way to associate to a primitive triple $(a,b,c)$ with $a+b=c$ an elliptic object $E_{(a,b,c)}$ endowed with enough structure for its discriminant, conductor, and height invariants to be controlled in terms of $\operatorname{rad}(abc)$. The second is an inequality of Szpiro type for such elliptic objects, formulated in whatever environment provides those invariants.

In Universe~H, these requirements are met by working with elliptic curves up to isomorphism over a fixed number field and with the classical invariants $\Delta_E$, $N_E$, and standard heights. In Universe~A, the role of $E_{(a,b,c)}$ may be played instead by an intensional package $X_{(a,b,c)}$ in a quantitative category in the sense of Definition~\ref{def:QuantitativeCategory}, and the Szpiro-type inequality takes the form of a transport bound
\[
  H\bigl(X_{(a,b,c)}^{\mathrm{tgt}}\bigr) \le H\bigl(X_{(a,b,c)}^{\mathrm{src}}\bigr) + c\!\left(f_{(a,b,c)}\right)
\]
for an appropriate coercive morphism $f_{(a,b,c)}$ and a height assignment $H$. Provided that the construction $(a,b,c) \mapsto X_{(a,b,c)}$ exhibits sufficient control, and that the cost $c(f_{(a,b,c)})$ is uniformly bounded in terms of $\operatorname{rad}(abc)$ subject to the standard $\varepsilon$-loss, then the deduction of an $abc$ bound from the Szpiro inequality proceeds entirely within the internal logic of Universe~A. The obstruction of Lemma~\ref{lemma:ToyNoBridge} and Theorem~\ref{thm:DescentObstruction} concerns only the further prospect of reading the same labeled comparison simultaneously as a transport statement in Universe~A and as a classical inequality in Universe~H. It does not impede the internal Universe~A deduction, where the cost is kept rather than collapsed.
\end{remark}

%==================================================================
\section{An explicit example}\label{sec:ExplicitExample}
%==================================================================

The quantitative category of elliptic packages affords an instance in which one may follow the variation of the model defect in a numerically accessible way. It is natural to choose a curve whose conductor is governed by a prime that is not among the very smallest, so that the local behavior is influenced by a slightly larger place. The computation below draws the discriminant height, the conductor complexity, and the model defect into a single inequality, and shows that the transport bound of Lemma~\ref{lem:TransportDiscriminant} is sharp for the particular coordinate change chosen.

Consider the elliptic curve
\[
  E \colon y^2 + y = x^3 - x
\]
defined over $K = \Q$. Identified as $37\mathrm{a}1$ within the Cremona database~\cite{Cremona1997}, this equation is the global minimal Weierstrass model in the sense of Silverman~\cite{SilvermanAEC}. It has conductor $N_E = 37$ and minimal discriminant $\Delta_{\min}(E) = 37$, so that the unique prime of bad reduction is the prime $37$. Let $\mathcal{E}_0^{\mathrm{can}}$ denote the canonical normalized package $\mathrm{NF}_\Q$ attached to $E$ in the sense of Definition~\ref{def:AccumulatedDefect}. By construction this package carries the global minimal Weierstrass model
\[
  \mathcal{W}_0 \colon y^2 + y = x^3 - x
\]
together with a natural choice of local sections and a fixed label.

To isolate the contribution of the bad place, it is convenient to consider the truncated discriminant height of Definition~\ref{def:DiscriminantHeight} at the prime $v = 37$. One sets
\[
  h_\Delta^{(37)}(\mathcal{E}) \coloneqq
  \log \max\bigl(1,\lvert \Delta_{\mathcal{W},37} \rvert_{37}^{-1}\bigr)
\]
for any elliptic package $\mathcal{E}$ whose underlying curve is $E$ and whose Weierstrass model at $37$ is denoted $\mathcal{W}$. For the canonical package this gives
\begin{equation}\label{eq:hDelta2-canonical}
  h_\Delta^{(37)}\bigl(\mathcal{E}_0^{\mathrm{can}}\bigr)
  =
  \log \max\bigl(1,\lvert \Delta_{\mathcal{W}_0,37} \rvert_{37}^{-1}\bigr)
  =
  \log 37
\end{equation}
since $\Delta_{\mathcal{W}_0,37} = 37$ and the normalized absolute value satisfies $\lvert 37 \rvert_{37} = 37^{-1}$.

A second elliptic package is now produced by modifying the Weierstrass model at the place $37$ through a simple dilation. Over $\Q_{37}$ consider the change of variables
\begin{equation}\label{eq:explicit-change-of-variables-37}
  x = 37^{-2} x'
  \qquad\text{and}\qquad
  y = 37^{-3} y'
\end{equation}
on the affine plane. After replacing $(x,y)$ by $(37^{-2} x', 37^{-3} y')$ in the equation of $E$ and simplifying, one obtains a new Weierstrass equation $\mathcal{W}_1$ for $E$ over $\Q_{37}$. The coefficients of $\mathcal{W}_1$ need not be written out, since only the transformation of the discriminant matters here, governed by Lemma~\ref{lem:LocalDiscriminantChange}. That lemma yields $\Delta_{\mathcal{W}_1,37} = u^{-12}\,\Delta_{\mathcal{W}_0,37}$ with $u = 37^{-1}$, so that
\[
  \Delta_{\mathcal{W}_1,37} = (37^{-1})^{-12} \cdot 37 = 37^{12}\cdot 37 = 37^{13}
\]
and $\lvert \Delta_{\mathcal{W}_1,37} \rvert_{37} = 37^{-13}$, hence $\lvert \Delta_{\mathcal{W}_1,37} \rvert_{37}^{-1} = 37^{13}$. The truncated discriminant height at $37$ for the new model is therefore
\begin{equation}\label{eq:hDelta2-new}
  h_\Delta^{(37)}(\mathcal{E}_1)
  =
  \log \max\bigl(1,\lvert \Delta_{\mathcal{W}_1,37} \rvert_{37}^{-1}\bigr)
  =
  13 \log 37.
\end{equation}

The change of variables~\eqref{eq:explicit-change-of-variables-37} defines a morphism of elliptic packages $f \colon \mathcal{E}_0^{\mathrm{can}} \to \mathcal{E}_1$ in the sense of Definition~\ref{def:EllipticPackage}. At $v=37$ the factor $c_{37}(f)$ of Definition~\ref{def:EllipticPackage} is $c_{37}(f) = u^{-12} = (37^{-1})^{-12} = 37^{12}$. Thus, by Definition~\ref{def:ModelDefect},
\begin{equation}\label{eq:explicit-mu}
  \mu(f)
  =
  \log^{+}\!\bigl(\lvert c_{37}(f) \rvert_{37}^{-1}\bigr)
  =
  \log^{+}\!\bigl(\lvert 37^{12} \rvert_{37}^{-1}\bigr)
  =
  \log(37^{12})
  =
  12 \log 37
\end{equation}
where we have used that $\lvert 37^{12}\rvert_{37}=37^{-12}$, hence $\lvert 37^{12}\rvert_{37}^{-1}=37^{12}>1$. Lemma~\ref{lem:TransportDiscriminant} then gives, after restriction to the contribution at $v=37$,
\[
  h_\Delta^{(37)}(\mathcal{E}_1)
  \le
  h_\Delta^{(37)}\bigl(\mathcal{E}_0^{\mathrm{can}}\bigr)
  +
  \mu(f).
\]
Inserting the values from~\eqref{eq:hDelta2-canonical}, \eqref{eq:hDelta2-new}, and~\eqref{eq:explicit-mu} yields the sharp identity
\[
  13\log 37
  \le
  \log 37 + 12\log 37
  =
  13\log 37.
\]
For this choice of local coordinate change at $37$, the transport inequality of Lemma~\ref{lem:TransportDiscriminant} holds with equality at the truncated level $v=37$, since the entire change in $h_\Delta^{(37)}$ is accounted for exactly by the defect term $\mu(f)$ computed in~\eqref{eq:explicit-mu}.

The conductor complexity of $\mathcal{E}_0^{\mathrm{can}}$ and of $\mathcal{E}_1$ in the sense of Definition~\ref{def:ConductorComplexity} is $n(\mathcal{E}_0^{\mathrm{can}}) = n(\mathcal{E}_1) = \log N_E = \log 37$. The three quantities $h_\Delta(\mathcal{E})$, $n(\mathcal{E})$, and $\defect(\mathcal{E})$ of Section~\ref{sec:IntensionalSzpiro} are thus brought together within a single frame of reference. The discriminant height changes under a visible modification of the local model at $37$, while the conductor complexity stays anchored to the underlying curve $E$. The morphism $f$ keeps the passage from one elliptic package to the other, and its cost $\mu(f)$ gives an explicit upper bound for the accumulated defect in the sense of Definition~\ref{def:AccumulatedDefect}. Any morphism $g \colon \mathcal{E}_0^{\mathrm{can}} \to \mathcal{E}_1$ in $\mathbf{EllPkg}_\Q$ provides an upper bound $\defect(\mathcal{E}_1) \le \mu(g)$, so that in particular $\defect(\mathcal{E}_1) \le \mu(f) = 12 \log 37$.

Under Hypothesis~\ref{hyp:ReducedNormalizedProfile}, Proposition~\ref{prop:NormalizedToIntensionalSzpiro} yields that for every real number $\varepsilon > 0$ there exists a constant $C_{\varepsilon,\Q}$ such that
\begin{equation}\label{eq:explicit-intensional-szpiro}
  h_\Delta(\mathcal{E}_1)
  \le
  (6+\varepsilon)\,n(\mathcal{E}_1)
  +
  \defect(\mathcal{E}_1)
  +
  C_{\varepsilon,\Q}.
\end{equation}
Combining the sharp identity above with the relation $n(\mathcal{E}_1)=\log 37$ and the estimate $\defect(\mathcal{E}_1)\le 12\log 37$ shows that \eqref{eq:explicit-intensional-szpiro} may be realized with a suitable choice of $C_{\varepsilon,\Q}$. The demonstrative value of the calculation is the simultaneous integration of the height, the conductor, and the model defect into the architecture of a single inequality.

This higher-conductor example over $\Q$ is elementary from a Diophantine perspective. It serves, however, to exhibit the behavior of the discriminant height under a controlled deformation at a non-trivial prime, with a minimum of technical overhead. In this way it illustrates how an intensional Szpiro-type relation in the sense of Proposition~\ref{prop:NormalizedToIntensionalSzpiro} arises from transport inequalities in a quantitative category of elliptic packages.

%==================================================================
\section{Descent to a pre-set-theoretic universe}\label{sec:UniverseP}
%==================================================================

Arithmetic practice remains attentive to numbers, primes, and finite packets of coefficients. Many of the statements that matter -- such as the Szpiro inequality and the $abc$ conjecture -- permit a formulation that appeals solely to such finite data, without invoking any general facility to form infinite sets of arbitrary objects. It is appropriate, therefore, to circumscribe an arithmetic universe that forbids the completed totalities which would erase intensional distinctions, while remaining a natural environment for working number theory. This universe should be regarded as a convenient internal presentation of objects and uniformity statements that are intended to be formalizable in $\mathsf{ACA}_0$, rather than as a separate foundational system.

\subsection{A pre-set-theoretic arithmetic universe}\label{subsec:PreSetArithmetic}

\begin{definition}[The pre-set-theoretic arithmetic universe]\label{def:UniverseP}
Fix $\mathsf{ACA}_0$ as the ambient theory, in the language of second-order arithmetic. The \emph{pre-set-theoretic arithmetic universe}, denoted $\mathfrak{P}$ and called Universe~P, is the category of arithmetically presented codes defined internally to $\mathsf{ACA}_0$. An object $X$ of $\mathfrak{P}$ consists of a code $c_X \subseteq \mathbb{N}$ representing a countable many-sorted presentation that includes the following data:
\begin{enumerate}
\item[(i)] for each sort $\sigma$, a domain $D_X^\sigma \subseteq \mathbb{N}$ defined by an arithmetic predicate relative to $c_X$,
\item[(ii)] for each basic relation symbol $R$, an arithmetic predicate relative to $c_X$ representing a subset of a finite product of these domains,
\item[(iii)] for each basic function symbol $F$, an arithmetic graph relative to $c_X$ representing a total function on the domains,
\item[(iv)] for each quantitative datum in the presentation, either a rational code or a fast Cauchy name of rationals, arithmetical in $c_X$, when a real-valued quantity is required.
\end{enumerate}
A morphism $f \colon X \to Y$ is a code $c_f \subseteq \mathbb{N}$ whose graph is arithmetical in $c_X \oplus c_Y \oplus c_f$, and one requires that $\mathsf{ACA}_0$ prove this code defines a total function respecting the structure of $X$ and $Y$. Composition is derived from the composition of coded graphs, and the identity morphisms are coded diagonals -- for an object $X$, the identity is the arithmetic predicate identifying the set of pairs $\langle n, n \rangle$ for every $n$ in the domains of $X$. Self-identification within this world is not a background fact, and it is rather a program requiring a definite arithmetic representation.
\end{definition}

\begin{proposition}[Formalizability of the $\mathfrak{P}$-constructions in \texorpdfstring{$\mathsf{ACA}_0$}{ACA0}]\label{prop:ACAFormalizable}
All constructions on $\mathfrak{P}$ used in the descent of Proposition~\ref{prop:APDescentToP} are formalizable in $\mathsf{ACA}_0$. More precisely, within $\mathsf{ACA}_0$ one can form and manipulate:
\begin{enumerate}
\item[(i)] finite tuples, words, matrices, and other standard combinatorial codes,
\item[(ii)] arithmetically defined domains, graphs, and subobjects associated with a presentation code,
\item[(iii)] pullbacks, finite products, and other finite-limit constructions on arithmetically presented objects,
\item[(iv)] rational-valued cost functions and real-valued quantities represented by fast Cauchy names of rationals,
\item[(v)] the comparisons, bounded searches, and finite minimization procedures required to descend from extensional cost data to combinatorial bounds.
\end{enumerate}
Hence the transcription of Proposition~\ref{prop:APDescentToP} may be read as a theorem internal to $\mathsf{ACA}_0$, once its hypotheses are expressed in the coding formalism of $\mathfrak{P}$.
\end{proposition}

\begin{proof}
All finite coding operations used in the definition of $\mathfrak{P}$ are primitive recursive. The only non-trivial set-existence principle required in passing from a presentation to its associated domains, graphs, and subobjects is the formation of sets defined by arithmetic predicates, which is exactly what arithmetical comprehension provides. Real-valued quantities are represented by fast Cauchy names of rationals, and the arithmetic manipulation of such names is standard in $\mathsf{ACA}_0$. The descent argument uses only these coding operations, arithmetically definable graphs, finite search procedures, and comparisons of coded rational or Cauchy data. No stronger comprehension scheme, and no higher-order set-existence principle, is invoked.
\end{proof}

\subsection{Heights and defects as finite integer data}\label{subsec:LogIntData}

\begin{definition}[Logarithmic integer datum]\label{def:LogIntDatum}
Let $K$ be a number field presented by finite data in Universe~P. A \emph{logarithmic integer datum over $K$} is a positive integer $M$. It represents the real number $\frac{1}{[K:\Q]}\,\log M$ and is read as an element of $\R_{\ge 0}$ whenever a numerical evaluation is required. A comparison between two logarithmic integer data $M_1$ and $M_2$ over the same field reduces to the integer inequality $M_1 \le M_2$. More generally, an inequality of the form
\[
  \frac{1}{[K:\Q]}\,\log M_1
  \le
  \bigl(6+\tfrac{a}{b}\bigr)\,\frac{1}{[K:\Q]}\,\log M_2
  +
  \frac{1}{[K:\Q]}\,\log M_3
  +
  C
\]
with $a,b$ positive integers and $C$ a real constant admits an \emph{exponentiated integer reformulation}. One multiplies through by $b\,[K:\Q]$ and exponentiates, producing the comparison
\[
  M_1^b \le D \cdot M_2^{6b+a} \cdot M_3^b
\]
where $D$ is any positive integer satisfying $D \ge \lceil e^{\,b\,[K:\Q]\,C} \rceil$. This last expression is a comparison of positive integers and lies entirely within the arithmetic of Universe~P. The integer $D$ is determined by a finite computation once $K$ and the rational precision $a/b$ have been fixed.
\end{definition}

\begin{remark}[Finite and archimedean quantitative data]\label{rem:HeightsAsLogIntData}
The finite contribution to the discriminant height is represented by a positive integer
\[
  M_{\Delta,\mathrm{fin}}
  =
  \prod_{v<\infty}
  p_v^{f_v\max\{0,\operatorname{ord}_v(\Delta_{\mathcal{W},v})\}}
\]
through the identity
\[
  h_\Delta^{\mathrm{fin}}(\mathcal{E})
  =
  \frac{\log M_{\Delta,\mathrm{fin}}}{[K:\Q]}
\]
where $p_v$ is the rational prime below $v$ and $f_v$ is the residue degree. The conductor complexity is represented by the positive integer $N_E$. The finite contribution to a model defect has the same form, computed from its local dilation exponents.

Archimedean contributions depend upon the package and upon the transport. They are supplied by certified fast Cauchy names, or by package-dependent bounds sufficient for the particular inequality. A bound depending only upon the field is not imposed.

An exact transcription retains these archimedean names. Replacing a contribution by an upper bound gives a sufficient inequality, whose relation to the initial statement must then be specified.
\end{remark}

\begin{definition}[Certified number field code]\label{def:CertifiedFieldCode}
Within Universe~P, a \emph{certified number field code} for a number field $K$ of degree $d = [K:\Q]$ is a finite tuple of integer data comprising the following items:
\begin{enumerate}
\item[\textit{(i)}] a monic irreducible polynomial $f(x) \in \Z[x]$ of degree $d$, together with a fixed ordering of its roots in $\C$ that determines the real and complex embeddings of $K$,
\item[\textit{(ii)}] a \emph{certified integral basis}, a list of $d$ elements $(\omega_1,\dots,\omega_d)$ of $K$, each presented as a polynomial in the root $\alpha$ of $f$ with rational coefficients, such that $\mathcal{O}_K = \Z\omega_1 + \cdots + \Z\omega_d$, the certification consisting of the field discriminant $d_K \in \Z$ and a finite proof witness that confirms the maximality of the order $\Z[\omega_1,\dots,\omega_d]$ (for instance a transcript of the steps of the Round Four algorithm of Ford--Zassenhaus as in~\cite{Cohen1993}),
\item[\textit{(iii)}] certified procedures for factoring a supplied rational prime in $\mathcal{O}_K$, for presenting the resulting prime ideals and residue fields, and for computing the valuations of supplied non-zero field elements. The code may retain a pre-computed table at the primes dividing the field discriminant, each entry a tuple $(e_v, f_v, \pi_v)$ keeping the ramification index, the residue degree, and a uniformizer encoded as a $\Z$-linear combination of the integral basis,
\item[\textit{(iv)}] certified procedures for evaluating the real and complex embeddings of algebraic elements to a prescribed rational precision. Archimedean quantities attached to a package are computed from its own coefficients together with additional analytic data.
\end{enumerate}
The purpose of this enriched presentation is to ensure that the computation of conductors, Tamagawa numbers, local dilation factors, and normalization morphisms proceeds within Universe~P by deterministic operations on the certified data, without requiring the execution of maximal-order algorithms at runtime.
\end{definition}

The descent from Universe~A to Universe~P admits an explicit transcription. The principle is lucid. An elliptic package comprises finitely many coefficients, finitely many local parameters, and a finite label. A morphism consists of finitely many coordinate changes, and a defect appears as a sum of finitely many local contributions. All of this registers as finite arithmetic data, manipulable without quantification over any global class of objects.

\begin{definition}[Arithmetic code for an elliptic package]\label{def:EllPkgCode}
Let $K$ be a number field equipped with a certified number field code in the sense of Definition~\ref{def:CertifiedFieldCode}. An \emph{elliptic package code over $K$} is a finite tuple of integers encoding the following data:
\begin{enumerate}
\item[\textit{(i)}] the coefficients $[a_1,a_2,a_3,a_4,a_6]$ of a Weierstrass equation for $E$, each expressed as a $\Z$-linear combination of the certified integral basis $(\omega_1,\dots,\omega_d)$ from the field code,
\item[\textit{(ii)}] a finite list representing the set $S$ of marked places, by listing the prime ideals above the corresponding rational primes through the indices supplied by the factoring procedure of the field code, together with a tag for each archimedean place,
\item[\textit{(iii)}] the finite parameters describing the local sections or level structures $\{\sigma_v\}_{v\in S}$, each encoded as a tuple of elements of $K$ presented by rational coordinates in the certified basis,
\item[\textit{(iv)}] a finite label $\lambda$, transcribed as a finite string over a fixed alphabet.
\end{enumerate}
A \emph{morphism code} contains coordinate-change parameters in the explicitly admitted arithmetic domain, together with exact identities verifying the change of equation. Parameters in $K$ are represented by rational coordinates in the certified basis. A parameter in a finite algebraic extension carries the corresponding extension code. Local precision data are admitted only when the required arithmetic assertions have certified finite precision bounds. Arbitrary elements of a completion require infinite names and lie outside the finite-code fragment. When one works in the variable-base category $\mathbf{EllPkg}_{\mathrm{var}}$ of Definition~\ref{def:EllPkgVar}, then a morphism code from a package over $K$ to a package over $L = K(\sqrt[\ell]{q})$ carries in addition a \emph{field extension code}, comprising the minimal polynomial of $\sqrt[\ell]{q}$ over $K$, a certified integral basis for $\mathcal{O}_L$ extending the basis of $\mathcal{O}_K$, and the polynomial-basis matrix expressing the old basis elements in terms of the new. The Kummer detachment cost $\mu_{\mathrm{kum}}(f_{\mathrm{rad}})$ of Definition~\ref{def:KummerDefect} is then computed from the degree $\ell$ alone, and the descent cost of Hypothesis~\ref{hyp:ControlledDescentWithSupport} is determined by the ramification data of the certified extension code, together with the tame support of the radicand. One writes $\mathsf{PkgCode}_K$ for the finite-type domain of such codes and $\mathsf{PkgCode}_{\mathrm{var}}$ for the variable-base variant.
\end{definition}

\begin{remark}[Computability of invariants from package codes]\label{rem:PkgCodeComputability}
Universe~P admits function symbols for arithmetic invariants computed from codes. The following quantities are determined by finite, deterministic procedures operating on the data of Definitions~\ref{def:EllPkgCode} and~\ref{def:CertifiedFieldCode}. The non-archimedean part of $h_\Delta$ is computed from the Weierstrass coefficients by evaluating valuations at each prime ideal of $\mathcal{O}_K$ dividing the discriminant, and by Remark~\ref{rem:HeightsAsLogIntData} this part is a logarithmic integer datum, while the archimedean part is supplied by a fast Cauchy name obtained from the certified embedding procedures of Definition~\ref{def:CertifiedFieldCode}\textit{(iv)}. The conductor $N_E$ is computed by Ogg's formula from the local reduction types at each prime dividing the discriminant, so that the conductor complexity is a logarithmic integer datum in the strict sense. The defect $\mu(f)$ of a morphism code $f$ is computed from the local dilation factors as in Definition~\ref{def:ModelDefect}, each $|u_v|_v$ being a power of the residue characteristic read from the uniformizer data. The normalization map $\mathrm{NF}_K$ and the canonical morphism $\nu_{\mathcal{E}}$ of Definition~\ref{def:AccumulatedDefect} are executed by Tate's algorithm following the deterministic discipline of Remark~\ref{rem:DeterministicNormalizationDiscipline}. No operation above requires quantification over an infinite collection or appeals to the existence of a completed totality, since all inputs and outputs are finite tuples of integers and all intermediate steps are bounded by the degree $[K:\Q]$ and by the number of primes dividing the conductor.
\end{remark}

\subsection{The transcription}\label{subsec:Transcription}

\begin{proposition}[Transcription of intensional Szpiro into Universe~P]\label{prop:APDescentToP}
Fix a number field $K$ and a normalization procedure $\mathrm{NF}_K$ in the sense of Definition~\ref{def:AccumulatedDefect}, both presented by finite data. Then there exists a Universe~P formula $\mathrm{Szp}^P_{\varepsilon,K}(d)$ with $d \in \mathsf{PkgCode}_K$ as its only free variable, possessing the following property. For every elliptic package $\mathcal{E}$ over $K$ with code $d$, the inequality
\[
h_\Delta(\mathcal{E})
\le
(6+\varepsilon)\,n(\mathcal{E})
+
\defect(\mathcal{E})
+
C_{\varepsilon,K}
\]
from~\eqref{eq:IntensionalSzpiroBasic} holds if and only if $\mathrm{Szp}^P_{\varepsilon,K}(d)$ holds when the terms $h_\Delta$, $n$, and $\defect$ are interpreted as the corresponding functions on codes. Consequently, the basic intensional Szpiro inequality may be phrased in Universe~P as a family of inequalities over finite codes, devoid of quantification over a totality of all elliptic curves.
\end{proposition}

\begin{proof}
By Definition~\ref{def:EllPkgCode}, a package is determined by a finite tuple $d$. The quantities $h_\Delta(\mathcal{E})$ and $n(\mathcal{E})$ are computed from $d$ by the explicit procedures of arithmetic geometry, and $\defect(\mathcal{E})$ is computed from $d$ by way of the normalization procedure of Definition~\ref{def:AccumulatedDefect} together with the defect $\mu$ of Definition~\ref{def:ModelDefect}. One may thus write a Universe~P formula asserting the inequality among these computed numbers, and no further ontological commitment is required.
\end{proof}

\begin{remark}[Universe~P as recipient]\label{rem:PAsTarget}
Universe~P acts as a natural recipient for the structures of Universe~A. Universe~A keeps transport and defect, while Universe~P stores the resulting numerical inequalities on codes and declines to enforce static identification in the sense of Definition~\ref{def:StaticDynamic} by fiat. In this manner the descent from A to P bypasses the extensional collapse of Definition~\ref{def:ExtensionalCollapse} entirely.
\end{remark}

\begin{proposition}[Soundness of descent along normalization]\label{prop:SoundnessDescentNormalization}
Fix a number field $K$ and a normalization procedure $\mathrm{NF}_K$ in the sense of Definition~\ref{def:AccumulatedDefect}, presented by finite data in Universe~P, and assume Hypothesis~\ref{hyp:ReducedNormalizedProfile}. Then for every real number $\varepsilon>0$ there exists a constant $C_{\varepsilon,K}$ such that for every package code $d \in \mathsf{PkgCode}_K$ one has the Universe~P inequality
\[
h_\Delta\bigl(\mathrm{NF}_K(d)\bigr)
\le
(6+\varepsilon)\,n\bigl(\mathrm{NF}_K(d)\bigr) + C_{\varepsilon,K}
\]
where $\mathrm{NF}_K(d)$ is the code obtained by applying the normalization algorithm to $d$, and $h_\Delta$ and $n$ are interpreted as the corresponding functions on codes.
\end{proposition}

\begin{proof}
Let $d$ be a package code and let $\mathcal{E}$ be an elliptic package over $K$ with code $d$. Consider the normalized package $\mathrm{NF}_K(\mathcal{E})$. By Definition~\ref{def:AccumulatedDefect}\textup{\ref{item:NF_idempotent}} one has $\defect(\mathrm{NF}_K(\mathcal{E}))=0$, so that the intensional inequality of Proposition~\ref{prop:NormalizedToIntensionalSzpiro} specializes to $h_\Delta(\mathrm{NF}_K(\mathcal{E})) \le (6+\varepsilon)\,n(\mathrm{NF}_K(\mathcal{E})) + C_{\varepsilon,K}$. Transcribe this inequality to codes using Definition~\ref{def:EllPkgCode} and Remark~\ref{rem:PkgCodeComputability}, and use that $\mathrm{NF}_K(d)$ is the code of $\mathrm{NF}_K(\mathcal{E})$.
\end{proof}

\subsection{How the coherence data is discharged}\label{subsec:TwoMorphismErasure}

\begin{remark}[How $2$-morphisms are resolved in the descent to finite codes]\label{rem:TwoMorphismErasure}
The passage from the labeled quantitative $(2,1)$-category $\mathfrak{A}$ of Definition~\ref{def:UniverseAFormal} to the finite codes of Definition~\ref{def:EllPkgCode} traverses the homotopy $1$-category $\Ho(\mathfrak{A})$ of Definition~\ref{def:HoOneCategory}, and this intermediate step absorbs the $2$-morphisms. The descent may be organized in three stages. In the first, by Definition~\ref{def:HoOneCategory}, the homotopy $1$-category $\Ho(\mathfrak{A})$ retains the same objects as $\mathfrak{A}$ and replaces each $1$-morphism by its $2$-isomorphism class. By Proposition~\ref{prop:CoherenceStability}\textup{(i)} the cost factors through this quotient, every invertible $2$-morphism $\alpha \colon f \Rightarrow g$ is collapsed to the identity $[f] = [g]$, and the numerical datum $c(f) = c(g)$ attached to the class is well defined, so that no cost accrues from the collapse. In the second stage, within each $2$-isomorphism class one selects a representative $1$-morphism. In the elliptic-package fragment of Universe~A of Section~\ref{sec:EllipticPackages}, this selection is governed by the normalization procedure $\mathrm{NF}_K$ of Definition~\ref{def:AccumulatedDefect}, which fixes a canonical presentation within each extensional fiber, and the choice depends only upon the finitely presented data of the source and target packages by the deterministic tie-breaking discipline of Definition~\ref{def:AccumulatedDefect}\textup{\ref{item:NF_finitary}}. In the third stage, the selected representative is transcribed into a tuple of integers by Definition~\ref{def:EllPkgCode}, and from these data one computes the cost $\mu(f)$ and the discriminant height $h_\Delta$ by the procedures of Remark~\ref{rem:PkgCodeComputability}. The net effect is that the $2$-morphisms of $\mathfrak{A}$ leave no residue in the finite code. Their role was to witness that two $1$-morphisms differed only by a coherent rearrangement of auxiliary choices, and Proposition~\ref{prop:CoherenceStability}\textup{(ii)} ensures that such rearrangements do not alter the cost. Once the cost has been kept as a rational number or a computable real in $\mathsf{PkgCode}_K$, the coherence data has been fully discharged, and the transport inequality of Lemma~\ref{lem:TransportDiscriminant} and the intensional Szpiro bound of Proposition~\ref{prop:NormalizedToIntensionalSzpiro} operate on the finite codes exactly as they would in $\Ho(\mathfrak{A})$.
\end{remark}

\begin{remark}[Compatibility with the descent obstruction]\label{rem:SoundnessVsNoBridge}
Proposition~\ref{prop:SoundnessDescentNormalization} does not contradict Lemma~\ref{lemma:ToyNoBridge} or Theorem~\ref{thm:DescentObstruction}. The obstruction requires a two-sided comparison passing through a functorial, label-transparent bridge in the sense of Definition~\ref{def:ToyAliasedStandard}, and collapses it to $H \le H + \Delta$ once the labels have been erased. The descent to Universe~P avoids this collapse by declining to perform the identification. In Universe~P the primary objects are finite codes $d \in \mathsf{PkgCode}_K$, equality between two codes is typographical, and two codes encoding isomorphic elliptic curves through different Weierstrass models remain distinct. The discriminant height computed from each code reflects the specific dilation factors that produced it, and the model defect $\mu(f)$ is itself a determination on morphism codes that keeps the structural friction between two presentations as an explicit numerical offset. Moreover the inequality that descends to Universe~P is one-sided. It states $h_{\Delta}(\mathrm{NF}_K(d)) \le (6+\varepsilon)\,n(\mathrm{NF}_K(d)) + C_{\varepsilon,K}$ for each code $d$, and it is not a comparison between two evaluations of a single invariant attached to an isomorphism class. No label-transparent bridge of the kind whose collapse Theorem~\ref{thm:DescentObstruction} governs is invoked.
\end{remark}

\begin{remark}[A brief link back to Universe~H]\label{rem:LinkBackToH}
If one chooses to read $\mathrm{NF}_K(\mathcal{E})$ as a canonical presentation of the underlying curve in the extensional sense, then Proposition~\ref{prop:SoundnessDescentNormalization} may be read as an extensional Szpiro-type inequality for that chosen representative. This form of harmonization is optional, since Universe~P is a sufficient recipient for arithmetic statements once they are expressed on finite codes.
\end{remark}

%==================================================================
\section{Formalization sketches in Coq}\label{sec:ToyFormalization}
%==================================================================

The Calculus of Inductive Constructions, as implemented in the Coq system, offers a place where one may test whether a quantitative argument truly survives the discipline of a small kernel. The kernel itself is indifferent to the arithmetic meaning of a transport. If one wishes to distinguish a transport of zero cost from a transport that carries a visible distortion, then this distinction must be written into the data.

In the present paper, the native environment is a labeled quantitative $(2,1)$-category in the sense of Definition~\ref{def:UniverseAFormal}. It is therefore natural, in a toy formalization, to keep $2$-morphisms literal and to impose the two requirements that matter most for the numerical accounting:
\begin{subtlelist}
\item Every $2$-morphism is invertible.
\item The distortion cost is invariant under $2$-isomorphism.
\end{subtlelist}
Under these conditions, coherence carries no numerical weight. It remains present as structure, while the scalar bounds attach only to $1$-morphisms.

\subsection{A minimal quantitative \texorpdfstring{$(2,1)$}{(2,1)}-category in Coq}\label{subsec:CoqQuantCat}

To fix ideas, we present a small interface for a strict $(2,1)$-category equipped with a cost and a height. The code is written as explanatory scaffolding. Several laws are stated as axioms, and the interface is kept minimal, while it remains enough to express the transport inequalities that recur throughout the paper.

\begin{scriptsize}
\begin{verbatim}
(*
  Toy quantitative (2,1)-categories.

  The interface below is written to mirror the data of a labeled quantitative
  (2,1)-category in the sense of Definition~\ref{def:UniverseAFormal}.

  Conventions:
    - comp1 f g means "first f, then g"
    - 2-cells are written as Two f g, and are all invertible
*)

From Coq Require Import Reals.
From Coq Require Import Lra.

Open Scope R_scope.

Record QuantTwoOneCat := {
  Obj : Type ;
  Hom : Obj -> Obj -> Type ;

  (* 2-cells between parallel 1-morphisms *)
  Two : forall {A B : Obj}, Hom A B -> Hom A B -> Type ;

  (* 1-categorical structure *)
  id1   : forall A : Obj, Hom A A ;
  comp1 : forall {A B C : Obj}, Hom A B -> Hom B C -> Hom A C ;

  comp1_assoc :
    forall (A B C D : Obj)
           (f : Hom A B) (g : Hom B C) (h : Hom C D),
      comp1 (comp1 f g) h = comp1 f (comp1 g h) ;

  comp1_id_l :
    forall (A B : Obj) (f : Hom A B),
      comp1 (id1 A) f = f ;

  comp1_id_r :
    forall (A B : Obj) (f : Hom A B),
      comp1 f (id1 B) = f ;

  (* 2-categorical structure, strict for convenience *)
  id2    : forall {A B : Obj} (f : Hom A B), Two f f ;
  vcomp2 : forall {A B : Obj} {f g h : Hom A B},
             Two f g -> Two g h -> Two f h ;

  vcomp2_assoc :
    forall (A B : Obj) (f g h k : Hom A B)
           (a : Two f g) (b : Two g h) (c : Two h k),
      vcomp2 a (vcomp2 b c) = vcomp2 (vcomp2 a b) c ;

  vcomp2_id_l :
    forall (A B : Obj) (f g : Hom A B) (a : Two f g),
      vcomp2 (id2 f) a = a ;

  vcomp2_id_r :
    forall (A B : Obj) (f g : Hom A B) (a : Two f g),
      vcomp2 a (id2 g) = a ;

  (* whiskering *)
  whiskerL :
    forall {A B C : Obj} (f : Hom A B) {g h : Hom B C},
      Two g h -> Two (comp1 f g) (comp1 f h) ;

  whiskerR :
    forall {A B C : Obj} {f g : Hom A B} (a : Two f g) (h : Hom B C),
      Two (comp1 f h) (comp1 g h) ;

  (* (2,1)-condition: explicit inverses for 2-cells *)
  inv2 :
    forall {A B : Obj} {f g : Hom A B}, Two f g -> Two g f ;

  inv2_left :
    forall {A B : Obj} {f g : Hom A B} (a : Two f g),
      vcomp2 a (inv2 a) = id2 f ;

  inv2_right :
    forall {A B : Obj} {f g : Hom A B} (a : Two f g),
      vcomp2 (inv2 a) a = id2 g ;

  (* distortion cost on 1-morphisms *)
  cost : forall {A B : Obj}, Hom A B -> R ;

  cost_nonneg :
    forall (A B : Obj) (f : Hom A B), 0 <= cost f ;

  cost_id :
    forall A : Obj, cost (id1 A) = 0 ;

  cost_comp :
    forall (A B C : Obj) (f : Hom A B) (g : Hom B C),
      cost (comp1 f g) <= cost f + cost g ;

  (* invariance under 2-isomorphism *)
  cost_two_invariant :
    forall (A B : Obj) (f g : Hom A B), Two f g -> cost f = cost g ;

  (* height on objects *)
  height : Obj -> R ;

  (* transport inequality *)
  height_trans :
    forall (A B : Obj) (f : Hom A B),
      height B <= height A + cost f
}.
\end{verbatim}
\end{scriptsize}

From this interface, the familiar consequences of the transport inequality follow directly. The axiom \texttt{cost\_two\_invariant} ensures that the scalar cost depends only upon a $1$-morphism up to coherent change of presentation, and the invertible $2$-cells carry no independent numerical weight.

\begin{scriptsize}
\begin{verbatim}
Section DerivedLemmas.

Context (C : QuantTwoOneCat).

Lemma height_no_increase_if_cost_zero :
  forall (A B : Obj C) (f : Hom C A B),
    cost (C:=C) f = 0 ->
    height C B <= height C A.
Proof.
  intros A B f Hc.
  pose proof (height_trans C A B f) as H.
  lra.
Qed.

Lemma height_increase_forces_cost :
  forall (A B : Obj C) (f : Hom C A B),
    height C A <= height C B ->
    height C B - height C A <= cost (C:=C) f.
Proof.
  intros A B f Hle.
  pose proof (height_trans C A B f) as H.
  lra.
Qed.

Lemma cost_equal_under_two_cell :
  forall (A B : Obj C) (f g : Hom C A B),
    Two C f g ->
    cost (C:=C) f = cost (C:=C) g.
Proof.
  intros A B f g a.
  apply (cost_two_invariant C A B f g a).
Qed.

End DerivedLemmas.
\end{verbatim}
\end{scriptsize}

\begin{remark}[Coherence is present, and it is numerically silent]\label{rem:CoqCoherenceSilent}
The lemma \texttt{cost\_equal\_under\_two\_cell} is the toy formal analogue of the requirement that the cost factors through $\Ho(\mathfrak{A})$. The $2$-cells remain present and invertible, and yet they do not perturb the scalar accounting. This silence is load-bearing rather than incidental. It is the very property on which the collapse of the descent to finite codes, described in Remark~\ref{rem:TwoMorphismErasure}, depends, and its consequence for the management of the $2$-cells in the formal interface is taken up in Remark~\ref{rem:CoqPseudoNat}.
\end{remark}

\subsection{A toy elliptic-package fragment}\label{subsec:CoqEllPkg}

We now sketch an elliptic-package fragment in which the $2$-morphisms encode harmless changes of presentation of transports. One may keep the geometric content abstract. The cost \texttt{mu} plays the role of the model defect $\mu$ of Definition~\ref{def:ModelDefect}, and the height \texttt{hDelta} plays the role of $h_\Delta$ of Definition~\ref{def:DiscriminantHeight}.

\begin{scriptsize}
\begin{verbatim}
(*
  Toy elliptic packages as a (2,1)-category.

  Objects: elliptic packages (left abstract)
  1-morphisms: transports between packages
  2-morphisms: coherent identifications between transports, all invertible
*)

Parameter EllPkg  : Type.
Parameter EllHom  : EllPkg -> EllPkg -> Type.
Parameter EllTwo  : forall {E F : EllPkg}, EllHom E F -> EllHom E F -> Type.

Parameter Ell_id1  : forall E : EllPkg, EllHom E E.
Parameter Ell_comp1 : forall {E F G : EllPkg},
  EllHom E F -> EllHom F G -> EllHom E G.

Axiom Ell_comp1_assoc :
  forall (E F G H : EllPkg)
         (f : EllHom E F) (g : EllHom F G) (h : EllHom G H),
    Ell_comp1 (Ell_comp1 f g) h = Ell_comp1 f (Ell_comp1 g h).

Axiom Ell_comp1_id_l :
  forall (E F : EllPkg) (f : EllHom E F),
    Ell_comp1 (Ell_id1 E) f = f.

Axiom Ell_comp1_id_r :
  forall (E F : EllPkg) (f : EllHom E F),
    Ell_comp1 f (Ell_id1 F) = f.

Parameter Ell_id2 : forall {E F : EllPkg} (f : EllHom E F), EllTwo f f.
Parameter Ell_vcomp2 :
  forall {E F : EllPkg} {f g h : EllHom E F},
    EllTwo f g -> EllTwo g h -> EllTwo f h.

Parameter Ell_inv2 :
  forall {E F : EllPkg} {f g : EllHom E F}, EllTwo f g -> EllTwo g f.

Axiom Ell_inv2_left :
  forall {E F : EllPkg} {f g : EllHom E F} (a : EllTwo f g),
    Ell_vcomp2 a (Ell_inv2 a) = Ell_id2 f.

Axiom Ell_inv2_right :
  forall {E F : EllPkg} {f g : EllHom E F} (a : EllTwo f g),
    Ell_vcomp2 (Ell_inv2 a) a = Ell_id2 g.

(* Model defect mu(f) *)
Parameter mu : forall {E F : EllPkg}, EllHom E F -> R.

Axiom mu_nonneg :
  forall (E F : EllPkg) (f : EllHom E F), 0 <= mu f.

Axiom mu_id :
  forall E : EllPkg, mu (Ell_id1 E) = 0.

Axiom mu_comp :
  forall (E F G : EllPkg) (f : EllHom E F) (g : EllHom F G),
    mu (Ell_comp1 f g) <= mu f + mu g.

Axiom mu_two_invariant :
  forall (E F : EllPkg) (f g : EllHom E F), EllTwo f g -> mu f = mu g.

(* Discriminant height *)
Parameter hDelta : EllPkg -> R.

(* Transport inequality, mirroring Lemma~\ref{lem:TransportDiscriminant} *)
Axiom hDelta_trans :
  forall (E F : EllPkg) (f : EllHom E F),
    hDelta F <= hDelta E + mu f.
\end{verbatim}
\end{scriptsize}

Once one has \texttt{mu\_two\_invariant}, the familiar consequences of the transport inequality may be written without collapsing $2$-cells into equalities.

\begin{scriptsize}
\begin{verbatim}
Lemma hDelta_no_increase_if_mu_zero :
  forall (E F : EllPkg) (f : EllHom E F),
    mu f = 0 ->
    hDelta F <= hDelta E.
Proof.
  intros E F f Hmu.
  pose proof (hDelta_trans E F f) as H.
  lra.
Qed.

Lemma hDelta_increase_forces_mu :
  forall (E F : EllPkg) (f : EllHom E F),
    hDelta E <= hDelta F ->
    hDelta F - hDelta E <= mu f.
Proof.
  intros E F f Hle.
  pose proof (hDelta_trans E F f) as H.
  lra.
Qed.
\end{verbatim}
\end{scriptsize}

\subsection{Normalization as a \texorpdfstring{$2$}{2}-endofunctor with a pseudonatural comparison}\label{subsec:CoqPseudoNat}

In the body of the paper, normalization is treated as a disciplined choice of a reference point in each extensional fiber, together with a distinguished transport. In a $(2,1)$-category it is pleasant to write this as a $2$-endofunctor equipped with a pseudonatural transformation. The toy interface below keeps this shape, while leaving the coherence laws as axioms. The purpose is to keep the $2$-cells explicit, so that the formal picture matches the ambient language of Universe~A.

\begin{scriptsize}
\begin{verbatim}
(*
  A toy 2-endofunctor on the elliptic-package (2,1)-category.

  We keep only what is needed for pseudonaturality.
*)

Record EllEndo2 := {
  Fobj : EllPkg -> EllPkg ;
  Fmor : forall {E F : EllPkg}, EllHom E F -> EllHom (Fobj E) (Fobj F) ;

  F_id1 :
    forall E : EllPkg, Fmor (Ell_id1 E) = Ell_id1 (Fobj E) ;

  F_comp1 :
    forall (E F G : EllPkg) (f : EllHom E F) (g : EllHom F G),
      Fmor (Ell_comp1 f g) = Ell_comp1 (Fmor f) (Fmor g) ;

  Ftwo :
    forall (E F : EllPkg) (f g : EllHom E F),
      EllTwo f g -> EllTwo (Fmor f) (Fmor g)
}.

(*
  Pseudonatural transformation between 2-endofunctors.

  For each 1-morphism f, we store an explicit invertible 2-cell witnessing the
  pseudonaturality square.
*)

Record EllPseudoNat (F G : EllEndo2) := {
  eta_obj :
    forall E : EllPkg, EllHom (Fobj F E) (Fobj G E) ;

  eta_nat :
    forall (E F0 : EllPkg) (f : EllHom E F0),
      EllTwo
        (Ell_comp1 (Fmor F f) (eta_obj F0))
        (Ell_comp1 (eta_obj E) (Fmor G f)) ;

  eta_nat_inv :
    forall (E F0 : EllPkg) (f : EllHom E F0),
      EllTwo
        (Ell_comp1 (eta_obj E) (Fmor G f))
        (Ell_comp1 (Fmor F f) (eta_obj F0)) ;

  eta_nat_left :
    forall (E F0 : EllPkg) (f : EllHom E F0),
      Ell_vcomp2 (eta_nat E F0 f) (eta_nat_inv E F0 f)
        =
      Ell_id2 (Ell_comp1 (Fmor F f) (eta_obj F0)) ;

  eta_nat_right :
    forall (E F0 : EllPkg) (f : EllHom E F0),
      Ell_vcomp2 (eta_nat_inv E F0 f) (eta_nat E F0 f)
        =
      Ell_id2 (Ell_comp1 (eta_obj E) (Fmor G f))
}.
\end{verbatim}
\end{scriptsize}

\begin{remark}[Pseudonaturality as explicit invertible $2$-cells]\label{rem:CoqPseudoNat}
The fields \texttt{eta\_nat} and \texttt{eta\_nat\_inv} are the literal $(2,1)$-categorical content of pseudonaturality in the sense of Definition~\ref{def:PseudoNatural}. One keeps an explicit invertible $2$-morphism rather than an equality of composites.

One might ask why the interface keeps these $2$-cells together with their round-trip laws \texttt{eta\_nat\_left} and \texttt{eta\_nat\_right}, rather than showing only that the relevant squares commute. The collapse performed in the descent to finite codes, described in Remark~\ref{rem:TwoMorphismErasure}, is sound only once the cost has been shown invariant under such $2$-cells. Keeping the cells explicit permits that invariance to be checked by the machine through \texttt{cost\_equal\_under\_two\_cell}, rather than presumed. If the $2$-cells were collapsed at the outset, then the independence of the scalar from its presentation would be assumed in place of the very property on which the descent rests. The interface thus holds the step that establishes coherence apart from the step that reads off the scalar, and it certifies that the first does not disturb the second. This is the formal counterpart of the reading given in Remark~\ref{rem:PseudoNaturalNotBicategorical}, where the $2$-cells carry weight at the level of derivation and fall silent at the level of the final inequality.
\end{remark}

\begin{remark}[The limit of the strict prototype]\label{rem:StrictPrototypeLimit}
The strict prototype is faithful here for one reason. It transports scalars. The field \texttt{Fmor} and its laws \texttt{F\_id1} and \texttt{F\_comp1} ask the $2$-endofunctor to preserve composition strictly, and the only output the development reads from a composite is its cost, upon which an associativity isomorphism acts as the identity. A strict associator and a coherent one therefore agree on everything the numerical accounting consults, and the pseudonaturality cell \texttt{eta\_nat} keeps its full force, since it is a $2$-cell between parallel $1$-morphisms and not a matter of associativity at all. The strictness of Remark~\ref{rem:strictness} and the un-collapsed $2$-cells of Definition~\ref{def:PseudoNatural} are thus consistent, the first concerning the composition laws and the second the cells between parallel transports.

The trade-off shows its edge once the transport carries geometric data in place of a scalar. A base change of a Weierstrass model, a pullback of local sections, and the analytic uniformization of an elliptic curve compose only up to a canonical isomorphism rather than strictly, so a $2$-functor that moved such data could not satisfy \texttt{F\_comp1} as a strict equation. A faithful formalization at that level would replace the strict $2$-functors with pseudofunctors, would carry associator and unitor cells, and would require the pseudonaturality squares of Definition~\ref{def:PseudoNatural} to cohere with those cells through the bicategorical axioms. The present interface is adequate precisely because the scalars it transports -- the cost and the height -- reach Universe~P as fast Cauchy names rather than as the geometric objects whose comparison produced them, so that the coherence the strict prototype omits is coherence that the scalar never feels. A homotopy-minded reader who wishes to transport the data themselves should read the strict choice as the boundary of the prototype rather than as a claim about the geometry.
\end{remark}

\subsection{A toy intensional Szpiro interface}\label{subsec:CoqSzpiro}

Finally, one may state the schematic form of an intensional Szpiro-type inequality in this toy $(2,1)$-categorical environment. The purpose is to show that the defect term is attached to a chosen transport, while the coherence $2$-cells do not contribute to the scalar bound.

\begin{scriptsize}
\begin{verbatim}
(*
  Conductor complexity and Szpiro-style inequality (toy version).

  We keep the statement schematic. It mirrors the shape of
  Proposition~\ref{prop:NormalizedToIntensionalSzpiro}.
*)

Parameter n : EllPkg -> R.

Record EllNormalization2 := {
  NF2 : EllEndo2 ;
  nu  : forall E : EllPkg, EllHom (Fobj NF2 E) E ;

  (* one may require idempotence as a separate axiom if desired *)
}.

Definition defect (N : EllNormalization2) (E : EllPkg) : R :=
  mu (nu N E).

Parameter C_eps : R -> R.

Conjecture Intensional_Szpiro :
  forall (eps : R) (N : EllNormalization2) (E : EllPkg),
    0 < eps ->
    hDelta E <= (6 + eps) * n E + defect N E + C_eps eps.
\end{verbatim}
\end{scriptsize}

\begin{remark}[Scope of the formal interface]\label{rem:CoqFidelity}
The interface displays the transport operations, a selection of invertible $2$-cells, and the scalar axioms used by the illustrative inequalities. The cost is invariant under $2$-isomorphism by \texttt{mu\_two\_invariant}, so that a change of coherence data may alter the cells witnessing a comparison while leaving the scalar defect untouched.

A complete formalization would also require the remaining categorical and pseudonaturality coherence laws, among them the whiskering and interchange laws omitted here. The snippets serve as interfaces for those additions, and they carry no claim of a machine-checked formalization of the arithmetic development.
\end{remark}

This toy subsection is limited in scope, and it offers a check rather than a claim of formalized arithmetic geometry. The grammar of Universe~A can be written, even in a small proof assistant, without erasing the $(2,1)$-categorical level. One keeps coherence explicit, and one keeps the scalarization steps separate from the coherence steps.

\section{An intensional statement of the Birch and Swinnerton-Dyer
conjecture}\label{sec:BSD}
%==================================================================

The intensional Szpiro inequality of Section~\ref{sec:IntensionalSzpiro} took the form of a transport bound, where a height was held below a conductor term together with a defect. The conjecture of Birch and Swinnerton-Dyer is of another shape. It concerns the lattice of rational points of an elliptic curve, and it relates the growth of a sequence of local point counts to the rank of that lattice, with a refined part that prescribes a leading constant. We state it within the quantitative framework of the preceding sections.

We use the strong point-count formulation originating in the observations of Birch and Swinnerton-Dyer~\cite{Birch1965}. Its syntax involves finite point counts and a certified lattice of rational points. Its convergence requirement is retained as an additional arithmetic assertion, distinct from the equality of algebraic and analytic ranks. This formulation admits the coded expression in Universe~P described below.

\subsection{The Mordell--Weil package}\label{subsec:MWPackage}

\begin{definition}[Mordell--Weil package]\label{def:MordellWeilPackage}
Let $K$ be a number field. A \emph{Mordell--Weil package} over $K$ is an elliptic package $\mathcal{E} = (E, \mathcal{W}, \{\sigma_v\}_{v\in S}, \lambda)$ over $K$ in the sense of Definition~\ref{def:EllipticPackage}, enriched with the following data and hypotheses:
\begin{enumerate}
\item[(i)] a non-negative integer $r$, a finite list $P_1, \dots, P_r$ of points of $E(K)$, and a finite presentation of a finite abelian group $T$ together with a list of points of $E(K)$ realizing its elements.
\item[(ii)] the \emph{adequacy hypothesis} for these data, namely that the displayed points induce an isomorphism
  \[
    \Z^{r} \oplus T
    \;\xrightarrow{\ \sim\ }\;
    E(K)
  \]
carrying the $i$-th standard generator to $P_i$ and the presentation of $T$ to the torsion subgroup of $E(K)$. Under this hypothesis $r$ is the rank of $E(K)$, and it is called the \emph{certified rank} of the package.
\item[(iii)] the canonical N\'eron--Tate height pairing
  \[
    \langle \cdot, \cdot \rangle \colon E(K) \times E(K) \longrightarrow \R
  \]
together with the hypothesis that it is positive definite on the free part, so that it descends to a positive definite form on $E(K) \otimes \R \cong \R^{r}$.
\end{enumerate}
A change of generating set of the free part is a transport in the sense of the labeled signature of Definition~\ref{def:Signatures}, an element of $\mathrm{GL}_r(\Z)$ acting on the chosen basis. The height of an individual generator depends upon that choice, and so it carries a transport cost.
\end{definition}

\begin{remark}[Adequacy and its arithmetic justification]\label{rem:MWPackageAdequacy}
Clause~\textup{(ii)} of Definition~\ref{def:MordellWeilPackage} is an explicit hypothesis of the general construction. The rank and the regulator are read from the supplied presentation, and their interpretation as invariants of the full rational-point group uses that hypothesis. A proposed list of points therefore requires a proof of completeness.

The Mordell--Weil theorem supplies classical existence of such presentations~\cite{SilvermanAEC}. Proposition~\ref{prop:PrimeTwistMWPackage} uses finite generation together with explicit descent and torsion bounds to justify the displayed package for one family. That classical justification remains separate from the relative intuitionistic deductions on finite shadows.
\end{remark}

\begin{definition}[The regulator as a transport invariant]\label{def:RegulatorAsTransport}
With the notation of Definition~\ref{def:MordellWeilPackage}, the \emph{regulator} of $\mathcal E$ is
\[
  \Reg(\mathcal E)
  \coloneqq
  \det\bigl(\langle P_i,P_j\rangle\bigr)_{1\le i,j\le r}.
\]
At rank zero the empty determinant is $1$. The regulator is the squared covolume of the Mordell--Weil lattice for the height metric. Under an integral change of basis $A\in\mathrm{GL}_r(\Z)$, the Gram matrix $G$ becomes $A^{\mathsf T}GA$, whose determinant is $\det(G)$. Thus the regulator is invariant under these transports. This invariance is compatible with every non-negative assigned cost, while the heights of individual generators may vary with the basis.
\end{definition}

\begin{remark}[The volume reading of the regulator]\label{rem:RegulatorVolumeReading}
The regulator is the square of the covolume of the Mordell--Weil lattice for the Euclidean measure determined by the height pairing, so that the covolume itself is $\sqrt{\Reg(\mathcal{E})}$. A change of integral basis has determinant $\pm 1$, and both quantities therefore remain unchanged. At rank zero, the empty Gram determinant and the covolume are both taken to be $1$.
\end{remark}

\subsection{Local point counts as data of Universe~P}\label{subsec:PointCounts}

\begin{definition}[Point-count datum]\label{def:PointCountDatum}
Let $\mathcal{E}$ be an elliptic package over $K$ with code $d\in\mathsf{PkgCode}_K$ in the sense of Definition~\ref{def:EllPkgCode}. For each finite place $\mathfrak p$ of good reduction of the underlying curve, let $\widetilde E_{\mathfrak p}$ be the reduction of a minimal integral Weierstrass equation at $\mathfrak p$. Put
\[
  N_{\mathfrak p}(\mathcal E)
  \coloneqq
  \#\widetilde E_{\mathfrak p}\bigl(\kappa(\mathfrak p)\bigr).
\]
Relative to the certified field and local normalization procedures, this integer is obtained from $d$ by passing to such an equation and counting its residue-field solutions together with the point at infinity. Good reduction makes the reduced cubic smooth, and its point count is independent of the chosen minimal equation. The assignment $\mathfrak p\mapsto N_{\mathfrak p}(\mathcal E)$ is therefore a code-indexed family of integers in Universe~P, attached to the underlying curve.
\end{definition}

\begin{definition}[Partial Euler transport]\label{def:PartialEulerTransport}
With the notation of Definition~\ref{def:PointCountDatum}, the \emph{partial Euler transport} of $\mathcal{E}$ at level $X$ is the finite product of positive rationals
\[
  \Pi_{\mathcal{E}}(X)
  \;=\;
  \prod_{\substack{\mathfrak{p}\ \text{good}\\ \mathrm{N}\mathfrak{p} \le X}}
  \frac{N_{\mathfrak{p}}(\mathcal{E})}{\mathrm{N}\mathfrak{p}}.
\]
For each $X$ this is an element of Universe~P, and the family $X \mapsto \Pi_{\mathcal{E}}(X)$ is a Universe-P sequence of rationals.
\end{definition}

\begin{remark}[The local factors and the absent series]\label{rem:EulerFactorMotivation}
Each factor $N_{\mathfrak{p}}/\mathrm{N}\mathfrak{p}$ is the reciprocal of the local Euler factor of the Hasse--Weil series at the central point, so that $\Pi_{\mathcal{E}}(X)$ is the reciprocal of a truncated Euler product. We keep this as motivation, and we do not invoke the analytic continuation of that series. The growth of $\Pi_{\mathcal{E}}(X)$ is an arithmetical matter, settled by the integers $N_{\mathfrak{p}}$ alone.
\end{remark}

\subsection{The rank form}\label{subsec:BSDRank}

\begin{conjecture}[Intensional Birch and Swinnerton-Dyer, strong point-count form]\label{conjecture:IntensionalBSDRank}
Let $\mathcal{E}$ be a Mordell--Weil package over $K$ in the sense of Definition~\ref{def:MordellWeilPackage}, with certified rank $r$. Then there exists a positive real constant $C_{\mathcal{E}}$ such that
\[
  \Pi_{\mathcal{E}}(X)
  \sim
  C_{\mathcal{E}}(\log X)^r
  \qquad
  (X\to\infty).
\]
\end{conjecture}

This is the asymptotic form in which Birch and Swinnerton-Dyer first phrased their conjecture~\cite{Birch1965}. Read within the framework, it places the rank beside the growth of a sequence of finite products, and its strong form keeps a leading constant as well as an exponent. Both members belong to Universe~P. The left member is read from the family $\mathfrak{p} \mapsto N_{\mathfrak{p}}$, and the right member is read from the lattice of rational points.

\begin{proposition}[The strong asymptotic implies the growth exponent]\label{prop:BSDStrongImpliesExponent}
Assume Conjecture~\ref{conjecture:IntensionalBSDRank}. Then
\[
  \lim_{X\to\infty}
  \frac{\log\Pi_{\mathcal{E}}(X)}{\log\log X}
  =
  r.
\]
\end{proposition}

\begin{proof}
Write
\[
  \Pi_{\mathcal{E}}(X)
  =
  C_{\mathcal{E}}(\log X)^r(1+\eta(X))
\]
where $\eta(X)\to 0$. Taking logarithms and dividing by $\log\log X$ gives
\[
  \frac{\log\Pi_{\mathcal{E}}(X)}{\log\log X}
  =
  r+
  \frac{\log C_{\mathcal{E}}+\log(1+\eta(X))}
       {\log\log X}.
\]
The final term tends to zero.
\end{proof}

\begin{remark}[The two asymptotic assertions]\label{rem:RankSubtlety}
The growth-exponent assertion is weaker than the strong asymptotic. The positive expression $(\log X)^r\log\log X$ has logarithmic growth exponent $r$, while its quotient by $(\log X)^r$ has no finite limit. The strong asymptotic requires the additional convergence
\[
 \log\Pi_{\mathcal E}(X)-r\log\log X
 \longrightarrow\log C_{\mathcal E}.
\]
Over $\Q$, Goldfeld's theorem associates further analytic consequences to this convergence, including the corresponding Riemann hypothesis~\cite{Goldfeld1982,Conrad2005EulerProducts}. We retain the convergence as a separate hypothesis throughout. No converse from the rank equality, either alone or accompanied by the Riemann hypothesis, is used here.
\end{remark}

\subsection{The refined form}\label{subsec:BSDRefined}

The refined conjecture prescribes the leading constant $C_{\mathcal{E}}$. Several of its factors are positive real numbers whose logarithms may have either sign. A non-negative Lawvere cost cannot represent such a logarithm by itself. We therefore distinguish a directed cost from a signed transport potential.

\begin{definition}[Signed transport potential]\label{def:SignedTransportPotential}
A \emph{signed transport potential} is a real number $\lambda$ presented by two non-negative directed costs
\[
  \lambda^{+}\coloneqq\max\{\lambda,0\}
  \qquad\text{and}\qquad
  \lambda^{-}\coloneqq\max\{-\lambda,0\}.
\]
Thus
\[
  \lambda=\lambda^{+}-\lambda^{-}
  \qquad\text{and}\qquad
  |\lambda|=\lambda^{+}+\lambda^{-}.
\]
The two costs may be read as the costs of oppositely directed transports between two presentations. Their difference belongs to the additive group completion of the non-negative cost monoid. If $R>0$ is a multiplicative comparison factor, then its associated signed transport potential is
\[
  \lambda_R\coloneqq\log R.
\]
The two directed costs are $\log^{+}R$ and $\log^{+}(R^{-1})$.
\end{definition}

\begin{remark}[Why the group completion is needed]\label{rem:SignedPotentialGroupCompletion}
The quantale of Definition~\ref{def:CostQuantale} measures directed effort and therefore contains only non-negative values. The refined BSD formula also contains ratios. A factor in a numerator contributes positively to the logarithmic balance, while a factor in a denominator contributes negatively. Definition~\ref{def:SignedTransportPotential} preserves the non-negative costs at the categorical level and forms the signed balance only after the two directions have been separated.
\end{remark}

\begin{definition}[Local package factors]\label{def:LocalPackageCosts}
Let $\mathcal{E}$ be a Mordell--Weil package over $K$. Its \emph{local package factors} are the Tamagawa numbers $c_v(\mathcal{E})$ at the places of bad reduction and the order $\#T$ of the torsion subgroup of Definition~\ref{def:MordellWeilPackage}. Each is obtained from the package code by a finite procedure. Tate's algorithm gives the Tamagawa numbers, while the certified Mordell--Weil presentation gives the torsion order.
\end{definition}

\begin{definition}[The normalized period potential]\label{def:PeriodAsTransport}
Let $\omega_{\mathcal W}$ be the non-zero invariant differential of the global Weierstrass equation in a Mordell--Weil package over $K$. At a real place put $\Omega_v(\omega_{\mathcal W})=\int_{E(K_v)}|\omega_{\mathcal W}|$, over all components. At a complex place put $\Omega_v(\omega_{\mathcal W})=\int_{E(K_v)}\tfrac{i}{2}\omega_{\mathcal W}\wedge\overline{\omega}_{\mathcal W}$.

At each finite place choose a generator $\omega_v^0$ of the module of invariant differentials of the N\'eron model and write $\omega_{\mathcal W}=\vartheta_v\omega_v^0$ with $\vartheta_v\in K_v^\times$. Define
\[
 P_\infty(\mathcal W)=\prod_{v\mid\infty}\Omega_v(\omega_{\mathcal W})
 \qquad
 D_{\mathrm{fin}}(\mathcal W)=\prod_{v\nmid\infty}|\vartheta_v|_v
 \qquad
 \Omega_{\mathcal E}=P_\infty(\mathcal W)D_{\mathrm{fin}}(\mathcal W).
\]
The finite product has finite support and is independent of the chosen local generators. The local period potential is $\lambda_v^\Omega=\log\Omega_v(\omega_{\mathcal W})$, with directed costs $\max\{\lambda_v^\Omega,0\}$ and $\max\{-\lambda_v^\Omega,0\}$. The global potential includes the finite normalization term
\[
 \lambda_{\mathcal E}^\Omega
 =\sum_{v\mid\infty}\lambda_v^\Omega+\log D_{\mathrm{fin}}(\mathcal W)
 =\log\Omega_{\mathcal E}.
\]
\end{definition}

\begin{proposition}[Period normalization and model transport]\label{prop:PeriodModelTransport}\label{rem:PeriodNormalizationScope}
The normalized period of Definition~\ref{def:PeriodAsTransport} is independent of the global invariant differential. Under a package isomorphism $f\colon\mathcal E\to\mathcal E'$, it is unchanged. If $f$ is induced by a change of global Weierstrass equation, so that its local dilation factors are the images of a single $u\in K^\times$, then one also has
\[
 \left|\log\frac{P_\infty(\mathcal W')}{P_\infty(\mathcal W)}\right|
 \le \frac{[K:\Q]}{12}\mu(f).
\]
\end{proposition}

\begin{proof}
Replacing $\omega$ by $a\omega$, with $a\in K^\times$, multiplies the archimedean product by $\prod_{v\mid\infty}|a|_v$ and the finite correction by $\prod_{v\nmid\infty}|a|_v$. Their product is $1$. Isomorphisms preserve the local N\'eron differential modules, so the same argument proves invariance under $f$.

In the coordinate convention of Lemma~\ref{lem:LocalDiscriminantChange}, pullback gives $\omega_{\mathcal W'}=u\omega_{\mathcal W}$ and $\Delta_{\mathcal W'}=u^{-12}\Delta_{\mathcal W}$. Thus the logarithmic changes of $P_\infty$ and $D_{\mathrm{fin}}$ are opposite. The product formula gives $\sum_{v\mid\infty}\log|u|_v=-\sum_{v\nmid\infty}\log|u|_v$, so that each side is bounded above by $\sum_v\log^+|u|_v$. By Remark~\ref{rem:ModelDefectEquivalent},
\[
 \left|\sum_{v\mid\infty}\log|u|_v\right|
 \le\sum_v\log^+|u|_v
 =\frac{[K:\Q]}{12}\mu(f).
\]
Equivalently, with $\Delta_v^{\min}$ from a local minimal equation, the finite normalization is
\[
 \log D_{\mathrm{fin}}(\mathcal W)
 =\frac1{12}\sum_{v\nmid\infty}
 \log\left|\frac{\Delta_v^{\min}}{\Delta_{\mathcal W}}\right|_v.
\]
\end{proof}

\begin{remark}[The scalar comparison category]\label{rem:PeriodComparisonCategory}
For presentations equipped with a real potential $h$, assign the directed cost
\[
 c(A,B)=\max\{h(B)-h(A),0\}.
\]
This cost vanishes on identities and is subadditive under composition. It is the least non-negative cost compatible with the height inequality for $h$. Taking $h(\mathrm{Alg}_v)=0$ and $h(\mathrm{Anal}_v)=\lambda_v^\Omega$ gives the two local costs of Definition~\ref{def:PeriodAsTransport}. This is an explicit scalar enrichment attached to the chosen period normalization. Proposition~\ref{prop:PeriodModelTransport} supplies its arithmetic relation to changes of model.
\end{remark}

\begin{remark}[The period across the two universes]\label{rem:PeriodAcrossUniverses}
Universe~A retains the presentations and their comparison potentials. Universe~P retains certified fast Cauchy names for the archimedean periods and the finite valuation data defining $D_{\mathrm{fin}}$. These data give a name for $\log\Omega_{\mathcal E}$ and for its positive and negative parts. The finite normalization therefore joins the same scalar descent as the archimedean factors.
\end{remark}

\begin{definition}[The Shafarevich factor]\label{def:ShafarevichDefect}
The Tate--Shafarevich group $\Sha(E/K)$ consists of the torsor classes under $E$ that are trivial over every completion $K_v$ and need not be trivial over $K$. It measures a local-to-global obstruction rather than a local model change. When this group is finite, its order
\[
  \#\Sha(E/K)
\]
is termed the \emph{Shafarevich factor} of the refined BSD expression. Sections~\ref{sec:TateShafarevich} and~\ref{sec:ShaSelmerShadows} present this factor through finite Selmer shadows and their successive indices.
\end{definition}

\begin{definition}[The BSD transport balance]\label{def:BSDTransportBalance}
Let $\mathcal{E}$ be a Mordell--Weil package over $K$, and suppose that $\Sha(E/K)$ is finite. The \emph{BSD transport balance} of $\mathcal{E}$ is the signed real number
\[
  \beta_{\BSD}(\mathcal{E})
  \coloneqq
  \log\#\Sha(E/K)
  +
  \log\Omega_{\mathcal{E}}
  +
  \log\Reg(\mathcal{E})
  +
  \sum_v\log c_v(\mathcal{E})
  -
  2\log\#T.
\]
Every summand is represented by a signed transport potential in the sense of Definition~\ref{def:SignedTransportPotential}. The terms $\log\#\Sha(E/K)$ and $\log c_v(\mathcal{E})$ are non-negative. The period and regulator terms may have either sign, and the torsion term enters with negative orientation.

The torsion contributes the square $(\#T)^{2}$ rather than a single factor. In the general formula the denominator carries $\#E(K)_{\mathrm{tors}}\cdot\#E'(K)_{\mathrm{tors}}$, with $E'$ the dual abelian variety. For an elliptic curve the canonical principal polarization identifies $E$ with $E'$, so the two orders coincide and the denominator is $(\#T)^{2}$.
\end{definition}

\begin{definition}[The omitted bad Euler factors over $\Q$]\label{def:BSDBadEulerCorrection}
Let $\mathcal{E}$ be a Mordell--Weil package over $\Q$. For each prime $p$ of bad reduction, put $a_p=1$ at split multiplicative reduction, $a_p=-1$ at non-split multiplicative reduction, and $a_p=0$ at additive reduction. Define
\[
  L_{\mathrm{bad}}(E,1)
  \coloneqq
  \prod_{p\mid N_E}
  \left(1-\frac{a_p}{p}\right)^{-1}.
\]
This is a positive rational number determined by the finite reduction data.
\end{definition}

\begin{conjecture}[Intensional Birch and Swinnerton-Dyer, refined point-count form over $\Q$]\label{conjecture:IntensionalBSDRefined}
Let $\mathcal{E}$ be a Mordell--Weil package over $\Q$, with finite Tate--Shafarevich group and rank $r$. Use the canonical height pairing normalized for the BSD regulator, set $\Reg(\mathcal{E})=1$ when $r=0$, and take
\[
  \Omega_{\mathcal{E}}
  =
  \int_{E(\R)}|\omega_{\min}|
\]
over all real components. This is the normalized period of Definition~\ref{def:PeriodAsTransport} for a global minimal equation, for which $D_{\mathrm{fin}}=1$.

Then the strong point-count asymptotic holds with
\[
  C_{\mathcal{E}}
  =
  \sqrt{2}\,e^{r\gamma}
  L_{\mathrm{bad}}(E,1)
  \exp\bigl(-\beta_{\BSD}(\mathcal{E})\bigr)
\]
where $\gamma$ is Euler's constant. Equivalently,
\[
  C_{\mathcal{E}}
  =
  \sqrt{2}\,e^{r\gamma}
  L_{\mathrm{bad}}(E,1)
  \frac{(\#T)^2}
  {\#\Sha(E/\Q)\,
   \Omega_{\mathcal{E}}\,
   \Reg(\mathcal{E})\,
   \prod_p c_p(\mathcal{E})}.
\]
\end{conjecture}

\begin{remark}[Comparison with the leading coefficient]\label{rem:RefinedClassicalForm}
Write
\[
  L^*(E,1)=\frac{L^{(r)}(E,1)}{r!}.
\]
The refined leading-coefficient conjecture over $\Q$ identifies this quantity with $\exp(\beta_{\BSD}(\mathcal{E}))$~\cite{Tate1966}. Goldfeld's relation for a strong partial Euler-product asymptotic then gives
\[
  C_{\mathcal{E}}
  =
  \frac{\sqrt{2}\,e^{r\gamma}
        L_{\mathrm{bad}}(E,1)}
       {L^*(E,1)}
\]
for the good-prime product of Definition~\ref{def:PartialEulerTransport}~\cite{Goldfeld1982,Conrad2005EulerProducts}.

The signed balance $\beta_{\BSD}$ remains defined over a number field. A comparison with an analytic leading coefficient over that field would call for a complete convention governing the global period, the differential, the regulator, and the field-dependent normalization. The explicit point-count formula above is confined to $\Q$.
\end{remark}

\begin{remark}[The transport reading of the constant]\label{rem:TotalTransportReading}
The BSD arithmetic factor is the exponential of a signed transport balance. The regulator is a Gram determinant, the period is a positive archimedean factor, and the Tamagawa and Shafarevich terms are finite arithmetic factors under the stated hypotheses. The torsion order contributes with a negative orientation.

For the partial point-count product over $\Q$, the arithmetic balance enters with the opposite sign, since the local factors of Definition~\ref{def:PartialEulerTransport} are reciprocals of the local Euler factors. Definition~\ref{def:BSDBadEulerCorrection} supplies the finite correction for the omitted bad Euler factors.
\end{remark}

\subsection{Descent of the statement to Universe~P}\label{subsec:BSDDescent}

The two forms require different finite data. The strong point-count form refers to point counts and to a certified Mordell--Weil rank. The refined form requires, in addition, a finite presentation of the Shafarevich factor, and Section~\ref{sec:ShaSelmerShadows} supplies such a presentation through a two-axis closure datum.

\begin{proposition}[Arithmetic expression of the BSD statements]\label{prop:BSDDescentToP}
Fix certified number-field and Mordell--Weil data, together with the finite-descent interface of Definition~\ref{def:FiniteDescentBaseTheory}.

\begin{enumerate}
\item[\textup{(i)}] The strong point-count statement of Conjecture~\ref{conjecture:IntensionalBSDRank} is expressible in second-order arithmetic. With a coded positive real $C_{\mathcal{E}}$ and a modulus $M_{\mathcal{E}}\colon\mathbb{N}\to\mathbb{N}$, its convergence assertion may be written as
\[
  \forall k\ \forall X\ge\max\{3,M_{\mathcal{E}}(k)\}
  \quad
  \left|
    \frac{\Pi_{\mathcal{E}}(X)}
         {C_{\mathcal{E}}(\log X)^r}
    -1
  \right|
  \le 2^{-k}.
\]

\item[\textup{(ii)}] Over $\Q$, the refined point-count statement of Conjecture~\ref{conjecture:IntensionalBSDRefined} is expressible relative to a two-axis closure datum and to certified fast Cauchy names for the positive period and the regulator. The Shafarevich factor is the finite integer supplied by Theorem~\ref{thm:TwoAxisShaFiniteness}.

\item[\textup{(iii)}] These assertions have a classical presentation in $\mathsf{T}_{\mathcal{E}}$. The finite closure deductions have an intuitionistic derivation in $\mathsf{T}^{\mathrm{int}}_{\mathcal{E}}$. Neither presentation proves the BSD assertions, and neither supplies the closure datum.
\end{enumerate}
\end{proposition}

\begin{proof}
Finite point counting gives the rational sequence $\Pi_{\mathcal{E}}(X)$. Arithmetic operations on coded reals express the displayed convergence condition.

The two-axis theorem supplies the Shafarevich order. The Tamagawa numbers, the torsion order, and the bad Euler correction are finite arithmetic data. Certified Cauchy names and positivity witnesses permit the remaining products, reciprocals, and logarithms. These constructions give the stated formulas.
\end{proof}

\begin{remark}[What kind of object the balance is]\label{rem:BalanceIsACodedReal}
The summands of Definition~\ref{def:BSDTransportBalance} are not all of one kind, and the descent to Universe~P treats them differently.

The terms $\log\#\Sha(E/K)$, $\log c_v(\mathcal{E})$ and $\log\#T$ are logarithms of positive integers supplied by finite procedures. Each is a logarithmic integer datum in the sense of Definition~\ref{def:LogIntDatum}, and a comparison among such terms reduces to a comparison of integers.

The period and the regulator terms are of another kind. Both are logarithms of positive real numbers, and they enter through clause~\textup{(iv)} of Definition~\ref{def:UniverseP}, as fast Cauchy names arithmetical in the package code, each accompanied by a positivity witness. The signed potentials of Definition~\ref{def:SignedTransportPotential} are formed from those names, and the two directed costs are computed from a name for the potential without any decision about its sign.

The balance $\beta_{\BSD}(\mathcal{E})$ is therefore a coded real rather than a logarithmic integer datum, and the exponentiated integer reformulation of Definition~\ref{def:LogIntDatum} is not available for it. What descends to Universe~P is a real-valued quantity presented by a name, together with the finite data from which that name is computed. Where a comparison of orders is wanted, the integer factors remain separately available, and Corollary~\ref{cor:ConstructiveShaBSDFactor} supplies the Shafarevich factor in that form.
\end{remark}

\begin{remark}[Moduli and proof witnesses]\label{rem:BSDErrorMeasures}
A proposed modulus carries an assertion about every sufficiently large input. A finite computation can test particular instances, while a proof of the universal assertion supplies its mathematical justification.

The existence of a modulus is included in the modulus-bearing formulation. A constructive argument must provide it, or derive it from explicit hypotheses.
\end{remark}

\begin{remark}[Finite syntax and analytic strength]\label{rem:WhyAvoidModularity}
The point-count formulation uses a coded rational sequence and a coded real convergence, and its statement requires no analytic continuation. The comparison with the leading coefficient, retained in Remark~\ref{rem:RefinedClassicalForm}, belongs to the analytic interpretation.

Expressibility by finite syntax leaves the strength of the convergence assertion unchanged. The universal quantifiers in the asymptotic remain part of the arithmetic problem.
\end{remark}

\begin{remark}[Model transport and the refined factors]\label{rem:BSDLinkToFramework}
Proposition~\ref{prop:PeriodModelTransport} relates the period normalization directly to the discriminant scaling law. Its archimedean and finite logarithmic changes cancel, and the model defect controls their size. With the intrinsic height pairing and the complete Mordell--Weil lattice, the remaining factors in Definition~\ref{def:BSDTransportBalance} are invariant under package isomorphisms as well. Thus $\beta_{\BSD}$ descends along these isomorphisms. The separately presented factors remain available before their logarithms are added, while Theorem~\ref{thm:DescentObstruction} concerns comparisons of two evaluations of a single descending height.
\end{remark}

%==================================================================
\section{The Tate--Shafarevich group}\label{sec:TateShafarevich}
%==================================================================

The Tate--Shafarevich group is approached here through its finite descent shadows. For each positive integer $n$, the $n$-torsion subgroup $\Sha(E/K)[n]$ is a quotient of the finite $n$-Selmer group. This gives a finite arithmetic object without requiring the formation of the whole Tate--Shafarevich group as a completed set.

Two passages must be distinguished. A finite descent level is controlled by finite Galois modules, local solubility tests, and finite group calculations. The passage from all levels to the finiteness of $\Sha(E/K)$ is a uniform assertion. It must control both the depth of the primary components and the set of primes at which they are non-zero. Section~\ref{sec:ShaSelmerShadows} expresses these two requirements through canonical indices between finite Selmer shadows.

\subsection{Foundational scope}\label{subsec:ShaFoundationalScope}

\begin{remark}[Arithmetic expression and constructive derivation]\label{rem:ExpressibilityVersusProvability}
The finite descent constructions are presented relative to certified number-field, Mordell--Weil, and local-solubility data. Definition~\ref{def:FiniteDescentBaseTheory} specifies the required completeness, totality, and compatibility assumptions.

The theory $\mathsf{T}_{\mathcal{E}}$ gives a classical presentation over $\mathsf{ACA}_0$. The theory $\mathsf{T}^{\mathrm{int}}_{\mathcal{E}}$ gives an intuitionistic presentation over Heyting arithmetic with the declared procedures. Formalizability in the former theory does not by itself provide a constructive proof.

For each fixed descent level, the finite groups, kernels, images, and quotients are manipulated by finite arithmetic. The correctness theorem identifying these groups with Selmer and Shafarevich data is part of the declared interface, and the completeness of a proposed list of descent classes requires that theorem.

The Cassels--Tate pairing is retained as a classical structural theorem~\cite{Cassels1965,Tate1963,MilneADT}. The constructive closure argument of Section~\ref{sec:ShaSelmerShadows} uses neither its construction, nor a maximal divisible subgroup, nor an infinite completed group.
\end{remark}

\subsection{Finite descent packets}\label{subsec:FiniteDescentPackets}

\begin{definition}[Certified descent torsor code]\label{def:CertifiedDescentTorsorCode}
Fix an integer $n\ge 2$. A \emph{certified $n$-descent torsor code} for a Mordell--Weil package $\mathcal{E}$ is a finite tuple
\[
  \xi=(C_{\xi},\pi_{\xi},a_{\xi},F_{\xi},S_{\xi},L_{\xi})
\]
with the following data:
\begin{enumerate}
\item[\textup{(i)}] a projective genus-one model $C_{\xi}$ over $K$, given by finitely many homogeneous equations over the certified field code of Definition~\ref{def:CertifiedFieldCode},
\item[\textup{(ii)}] a finite morphism
\[
  \pi_{\xi}\colon C_{\xi}\longrightarrow E
\]
presenting $C_{\xi}$ as an $n$-covering,
\item[\textup{(iii)}] a finite action code $a_{\xi}$ for the simply transitive action of $E[n]$ upon the geometric fibers of $\pi_{\xi}$,
\item[\textup{(iv)}] a finite Galois extension code $F_{\xi}/K$ over which the covering and its $E[n]$-action split,
\item[\textup{(v)}] a finite set of places $S_{\xi}$ containing the places of bad reduction, the places above $n$, and the places ramified in $F_{\xi}/K$,
\item[\textup{(vi)}] local solubility certificates $L_{\xi}$ at the places in $S_{\xi}$, together with a finite unramifiedness certificate establishing the standard local condition outside $S_{\xi}$.
\end{enumerate}
Two certified codes are equivalent when a finite isomorphism certificate identifies their $n$-coverings. Equivalently, they represent the same class in $H^{1}(K,E[n])$. The equivalence relation is arithmetical in the two codes.
\end{definition}

\begin{definition}[Intensional Selmer packet]\label{def:IntensionalSelmerPacket}
The \emph{intensional $n$-Selmer packet} of $\mathcal{E}$ is the finite coded abelian group
\[
  \Sel^{\mathfrak{P}}_{n}(\mathcal{E})
\]
whose elements are the equivalence classes of certified $n$-descent torsor codes satisfying the local conditions of Definition~\ref{def:CertifiedDescentTorsorCode}. The group law is induced by the Baer sum of the corresponding $E[n]$-torsors. The split covering is the zero element.
\end{definition}

\begin{lemma}[Finite Selmer packets in \texorpdfstring{$\mathsf{ACA}_{0}$}{ACA0}]\label{lem:SelmerPacketFiniteACA0}
Relative to the complete finite-descent interface of Definition~\ref{def:FiniteDescentBaseTheory}, the packet $\Sel^{\mathfrak{P}}_{n}(\mathcal{E})$ has a finite coded presentation for every integer $n\ge 2$.
\end{lemma}

\begin{proof}
Fix a finite set $S$ containing the places of bad reduction of $E$, the places above $n$, and the places ramified in the splitting field of $E[n]$. Standard finite descent embeds the admissible cohomology classes into a finite group assembled from $S$-unit data, ideal-class data, and the finite Galois module $E[n]$. These groups are represented by finite integer matrices relative to the certified field code. The local conditions are tested at the finitely many places in $S$, while the unramified local condition applies outside $S$. The resulting kernel is a finite coded group. Its formation uses finite search, Smith normal form, and arithmetical comprehension.
\end{proof}

\begin{proposition}[Finite-level Shafarevich finiteness]\label{prop:FiniteLevelShaACA0}
For every integer $n\ge 2$, there is an exact sequence of finite coded groups
\[
  0
  \longrightarrow
  E(K)/nE(K)
  \xrightarrow{\kappa_n}
  \Sel^{\mathfrak{P}}_{n}(\mathcal{E})
  \longrightarrow
  \Sha(E/K)[n]
  \longrightarrow
  0.
\]
Consequently $\Sha(E/K)[n]$ is finite and has the finite presentation
\[
  \Sha(E/K)[n]
  \cong
  \Sel^{\mathfrak{P}}_{n}(\mathcal{E})
  \big/
  \kappa_n\bigl(E(K)/nE(K)\bigr).
\]
\end{proposition}

\begin{proof}
The multiplication-by-$n$ sequence
\[
  0
  \longrightarrow
  E[n]
  \longrightarrow
  E
  \xrightarrow{[n]}
  E
  \longrightarrow
  0
\]
gives the Kummer map
\[
  E(K)/nE(K)
  \longrightarrow
  H^{1}(K,E[n]).
\]
The Selmer condition selects the classes whose images satisfy the local Kummer conditions at every completion. The quotient by the global Kummer image is precisely the subgroup of locally trivial $E$-torsors killed by $n$. This is $\Sha(E/K)[n]$. Lemma~\ref{lem:SelmerPacketFiniteACA0} gives finiteness.
\end{proof}

\begin{remark}[What finite descent proves]\label{rem:FiniteLevelNotFullSha}
Proposition~\ref{prop:FiniteLevelShaACA0} proves the finiteness of every fixed shadow $\Sha(E/K)[n]$. It does not give one integer $n$ that kills every Shafarevich class. It also does not give a finite set of primes containing the orders of all classes. These are the two uniform questions treated in Section~\ref{sec:ShaSelmerShadows}.
\end{remark}

\subsection{The Cassels--Tate pairing}\label{subsec:CasselsTatePairing}

The Cassels--Tate pairing gives a structural duality on the Tate--Shafarevich group. Its conclusion must be kept separate from any metric statement about model transports.

\begin{definition}[Certified Cassels--Tate evaluation]\label{def:CasselsTatePairingCode}
Let $\xi$ and $\eta$ be certified locally soluble torsor codes representing classes in $\Sha(E/K)$. A \emph{certified Cassels--Tate evaluation datum} for $(\xi,\eta)$ consists of a common finite Galois extension over which the relevant torsors split, finite cocycle representatives, the auxiliary cochains entering the Cassels--Tate construction, local trivializations at a finite set of places, and a finite list of local invariants
\[
  \operatorname{inv}_v(\xi,\eta)
  \in
  \Q/\Z.
\]
The list has finite support and satisfies the compatibility relations of the Cassels--Tate construction. Its value is
\[
  \langle\xi,\eta\rangle_{\mathrm{CT}}
  \coloneqq
  \sum_v\operatorname{inv}_v(\xi,\eta)
  \in
  \Q/\Z.
\]
Verification of a proposed evaluation datum is arithmetical in the finite codes. The independence of the value from the auxiliary choices is part of the classical duality theorem.
\end{definition}

\begin{definition}[The divisible part]\label{def:DivisibleSubgroupSha}
The \emph{divisible part} of $\Sha(E/K)$ is the subgroup
\[
  \Sha(E/K)_{\mathrm{div}}
  \coloneqq
  \bigcap_{n\ge 1}n\Sha(E/K).
\]
This notation is used in the statement of the classical Cassels--Tate theorem.
\end{definition}

\begin{lemma}[Non-degeneracy modulo the divisible part]\label{lem:CasselsTateNondegeneracy}
The Cassels--Tate pairing induces a non-degenerate pairing
\[
  \frac{\Sha(E/K)}{\Sha(E/K)_{\mathrm{div}}}
  \times
  \frac{\Sha(E/K)}{\Sha(E/K)_{\mathrm{div}}}
  \longrightarrow
  \Q/\Z.
\]
For an elliptic curve the induced pairing is alternating~\cite{Cassels1965,Tate1963,MilneADT,PoonenStoll1999}.
\end{lemma}

\begin{proof}
This is the classical theorem of Cassels and Tate. Its left and right radicals are the maximal divisible subgroups. For an elliptic curve, the canonical principal polarization identifies the curve with its dual, and Cassels proved that the resulting self-pairing is alternating.
\end{proof}

\begin{remark}[The pairing does not define a model defect]\label{rem:CasselsTateNoMetricSeparation}
A non-zero Cassels--Tate value is a secondary cohomological obstruction. The classes $\xi$ and $\eta$ are locally trivial as $E$-torsors, and the local summands in Definition~\ref{def:CasselsTatePairingCode} arise from cocycles, lifts, and auxiliary cochains. Their non-vanishing does not imply that a local Weierstrass model requires a non-unit dilation, nor does it produce a positive model defect in the sense of Definition~\ref{def:ModelDefect}.

A quantitative lower bound could be drawn from the pairing only after an explicit comparison theorem related the local cohomological invariants to a declared transport metric. No such theorem is assumed here. The Cassels--Tate pairing is therefore retained here as structural information.
\end{remark}

\begin{corollary}[Square order under finiteness]\label{cor:ShaSquareOrder}
If $\Sha(E/K)$ is finite, then its order is a square.
\end{corollary}

\begin{proof}
A finite group has trivial divisible part. Lemma~\ref{lem:CasselsTateNondegeneracy} therefore gives a non-degenerate alternating pairing on $\Sha(E/K)$. A finite abelian group carrying such a pairing has square order.
\end{proof}

\subsection{Finite exponent and the two directions of escape}\label{subsec:ShaTwoDirections}

\begin{proposition}[Finite exponent criterion]\label{prop:FiniteExponentCriterionSha}
The following statements are equivalent:
\begin{enumerate}
\item[\textup{(i)}] The group $\Sha(E/K)$ is finite.
\item[\textup{(ii)}] There exists an integer $N\ge 2$ such that
\[
  \Sha(E/K)=\Sha(E/K)[N].
\]
\item[\textup{(iii)}] The group $\Sha(E/K)$ has finite exponent.
\end{enumerate}
\end{proposition}

\begin{proof}
A finite group has finite exponent, so \textup{(i)} implies \textup{(iii)}. Statements \textup{(ii)} and \textup{(iii)} are equivalent by definition. Under \textup{(ii)}, Proposition~\ref{prop:FiniteLevelShaACA0} identifies $\Sha(E/K)$ with a quotient of the finite $N$-Selmer packet. Thus \textup{(ii)} implies \textup{(i)}.
\end{proof}

\begin{proposition}[Two directions of non-finiteness]\label{prop:TwoDirectionsOfShaEscape}
Assume that $\Sha(E/K)[n]$ is finite for every $n\ge 2$, as in Proposition~\ref{prop:FiniteLevelShaACA0}. If $\Sha(E/K)$ is infinite, then at least one of the following holds:
\begin{enumerate}
\item[\textup{(i)}] There are infinitely many primes $p$ for which
\[
  \Sha(E/K)[p]\ne 0.
\]
\item[\textup{(ii)}] There is a prime $p$ for which $\Sha(E/K)$ contains elements of order $p^e$ for arbitrarily large integers $e$.
\end{enumerate}
\end{proposition}

\begin{proof}
Suppose that neither condition holds. Then there is a finite set of primes $S$ containing every prime $p$ for which $\Sha(E/K)[p]\ne 0$. For each $p\in S$, there is an exponent $e_p$ such that no element has order greater than $p^{e_p}$ in the $p$-primary component. Put
\[
  N\coloneqq\prod_{p\in S}p^{e_p}.
\]
Every Shafarevich class is killed by $N$, so Proposition~\ref{prop:FiniteExponentCriterionSha} gives finiteness. This contradicts the hypothesis that $\Sha(E/K)$ is infinite.
\end{proof}

\begin{remark}[The role of divisibility]\label{rem:DivisibilityOnlyVertical}
Classically, finiteness of $\Sha(E/K)[p]$ gives a decomposition
\[
 \Sha(E/K)[p^\infty]\cong(\Q_p/\Z_p)^{d_p}\oplus F_p
\]
with $d_p$ finite and $F_p$ a finite $p$-group. Indeed, a torsion $p$-group with finite $p$-torsion has an injective envelope that is a finite sum of Pr\"ufer groups. Its Pontryagin dual is therefore a quotient of a finite free $\Z_p$-module, and the structure theorem gives the displayed decomposition.

Thus unbounded depth at $p$ is equivalent to $d_p>0$. The finite-shadow proofs express this alternative without constructing the divisible summand. Horizontal escape remains separate, as an infinite direct sum of finite groups at distinct primes has trivial divisible part.
\end{remark}

\begin{remark}[Passage to finite Selmer shadows]\label{rem:PassageToFiniteSelmerShadows}
The two alternatives of Proposition~\ref{prop:TwoDirectionsOfShaEscape} can be detected without forming the whole group. Vertical growth appears through the indices
\[
  \bigl[
    \Sha(E/K)[p^{e+1}]
    :
    \Sha(E/K)[p^e]
  \bigr].
\]
Horizontal growth appears through the non-vanishing of $\Sha(E/K)[p]$. Each term is a finite Selmer quotient by Proposition~\ref{prop:FiniteLevelShaACA0}. Section~\ref{sec:ShaSelmerShadows} organizes these finite quantities into a quantitative category and proves the exact two-axis criterion.
\end{remark}

%==================================================================
\section{Finite Selmer shadows and constructive closure}\label{sec:ShaSelmerShadows}
%==================================================================

The finite descent shadows permit a quantitative reading of the Tate--Shafarevich group in which every individual comparison belongs to finite arithmetic. A passage from one descent level to another has an index, and the logarithm of this index measures the growth of the shadow. The question of finiteness becomes a question about the termination of that growth, along a fixed prime and across the primes.

The distinction between these two directions deserves some care. An equality between two consecutive complete primary shadows has consequences at every greater depth. A finite list of vanishing prime shadows has a different scope, since further primes remain available. The first equality is a finite certificate, while the second passage calls for a uniform theorem.

We develop the finite-group argument over intuitionistic arithmetic. The arithmetic correctness of the descent procedures remains a declared input. The resulting statements give constructive reductions of finiteness, together with algorithms producing a finite presentation from the supplied closure data.

\subsection{The finite-descent interface}\label{subsec:ConstructiveDescentInterface}

\begin{definition}[Finite-descent base theory]\label{def:FiniteDescentBaseTheory}
Fix a number field $K$ and an elliptic curve $E/K$ with a certified Mordell--Weil presentation
\[
  E(K)\cong \Z^r\oplus T
\]
where $T$ is a finite coded abelian group.

A \emph{finite-descent interface} $\mathcal{D}_{\mathcal{E}}$ consists of specified procedures, with totality and correctness assumptions, supplying the following data:

\begin{enumerate}
\item For every integer $n\ge 1$, a finite coded Selmer group $S_n$, a finite coded group
\[
  M_n\cong E(K)/nE(K)
\]
and an injective Kummer map
\[
  \kappa_n\colon M_n\longrightarrow S_n.
\]
At $n=1$, both groups are trivial.

\item The finite quotient
\[
  X_n\coloneqq S_n/\kappa_n(M_n)
\]
with its addition table, its decidable equality, and the identity
\[
  [n]X_n=0.
\]

\item For every pair $n\mid m$, an injective homomorphism
\[
  j_{n,m}\colon X_n\longrightarrow X_m
\]
satisfying
\[
  j_{n,n}=\mathrm{id}_{X_n}
  \qquad\text{and}\qquad
  j_{m,\ell}j_{n,m}=j_{n,\ell}
\]
whenever $n\mid m\mid\ell$.

\item For every pair $n\mid m$, the torsion compatibility identity
\[
  j_{n,m}(X_n)
  =
  \ker\bigl([n]\colon X_m\longrightarrow X_m\bigr).
\]
The interface supplies the finite maps witnessing this identity.

\item A correctness interpretation identifying $X_n$ with $\Sha(E/K)[n]$ and $j_{n,m}$ with the inclusion of torsion subgroups. This interpretation includes the correspondence between finite descent classes and locally soluble torsor classes.
\end{enumerate}

All finite groups are presented on finite initial segments of the natural numbers. Quotients are formed by a fixed enumeration of cosets. Finite inverses of injective maps on their images are obtained by bounded search.

Write $\mathsf{T}^{\mathrm{int}}_{\mathcal{E}}$ for Heyting arithmetic, expanded by the specified procedures together with their totality and correctness assumptions. Induction is available in the expanded first-order language. Write $\mathsf{T}_{\mathcal{E}}$ for the corresponding classical presentation over $\mathsf{ACA}_0$, with graphs for the same procedures and the same arithmetic assumptions.
\end{definition}

\begin{remark}[The scope of the interface]\label{rem:FiniteDescentInterfaceScope}
The finite-group arguments below use the first four parts of Definition~\ref{def:FiniteDescentBaseTheory}. The final part identifies their arithmetic interpretation through the Kummer exact sequence. Completeness of the Selmer packets and of the Mordell--Weil presentation remains part of this input.

The rank certificate and a vertical closure certificate may be obtained in the same descent calculation. The interface separates their statements without requiring independent proofs. Higher descent or a Cassels--Tate calculation may provide the arithmetic certificate subsequently used by the finite-group argument. Thus the absence of the pairing from the closure deduction places no restriction on how its hypotheses are established.

No conservativity over bare Heyting arithmetic is asserted for the added totality and correctness assumptions. Their proof-theoretic strength remains part of the arithmetic input.
\end{remark}

\begin{definition}[The coded Shafarevich presentation]\label{def:CodedShaPresentation}
A representative of the \emph{coded Shafarevich presentation} is a pair $(n,x)$ with $n\ge 1$ and $x\in X_n$. Two representatives $(n,x)$ and $(m,y)$ are equivalent when
\[
  j_{n,\ell}(x)=j_{m,\ell}(y)
  \qquad
  \ell=\operatorname{lcm}(n,m).
\]
Addition is computed after transport to $X_\ell$.

This presentation will be denoted $\Sha^{\mathfrak P}(\mathcal{E})$. The notation refers to the coded representatives, their decidable equivalence relation, and their arithmetic operations. Its use requires no set of equivalence classes.
\end{definition}

\begin{lemma}[Decidability of comparison]\label{lem:ShaCodeComparison}
The equivalence relation of Definition~\ref{def:CodedShaPresentation} is decidable relative to $\mathcal{D}_{\mathcal{E}}$, and its addition is well defined.
\end{lemma}

\begin{proof}
The least common multiple is computed by integer arithmetic. The interface supplies the two transition maps and the decidable equality in the finite group $X_\ell$. Compatibility of the transition maps proves transitivity after transport to a common multiple of the three levels. Injectivity shows that an equality obtained at a greater common level holds at the least common multiple as well. The same compatibility proves that addition respects equivalence.
\end{proof}

\subsection{An index category}\label{subsec:ShaIndexCategory}

\begin{definition}[The category of finite shadows]\label{def:ShaIndexCategory}
Let $\mathcal{I}_{\mathcal{E}}$ have the positive integers as objects, with one arrow $n\to m$ when $n\mid m$. The arrow is interpreted by $j_{n,m}$. Put
\[
  A_n\coloneqq \#X_n
\]
and attach to $n\to m$ the positive integer index
\[
  I(n,m)\coloneqq \frac{A_m}{A_n}.
\]
Its logarithmic cost is
\[
  c(n,m)\coloneqq \log I(n,m).
\]
The height of the object $n$ is
\[
  H(n)\coloneqq \log A_n.
\]
\end{definition}

\begin{proposition}[Exact transport of the shadow height]\label{prop:ShaExactTransport}
For $n\mid m\mid\ell$, one has
\[
  I(n,\ell)=I(n,m)I(m,\ell)
\]
and
\[
  c(n,\ell)=c(n,m)+c(m,\ell).
\]
Moreover,
\[
  H(m)=H(n)+c(n,m).
\]
Thus $\mathcal{I}_{\mathcal{E}}$ carries a quantitative structure with an exact transport law.
\end{proposition}

\begin{proof}
The image of $X_n$ is a subgroup of $X_m$, so its order divides $A_m$. Counting cosets gives the index. Cancellation of the three positive integer orders proves multiplicativity. Taking logarithms gives the remaining identities.
\end{proof}

\begin{remark}[Integer data beneath the logarithms]\label{rem:ShaIntegerCosts}
Every assertion about a vanishing cost is the integer assertion $I(n,m)=1$. Every comparison of costs may be expressed by a comparison of products of positive integers. The constructive closure proofs therefore use the indices directly, while the logarithms provide the connection with the Lawvere reading of Section~\ref{sec:QuantitativeCategories}.

The category may be viewed as a locally discrete $(2,1)$-category. A richer presentation may include invertible $2$-morphisms between finite isomorphism certificates, provided that the induced indices remain unchanged. The arithmetic conclusions below depend only upon the finite groups and their transition maps.
\end{remark}

\subsection{Closure of a prime column}\label{subsec:PrimaryShadowClosure}

For a prime $p$ and an integer $e\ge 0$, write
\[
  \mathscr{X}_{p,e}(\mathcal{E})\coloneqq X_{p^e}
  \qquad\text{and}\qquad
  a_{p,e}\coloneqq v_p(\#X_{p^e}).
\]
In particular, $\mathscr{X}_{p,0}(\mathcal{E})=0$ and $a_{p,0}=0$.

\begin{definition}[Primary increment]\label{def:PrimaryShadowIncrement}
The \emph{primary increment} at depth $e$ is
\[
  b_{p,e}\coloneqq a_{p,e+1}-a_{p,e}.
\]
It is a non-negative integer, and
\[
  I(p^e,p^{e+1})=p^{b_{p,e}}.
\]
A \emph{vertical closure certificate at depth $e$} is the finite equality
\[
  b_{p,e}=0.
\]
\end{definition}

\begin{theorem}[Primary increments and one-step closure]\label{thm:OneStepPrimaryClosure}
For every prime $p$ and every integer $e\ge0$, one has
\[
  0\le b_{p,e+1}\le b_{p,e}.
\]
If $b_{p,e}=0$, then for every integer $f\ge e$ the transition map
\[
  j_{p^e,p^f}\colon X_{p^e}\longrightarrow X_{p^f}
\]
is an isomorphism. Its inverse is obtained by bounded search in $X_{p^e}$.
\end{theorem}

\begin{proof}
Use the transition maps to regard the three consecutive shadows as subgroups of $X_{p^{e+2}}$. Torsion compatibility gives a homomorphism
\[
  X_{p^{e+2}}/X_{p^{e+1}}
  \longrightarrow
  X_{p^{e+1}}/X_{p^e}
  \qquad
  [x]\longmapsto[px].
\]
Indeed, multiplication by $p$ carries each shadow into the preceding one. If $px\in X_{p^e}$, then $p^{e+1}x=0$, and torsion compatibility places $x$ in $X_{p^{e+1}}$. The displayed map is therefore injective. Its source and target have orders $p^{b_{p,e+1}}$ and $p^{b_{p,e}}$, which proves the inequalities.

If $b_{p,e}=0$, then induction gives $b_{p,k}=0$ for every $k\ge e$. Every subsequent transition joins finite groups of equal order and is therefore bijective. Their composites give the asserted isomorphisms, and bounded search in the finite source gives their inverses.
\end{proof}

\begin{corollary}[Closure from the first shadow]\label{cor:FirstShadowClosure}
If $X_p=0$, then $X_{p^f}=0$ for every integer $f\ge 0$.
\end{corollary}

\begin{proof}
Apply Theorem~\ref{thm:OneStepPrimaryClosure} with $e=0$.
\end{proof}

\begin{remark}[Why consecutive levels suffice]\label{rem:ConsecutiveLevelsSuffice}
The conclusion uses the torsion compatibility identity of Definition~\ref{def:FiniteDescentBaseTheory}. An arbitrary increasing sequence of finite groups may pause and later grow again. The shadows considered here are the complete subgroups killed by the successive powers of $p$, and their multiplication maps prevent such a pause.

The equality concerns the quotient by the complete global Kummer image. An equality between two incomplete search outputs has no corresponding consequence.
\end{remark}

\begin{proposition}[A Selmer-order test for vertical closure]\label{prop:SelmerOrderClosureTest}
Put
\[
  s_{p,e}\coloneqq v_p(\#S_{p^e})
  \qquad\text{and}\qquad
  t_{p,e}\coloneqq v_p(\#(T/p^eT)).
\]
Then
\[
  a_{p,e}=s_{p,e}-er-t_{p,e}.
\]
Consequently,
\[
  b_{p,e}
  =
  s_{p,e+1}-s_{p,e}
  -
  r
  -
  \bigl(t_{p,e+1}-t_{p,e}\bigr).
\]
A vertical closure certificate is therefore equivalent to the finite equality
\[
  s_{p,e+1}-s_{p,e}
  =
  r+t_{p,e+1}-t_{p,e}.
\]
\end{proposition}

\begin{proof}
The certified Mordell--Weil presentation gives
\[
  \#M_{p^e}
  =
  p^{er}\#(T/p^eT).
\]
The quotient defining $X_{p^e}$ gives
\[
  \#S_{p^e}=\#M_{p^e}\#X_{p^e}.
\]
Taking $p$-adic valuations proves the first identity. Subtracting consecutive identities proves the remaining assertions.
\end{proof}

\subsection{The horizontal datum}\label{subsec:HorizontalShadowClosure}

\begin{definition}[Two-axis closure datum]\label{def:TwoAxisClosureDatum}
A \emph{two-axis closure datum} for $\mathcal{E}$ consists of the following data:

\begin{enumerate}
\item An integer $B\ge 2$.

\item For every prime $p\le B$, an integer $e_p\ge 0$ and a vertical closure certificate
\[
  b_{p,e_p}=0.
\]

\item A proof of the uniform assertion
\[
  \forall p\,
  \bigl(
    p\text{ prime}\ \wedge\ p>B
    \ \Longrightarrow\
    X_p=0
  \bigr).
\]
\end{enumerate}

The last item is termed the \emph{horizontal closure assertion}. In a proof-carrying implementation, it is supplied by a derivation in the declared arithmetic theory. A finite list of calculations at primes below a search bound does not supply this derivation.
\end{definition}

\begin{remark}[Two senses of finite certification]\label{rem:TwoSensesOfCertification}
Each vertical equality is checked by finite group arithmetic. The horizontal assertion quantifies over arbitrarily large primes. A formal derivation of that assertion is a finite string whose proof steps can be checked, while the assertion itself has an unbounded arithmetic scope.

The verification of a derivation and the verification of finitely many numerical instances are distinct procedures, and the two-axis datum retains this distinction.
\end{remark}

\begin{definition}[A finite exponent witness]\label{def:ConstructiveShaExponent}
A \emph{finite exponent witness} is an integer $N\ge 1$, together with a proof that
\[
  \forall n\ge 1\ \forall x\in X_n\quad Nx=0.
\]
\end{definition}

\begin{theorem}[Constructive two-axis finiteness]\label{thm:TwoAxisShaFiniteness}
Over $\mathsf{T}^{\mathrm{int}}_{\mathcal{E}}$, a two-axis closure datum yields a finite exponent witness
\[
  N=\prod_{\substack{p\le B\\p\ {\rm prime}}}p^{e_p}
\]
and a finite presentation
\[
  \Sha^{\mathfrak P}(\mathcal{E})
  \cong
  X_N
  \cong
  \prod_{\substack{p\le B\\p\ {\rm prime}}}
  \mathscr{X}_{p,e_p}(\mathcal{E}).
\]
Every representative $(n,x)$ can be reduced to its unique representative in $X_N$ by a terminating procedure.

Conversely, a finite exponent witness yields a two-axis closure datum. Under the correctness interpretation of Definition~\ref{def:FiniteDescentBaseTheory}, these conclusions give a finite presentation of $\Sha(E/K)$ and the order formula
\[
  \#\Sha(E/K)
  =
  \prod_{\substack{p\le B\\p\ {\rm prime}}}
  \#\mathscr{X}_{p,e_p}(\mathcal{E}).
\]
\end{theorem}

\begin{proof}
Take a two-axis closure datum. For every prime $p\le B$, Theorem~\ref{thm:OneStepPrimaryClosure} identifies all deeper $p$-primary shadows with $X_{p^{e_p}}$. For every prime $p>B$, the horizontal assertion and Corollary~\ref{cor:FirstShadowClosure} make every $p$-primary shadow trivial.

Now take $x\in X_n$. Factor $n$ into prime powers. The Chinese remainder identity decomposes $x$ into its primary components by integer multiples of $x$. Torsion compatibility places each component in the image of the corresponding prime-power shadow. The preceding paragraph shows that every component is killed by $N$, whence $Nx=0$.

Put $\ell=\operatorname{lcm}(n,N)$. The image $j_{n,\ell}(x)$ is killed by $N$. Torsion compatibility therefore places it in the image of
\[
  j_{N,\ell}\colon X_N\longrightarrow X_\ell.
\]
Search the finite group $X_N$ for its preimage. Existence follows from torsion compatibility, and uniqueness follows from injectivity. This gives the reduction procedure.

The primary decomposition of the finite group $X_N$, computed by the Chinese remainder identity, gives the displayed product. Taking orders gives the order formula.

Conversely, suppose that $N$ is a finite exponent witness. Put $B=\max\{2,N\}$ and $e_p=v_p(N)$ for the primes $p\le B$. Every element of $X_{p^{e_p+1}}$ is killed by both $N$ and $p^{e_p+1}$. B\'ezout's identity shows that it is killed by
\[
  \gcd(N,p^{e_p+1})=p^{e_p}.
\]
Torsion compatibility renders the transition from $X_{p^{e_p}}$ surjective, so $b_{p,e_p}=0$.

For a prime $p>B$, one has $\gcd(p,N)=1$. Every element of $X_p$ is killed by both $p$ and $N$, hence is zero by B\'ezout's identity. This proves the horizontal assertion.
\end{proof}

\begin{remark}[The meaning of finiteness here]\label{rem:ConstructiveShaFinitenessMeaning}
The conclusion supplies a finite group together with a procedure reducing every coded class to that group. Equality of classes is decidable before this conclusion, by Lemma~\ref{lem:ShaCodeComparison}. Closure adds a common finite domain of representatives.

A finite exponent witness and a two-axis closure datum therefore carry the same constructive information, with explicit translations in both directions. Their existence remains an arithmetic question about the complete descent shadows.
\end{remark}

\subsection{A quantitative route to closure}\label{subsec:ShaBudgetClosure}

The index category permits a second formulation of the missing input. A uniform bound on its object heights controls both the primary depth and the prime support, and the integer form of that bound is particularly convenient.

\begin{definition}[Uniform shadow budget]\label{def:UniformShaBudget}
A \emph{uniform shadow budget} is an integer $M\ge 1$, together with a proof that
\[
  \forall n\ge 1\quad \#X_n\le M.
\]
Its logarithmic form is
\[
  H(n)\le \log M
\]
for every object of $\mathcal{I}_{\mathcal{E}}$.
\end{definition}

\begin{theorem}[A budget gives constructive finiteness]\label{thm:ShaBudgetFiniteness}
A uniform shadow budget $M$ yields a finite exponent witness
\[
  N=\operatorname{lcm}(1,2,\ldots,M).
\]
Consequently, it yields a finite presentation of $\Sha^{\mathfrak P}(\mathcal{E})$, a two-axis closure datum, and the bound
\[
  \#\Sha^{\mathfrak P}(\mathcal{E})\le M.
\]
All the constructions terminate relative to the finite-descent interface.
\end{theorem}

\begin{proof}
Take $x\in X_n$. Its order is computed by testing the successive multiples of $x$ in the finite group. The cyclic subgroup generated by $x$ has at most $\#X_n$ elements, so the order $d$ of $x$ satisfies
\[
  1\le d\le \#X_n\le M.
\]
By construction, $d$ divides $N$, whence $Nx=0$. This proves that $N$ is a finite exponent witness.

Theorem~\ref{thm:TwoAxisShaFiniteness} produces the finite presentation and the two-axis datum. The budget applied at level $N$ gives $\#X_N\le M$.
\end{proof}

\begin{proposition}[A finite presentation gives a budget]\label{prop:ShaFinitenessGivesBudget}
Suppose that a finite exponent witness $N$ has been supplied. Then
\[
  M=\#X_N
\]
is a uniform shadow budget.
\end{proposition}

\begin{proof}
For every $n$, the reduction constructed in Theorem~\ref{thm:TwoAxisShaFiniteness} gives a homomorphism from $X_n$ to $X_N$. It is injective, since an equality of the reduced representatives implies an equality after transport to $X_{\operatorname{lcm}(n,N)}$, and the transition from $X_n$ is injective. Hence $\#X_n\le\#X_N$.
\end{proof}

\begin{remark}[The arithmetic strength of the budget]\label{rem:ShaBudgetStrength}
Theorems~\ref{thm:TwoAxisShaFiniteness} and~\ref{thm:ShaBudgetFiniteness}, together with Proposition~\ref{prop:ShaFinitenessGivesBudget}, give witness-preserving passages among a finite exponent, a two-axis closure, and a uniform bound on all shadow orders. The budget therefore contains the whole finiteness problem in a quantitative form.

Subadditivity of the transport costs provides no such budget. The identity
\[
  H(m)=H(n)+c(n,m)
\]
permits an unbounded growth as the descent level varies.
\end{remark}

\begin{proposition}[A bounded realization criterion]\label{prop:BoundedShaRealization}
Suppose that the following data are supplied:

\begin{enumerate}
\item A finite coded set $P_{\mathcal{E}}$.

\item For every $n\ge 1$, a specified injection
\[
  \rho_n\colon X_n\longrightarrow P_{\mathcal{E}}.
\]

\item Proofs of the totality and the injectivity of these maps, uniformly in $n$.
\end{enumerate}

Then $\Sha^{\mathfrak P}(\mathcal{E})$ has a finite presentation.
\end{proposition}

\begin{proof}
The integer $\#P_{\mathcal{E}}$ bounds $\#X_n$ for every $n$. Apply Theorem~\ref{thm:ShaBudgetFiniteness}.
\end{proof}

\begin{remark}[An arithmetic task for the intensional framework]\label{rem:ShaBoundedRealizationTask}
Proposition~\ref{prop:BoundedShaRealization} gives one possible form of a future arithmetic comparison theorem. A normalization of torsor data would have to place every finite shadow inside one bounded finite collection while preserving distinct classes. A bound for the Weierstrass dilations of the Jacobian does not supply this conclusion, since all torsors under $E$ share the same Jacobian.

A useful comparison must therefore control the descent classes themselves. A finite bound on the encoding length, together with an injective encoding, would suffice. Its proof would have to remain uniform in the covering degree and in the descent level. Section~\ref{sec:BoundedShaPrimeTwo} supplies such a bound along one primary column, uniformly in a family parameter.
\end{remark}

\subsection{What the transport axioms permit}\label{subsec:ShaFormalGrowthExamples}

\begin{example}[Unbounded primary depth]\label{ex:UnboundedPrimaryShadow}
Fix a prime $q$. At level $n$, put
\[
  Y_n=
  \bigl(\Z/q^{v_q(n)}\Z\bigr)^2.
\]
For $n\mid m$, use multiplication by $q^{v_q(m)-v_q(n)}$ in each coordinate. These maps are injective and satisfy the torsion compatibility identity.

Every $Y_n$ is finite and has square order. Along the $q$-column one has
\[
  \#Y_{q^e}=q^{2e}.
\]
Thus the primary increment is $2$ at every depth. The index category has non-negative additive costs and no finite exponent witness.
\end{example}

\begin{example}[Unbounded prime support]\label{ex:UnboundedHorizontalShadow}
At level $n$, put
\[
  Z_n=
  \prod_{p\mid n}
  (\Z/p\Z)^2.
\]
For $n\mid m$, use the inclusion into the coordinates indexed by the primes dividing $n$. These maps satisfy torsion compatibility.

Every prime column closes at depth $1$, and every finite shadow has square order. Nevertheless, $Z_p$ is non-trivial for every prime $p$. A finite exponent witness would have to be divisible by every prime, which is impossible.
\end{example}

\begin{remark}[The scope of the examples]\label{rem:ShaExamplesScope}
The examples are systems of finite groups, and they make no assertion about the Tate--Shafarevich group of an elliptic curve. They show that finite-level finiteness, torsion compatibility, square orders, and additive transport costs do not force a global closure.

A finite examination of levels also leaves both possibilities open. Given a bound $R$, choose a prime $q>R$. The system of Example~\ref{ex:UnboundedPrimaryShadow} is trivial at every level $n\le R$, just as the everywhere trivial system is. Arithmetic information about the particular curve must provide the further restriction.
\end{remark}

\subsection{The BSD factor and the constructive boundary}\label{subsec:ShaClosureBSD}

\begin{corollary}[A finite Shafarevich factor]\label{cor:ConstructiveShaBSDFactor}
A two-axis closure datum supplies the integer
\[
  S_{\mathcal{E}}
  =
  \prod_{\substack{p\le B\\p\ {\rm prime}}}
  \#\mathscr{X}_{p,e_p}(\mathcal{E}).
\]
Under the correctness interpretation of the finite-descent interface, one has
\[
  S_{\mathcal{E}}=\#\Sha(E/K).
\]
Thus the Shafarevich term of Definition~\ref{def:BSDTransportBalance} is represented by the logarithmic integer datum $\log S_{\mathcal{E}}$.
\end{corollary}

\begin{proof}
Apply Theorem~\ref{thm:TwoAxisShaFiniteness} and its order formula.
\end{proof}

\begin{remark}[Independence from the asymptotic conjecture]\label{rem:ShaClosureIndependentOfEuler}
The closure argument uses no partial Euler product, no period integral, no regulator estimate, and no Cassels--Tate pairing. A proof of a numerical BSD identity would call for an additional argument before its predicted Shafarevich factor could serve as a uniform bound on all the finite shadows.

The Cassels--Tate pairing contributes structural information through its alternating property and through its radical, as discussed in Section~\ref{subsec:CasselsTatePairing}. The closure proofs above require neither a maximal divisible subgroup nor a construction of that pairing.
\end{remark}

\begin{proposition}[Constructive formal scope]\label{prop:ConstructiveShaFormalScope}
Relative to the procedures and the assumptions of Definition~\ref{def:FiniteDescentBaseTheory}, the proofs from Lemma~\ref{lem:ShaCodeComparison} through Corollary~\ref{cor:ConstructiveShaBSDFactor} are derivable in $\mathsf{T}^{\mathrm{int}}_{\mathcal{E}}$.

The constructions use finite group tables, integer arithmetic, bounded search, and induction. They use no excluded middle for arbitrary propositions, no choice principle, no compactness principle, and no comprehension forming an infinite group.

The same arithmetic arguments are available in $\mathsf{T}_{\mathcal{E}}$.
\end{proposition}

\begin{proof}
Every case distinction in the constructions concerns a decidable equality in a finite group, a divisibility of integers, or a membership in a finite list. The preimage searches have finite domains and existence proofs. The primary decomposition uses the Euclidean algorithm and the Chinese remainder identity. The propagation of a vertical closure certificate uses induction on the depth. The budget argument uses finite cyclic subgroups and a finite least common multiple.

These operations and deductions are available in Heyting arithmetic with the specified procedures. Their classical presentation is available over $\mathsf{ACA}_0$ with the corresponding graphs and assumptions.
\end{proof}

\begin{conjecture}[Constructive Shafarevich closure]\label{conj:ConstructiveShaClosure}
For every elliptic curve over a certified number field, supplied with a constructive finite-descent interface, there exists a finite exponent witness in the sense of Definition~\ref{def:ConstructiveShaExponent}.
\end{conjecture}

The results of this section translate the conjectural witness into a finite presentation and an explicit Shafarevich factor. They translate a uniform shadow budget into such a witness as well. The remaining arithmetic task is the construction of one of these uniform inputs for the curve under consideration. Section~\ref{sec:BoundedShaPrimeTwo} carries out that construction for one prime and one parameterized family, where the two-primary column closes uniformly.

%==================================================================
\section{A bounded realization at the prime two}\label{sec:BoundedShaPrimeTwo}
%==================================================================

The bounded realization criterion of Proposition~\ref{prop:BoundedShaRealization} becomes an arithmetic theorem once a family of descent classes admits a uniform finite reduction. We examine such a reduction for the prime $2$ and for the curves
\[
  E_q\colon y^2=x^3-q^2x
\]
where $q$ is a prime satisfying $q\equiv3\pmod8$. The parameter $q$ varies, while the primary column under examination remains the column at $2$.

The calculation proceeds through a two-isogeny and its dual. Each isogeny admits eight candidate square classes. Local obstructions remove the candidates that cannot survive, and explicit rational points account for every remaining class. The resulting isogeny quotients vanish. Their composition closes the first $2$-primary Shafarevich shadow, after which torsion compatibility closes every greater depth.

The arithmetic of these descents belongs to the established study of the congruent-number curves. The two-isogeny descent is set out by Silverman~\cite[Chapter~X]{SilvermanAEC}. Feng and Xiong determine by this descent the curves \(y^2=x^3-n^2x\) of rank zero~\cite{FengXiong2004}, and Kane and Thorne describe the distribution of the \(\phi\)-Selmer groups across the related family \(y^2=x^3-Dx\)~\cite{KaneThorne2017}. We retain the finite witnesses throughout, so that the conclusion supplies the bounded target the quantitative framework requires.

\subsection{The family and its isogenies}\label{subsec:PrimeTwoFamily}

\begin{definition}[The prime-twist family]\label{def:PrimeThreeModEightFamily}
Let $q$ be a positive prime with $q\equiv3\pmod8$. Define
\[
  E_q\colon y^2=x^3-q^2x
  \qquad\text{and}\qquad
  F_q\colon Y^2=X^3+4q^2X.
\]
The two-isogeny
\[
  \phi_q\colon E_q\longrightarrow F_q
\]
is given away from $x=0$ by
\[
  \phi_q(x,y)
  =
  \left(
    x-\frac{q^2}{x}
    \ ,
    y\left(1+\frac{q^2}{x^2}\right)
  \right).
\]
Its dual
\[
  \widehat\phi_q\colon F_q\longrightarrow E_q
\]
is given away from $X=0$ by
\[
  \widehat\phi_q(X,Y)
  =
  \left(
    \frac{X+4q^2/X}{4}
    \ ,
    \frac{Y(1-4q^2/X^2)}{8}
  \right).
\]
Each map sends the point at infinity to the point at infinity and the point $(0,0)$ to the point at infinity.
\end{definition}

\begin{lemma}[The two compositions]\label{lem:PrimeTwistDualCompositions}
The maps of Definition~\ref{def:PrimeThreeModEightFamily} are isogenies of degree $2$ and satisfy
\[
  \widehat\phi_q\circ\phi_q=[2]_{E_q}
  \qquad\text{and}\qquad
  \phi_q\circ\widehat\phi_q=[2]_{F_q}.
\]
Their kernels are generated by the respective points $(0,0)$.
\end{lemma}

\begin{proof}
Substitution verifies the equations of the target curves. The formulas extend across their exceptional points to the degree-$2$ isogenies associated with the rational points of order $2$. Substitution into the duplication formula gives the two composition identities. These are the explicit two-isogeny formulas of descent~\cite[Chapter~X]{SilvermanAEC}.
\end{proof}

\begin{remark}[The parameter and the primary prime]\label{rem:ParameterVersusPrimaryPrime}
The prime $q$ determines the curve together with one of its local descent conditions. The prime $2$ determines the torsion shadows. The assertion to be proved is uniform in $q$ and in the exponent $e$ of the shadow $\Sha(E_q/\Q)[2^e]$.
\end{remark}

\subsection{The arithmetic interface}\label{subsec:PrimeTwoDescentInterface}

\begin{definition}[Two-isogeny descent interface]\label{def:PrimeTwoIsogenyInterface}
For each curve in the pair $(E_q,F_q)$, assume finite coded abelian groups $X_n$ for $n\ge1$, with $X_1=0$, decidable equality, and $[n]X_n=0$. Assume the injective transition maps, their compatibility, and the torsion compatibility identity of Definition~\ref{def:FiniteDescentBaseTheory}, together with the correctness interpretation $X_n\cong\Sha(E/\Q)[n]$. These shadow data are supplied directly, independently of a Mordell--Weil presentation.

The \emph{two-isogeny descent interface} also supplies the following data:
\begin{enumerate}
\item Finite coded isogeny Selmer groups
\[
  \mathcal S_q^+\coloneqq\Sel^{(\phi_q)}(E_q/\Q)
  \qquad\text{and}\qquad
  \mathcal S_q^-\coloneqq\Sel^{(\widehat\phi_q)}(F_q/\Q).
\]

\item Their complete descriptions by the square classes of the quartics in Definition~\ref{def:PrimeTwistQuartics}. Every locally soluble class has a representative in that description, and a finite local obstruction excludes the class of that representative.

\item The finite Kummer images and the exact sequences
\[
  0\longrightarrow F_q(\Q)/\phi_q E_q(\Q)
  \longrightarrow\mathcal S_q^+
  \longrightarrow\Sha(E_q/\Q)[\phi_q]
  \longrightarrow0
\]
and
\[
  0\longrightarrow E_q(\Q)/\widehat\phi_q F_q(\Q)
  \longrightarrow\mathcal S_q^-
  \longrightarrow\Sha(F_q/\Q)[\widehat\phi_q]
  \longrightarrow0.
\]
The rational-point quotients are represented by their finite Kummer images, with evaluation of the class of a supplied rational point.

\item Maps on the finite shadows induced by $\phi_q$ and $\widehat\phi_q$, compatible with transition maps. At every even level their kernels are the respective final quotients in these exact sequences. At every level their compositions are multiplication by $2$.
\end{enumerate}

All procedures carry uniform totality and correctness assumptions. The notation involving $\Sha$ is interpreted through the specified finite groups, kernels, and quotients. Write $\mathsf T_q^{\mathrm{int}}$ for Heyting arithmetic expanded by these procedures and assumptions, with induction in the expanded language.
\end{definition}

\begin{remark}[Scope of the correctness assumptions]\label{rem:PrimeTwoCorrectnessScope}
The description by square classes and the Kummer exact sequences are the arithmetic correctness input of two-isogeny descent~\cite[Chapter~X]{SilvermanAEC}. The calculation below proves the required local exclusions and supplies a rational witness for every surviving square class.

The argument requires no complete list of rational points on either curve. The displayed points supply enough Kummer classes to exhaust the two finite Selmer groups. A full formal derivation over bare Heyting arithmetic would also call for a constructive formalization of the descent correctness input, and we keep that input explicit.
\end{remark}

\subsection{The quartic candidates}\label{subsec:PrimeTwistQuarticCandidates}

\begin{definition}[The two quartic families]\label{def:PrimeTwistQuartics}
For a non-zero squarefree integer $d$, let $C^-_{q,d}$ and $C^+_{q,d}$ denote the smooth projective curves represented by
\[
  C^-_{q,d}\colon
  dW^2=d^2U^4-q^2V^4
\]
and
\[
  C^+_{q,d}\colon
  dW^2=d^2U^4+4q^2V^4.
\]
The coordinates have weights $1,1,2$, so that
\[
  (U,V,W)
  \longmapsto
  (tU,tV,t^2W)
\]
preserves the represented point for every non-zero scalar $t$. A solution must satisfy $(U,V)\ne(0,0)$.

The locally soluble classes of $C^-_{q,d}$ represent $\mathcal{S}_q^-$, while those of $C^+_{q,d}$ represent $\mathcal{S}_q^+$.
\end{definition}

The formulas can be read directly from the square class of the first coordinate. On a curve
\[
  y^2=x^3+bx
\]
write $x=d(U/V)^2$ away from the exceptional points. Substitution gives
\[
  dW^2=d^2U^4+bV^4.
\]
The exceptional points are retained in the projective model.

\begin{lemma}[Support of the candidate square classes]\label{lem:PrimeTwistSquareclassSupport}
Every class in either $\mathcal{S}_q^-$ or $\mathcal{S}_q^+$ has a representative in
\[
  \mathcal{D}_q
  =
  \{1,-1,2,-2,q,-q,2q,-2q\}.
\]
\end{lemma}

\begin{proof}
Take a squarefree representative $d$ and suppose that an odd prime $\ell\nmid2q$ divides $d$. Write
\[
  dW^2=d^2U^4+bV^4
\]
with $b=-q^2$ or $b=4q^2$. At $\ell$, the coefficient $b$ is a unit and $v_\ell(d)=1$.

A local point can be scaled so that $U,V\in\Z_\ell$ and at least one of them is a unit. The right member then has non-negative valuation, while the left member has valuation $1+2v_\ell(W)$, so that $W\in\Z_\ell$.

If $V$ is a unit, then the right member is a unit, while the left member is divisible by $\ell$, which is impossible. If $V$ is divisible by $\ell$, then $U$ is a unit. The two terms on the right have valuations $2$ and at least $4$, so the right member has valuation $2$, while the left member has odd valuation. This is again impossible.

Thus every prime divisor of $d$ belongs to $\{2,q\}$. The squarefree representatives with this support are precisely the eight members of $\mathcal{D}_q$.
\end{proof}

\begin{remark}[Primitive local coordinates]\label{rem:PrimeTwistPrimitiveCoordinates}
All the local exclusions below use one normalization. At the prime under consideration, the coordinates $U,V$ are integral and at least one of them is a unit. The equation then determines the integrality of $W$.

The local notation is convenient for the mathematical interpretation. Each obstruction will also be given by a finite congruence condition on primitive residue coordinates. The proof uses no compactness argument, and it uses no passage from solutions at every precision to a local point.
\end{remark}

\subsection{The minus quartics}\label{subsec:MinusQuarticCalculation}

\begin{lemma}[Exclusion of the even minus classes]\label{lem:MinusEvenClassObstruction}
For every
\[
  d\in\{2,-2,2q,-2q\}
\]
the curve $C^-_{q,d}$ has no point over $\Q_2$.
\end{lemma}

\begin{proof}
Write $d=2a$ with $a$ odd. A normalized local point would satisfy
\[
  2aW^2=4a^2U^4-q^2V^4
\]
with $U,V,W\in\Z_2$ and at least one of $U,V$ odd.

If $V$ is odd, then the right member is odd while the left member is even. If $V$ is even, then $U$ is odd, and in this case the right member is congruent to $4$ modulo $16$.

The residue of $2aW^2$ modulo $16$ cannot be $4$. For odd $W$, it is congruent to $2a$ modulo $16$, which is twice an odd number. For $v_2(W)=1$, it is $8$ modulo $16$. For $v_2(W)\ge2$, it is $0$ modulo $16$. These cases exhaust the possible residues of $W$.

Thus the normalized equation has no solution modulo $16$, and no point over $\Q_2$ exists.
\end{proof}

\begin{lemma}[Rational witnesses for the odd minus classes]\label{lem:MinusRationalWitnesses}
Every class in
\[
  \{1,-1,q,-q\}
\]
has a rational point on its minus quartic. Explicit witnesses are
\[
\begin{array}{c|c}
d & (U,V,W)\\
\hline
1 & (1,0,1)\\
-1 & (0,1,q)\\
q & (1,1,0)\\
-q & (1,1,0)
\end{array}
\]
and their Kummer classes are represented respectively by
\[
  O
  \qquad
  (0,0)
  \qquad
  (q,0)
  \qquad
  (-q,0)
\]
on $E_q$.
\end{lemma}

\begin{proof}
Substitution verifies every quartic identity in the table. For the Kummer map on $E_q$, the values at the four displayed points are
\[
  1
  \qquad
  -q^2
  \qquad
  q
  \qquad
  -q
\]
in $\Q^\times/(\Q^\times)^2$. The class of $-q^2$ is the class of $-1$. The four witnesses therefore give precisely the stated classes.
\end{proof}

\begin{proposition}[The minus Selmer group]\label{prop:MinusSelmerCalculation}
One has
\[
  \mathcal{S}_q^-=\{1,-1,q,-q\}.
\]
Every element lies in the global Kummer image. Consequently,
\[
  \Sha(F_q/\Q)[\widehat\phi_q]=0.
\]
\end{proposition}

\begin{proof}
Lemma~\ref{lem:PrimeTwistSquareclassSupport} gives the eight candidates. Lemma~\ref{lem:MinusEvenClassObstruction} excludes the four even classes. Lemma~\ref{lem:MinusRationalWitnesses} supplies a global point for each remaining class, hence a point over every completion.

The Kummer image therefore exhausts $\mathcal{S}_q^-$, and the second exact sequence of Definition~\ref{def:PrimeTwoIsogenyInterface} gives the vanishing quotient.
\end{proof}

\subsection{The plus quartics}\label{subsec:PlusQuarticCalculation}

The plus quartics call for three short exclusions. The real place removes the negative classes, the prime $q$ removes the classes divisible by $q$, and the prime $2$ removes the remaining class $2$.

\begin{lemma}[The real obstruction]\label{lem:PlusRealObstruction}
If $d<0$, then $C^+_{q,d}(\R)$ is empty.
\end{lemma}

\begin{proof}
For $(U,V)\ne(0,0)$, the expression
\[
  d^2U^4+4q^2V^4
\]
is strictly positive, while the expression $dW^2$ is non-positive. Their equality is impossible.
\end{proof}

\begin{lemma}[A finite residue fact]\label{lem:PrimeThreeModFourResidues}
For a prime $q\equiv3\pmod4$, the element $-1$ is not a square in $\mathbb{F}_q$. Consequently, neither
\[
  U^4+V^4
  \qquad\text{nor}\qquad
  U^4+4V^4
\]
vanishes modulo $q$ when $(U,V)$ is a non-zero pair modulo $q$.
\end{lemma}

\begin{proof}
Suppose that $t^2=-1$ in $\mathbb{F}_q$. Then $t\ne0$, and Fermat's identity gives
\[
  1=t^{q-1}=(-1)^{(q-1)/2}=-1
\]
since $(q-1)/2$ is odd. This is impossible because $q$ is odd.

If $V=0$ and $U\ne0$, then both quartic expressions are non-zero. If $V\ne0$ and $U^4+V^4=0$, then $(U^2/V^2)^2=-1$. If $V\ne0$ and $U^4+4V^4=0$, then $(U^2/(2V^2))^2=-1$. Both conclusions contradict the first assertion.
\end{proof}

\begin{lemma}[Exclusion at the parameter prime]\label{lem:PlusParameterPrimeObstruction}
The curves $C^+_{q,q}$ and $C^+_{q,2q}$ have no points over $\Q_q$.
\end{lemma}

\begin{proof}
For $d=q$, the quartic equation becomes
\[
  W^2=q(U^4+4V^4).
\]
Normalize $U,V$ at $q$. Lemma~\ref{lem:PrimeThreeModFourResidues} shows that the expression in parentheses is a unit. The right member therefore has valuation $1$, while a non-zero square has even valuation.

For $d=2q$, the equation becomes
\[
  W^2=2q(U^4+V^4).
\]
The same lemma shows that the expression in parentheses is a unit, and $2$ is a unit at $q$ as well. The right member again has valuation $1$.

Both contradictions can be checked modulo $q^2$. A square divisible by $q$ is zero modulo $q^2$, while the respective right members are non-zero multiples of $q$ modulo $q^2$.
\end{proof}

\begin{lemma}[Exclusion of the remaining even class]\label{lem:PlusTwoAdicObstruction}
For a prime $q\equiv3\pmod8$, the curve $C^+_{q,2}$ has no point over $\Q_2$.
\end{lemma}

\begin{proof}
The equation for $d=2$ is
\[
  W^2=2(U^4+q^2V^4).
\]
Normalize $U,V$ at $2$. If precisely one of them is odd, then the right member has valuation $1$, which is impossible for a square.

There remains the case of odd $U$ and odd $V$. One has
\[
  U^4\equiv V^4\equiv1\pmod{16}
  \qquad\text{and}\qquad
  q^2\equiv9\pmod{16}.
\]
Hence
\[
  U^4+q^2V^4\equiv10\pmod{16}.
\]
The equation renders $W$ even. Write $W=2Z$. Dividing by $2$ gives
\[
  2Z^2=U^4+q^2V^4
\]
and therefore
\[
  Z^2\equiv5\pmod8.
\]
The square residues modulo $8$ are $0,1,4$, which is a contradiction.

Equivalently, the primitive equation for $W$ has no solution modulo $32$.
\end{proof}

\begin{proposition}[The plus Selmer group]\label{prop:PlusSelmerCalculation}
One has
\[
  \mathcal{S}_q^+=\{1\}.
\]
Consequently,
\[
  \Sha(E_q/\Q)[\phi_q]=0.
\]
\end{proposition}

\begin{proof}
The candidate list is $\mathcal{D}_q$. Lemma~\ref{lem:PlusRealObstruction} excludes its four negative members. Lemma~\ref{lem:PlusParameterPrimeObstruction} excludes $q$ and $2q$. Lemma~\ref{lem:PlusTwoAdicObstruction} excludes $2$.

The class $1$ has the rational witness
\[
  (U,V,W)=(1,0,1)
\]
and is the identity Kummer class. The first exact sequence of Definition~\ref{def:PrimeTwoIsogenyInterface} therefore has trivial final quotient.
\end{proof}

\subsection{The first Shafarevich shadow}\label{subsec:FirstPrimeTwoShadow}

\begin{theorem}[Vanishing of the first two-primary shadows]\label{thm:PrimeTwistShaTwoVanishing}
For every prime $q\equiv3\pmod8$, one has
\[
  \Sha(E_q/\Q)[2]=0
  \qquad\text{and}\qquad
  \Sha(F_q/\Q)[2]=0.
\]
The conclusions follow from the two-isogeny descent interface together with the finite calculations above.
\end{theorem}

\begin{proof}
Take a class $\xi\in\Sha(E_q/\Q)[2]$. Its image under $\phi_q$ satisfies
\[
  \widehat\phi_q(\phi_q(\xi))=2\xi=0.
\]
It therefore belongs to $\Sha(F_q/\Q)[\widehat\phi_q]$, which vanishes by Proposition~\ref{prop:MinusSelmerCalculation}. Thus $\phi_q(\xi)=0$, and Proposition~\ref{prop:PlusSelmerCalculation} now gives $\xi=0$.

For a class $\eta\in\Sha(F_q/\Q)[2]$, the identity
\[
  \phi_q(\widehat\phi_q(\eta))=2\eta=0
\]
places $\widehat\phi_q(\eta)$ in $\Sha(E_q/\Q)[\phi_q]$. The plus calculation renders this class zero, and the minus calculation then gives $\eta=0$.

Both arguments take place in finite coded groups. They require only the kernel descriptions and the composition identities of Definition~\ref{def:PrimeTwoIsogenyInterface}.
\end{proof}

\begin{proposition}[The Mordell--Weil package of the prime family]\label{prop:PrimeTwistMWPackage}\label{rem:PrimeTwistNoMWSearch}
In classical arithmetic, every prime $q\equiv3\pmod8$ satisfies
\[
 E_q(\Q)=\{O,(0,0),(q,0),(-q,0)\}\cong(\Z/2\Z)^2.
\]
These points supply the Mordell--Weil data of Definition~\ref{def:MordellWeilPackage}, with $r=0$ and $\Reg=1$.
\end{proposition}

\begin{proof}
Proposition~\ref{prop:PlusSelmerCalculation} gives $F_q(\Q)=\phi_q E_q(\Q)$, so $\widehat\phi_q F_q(\Q)=2E_q(\Q)$. Proposition~\ref{prop:MinusSelmerCalculation} therefore gives $\#(E_q(\Q)/2E_q(\Q))=4$. By finite generation and the four rational points of order dividing $2$, this order is $2^{r+2}$, whence $r=0$~\cite[Chapter~VIII]{SilvermanAEC}.

For every good prime $\ell\equiv3\pmod4$, pairing $x$ with $-x$ in the residue-field count gives $\#E_q(\mathbb F_\ell)=\ell+1$. If $q\ne3$, then use the good primes $3$ and $7$. If $q=3$, then use $7$ and $11$. Injectivity of reduction on torsion of order prime to the residue characteristic shows in either case that the rational torsion order divides $4$~\cite[Chapter~VII]{SilvermanAEC}. The four displayed points exhaust it. The free part is empty, so its height pairing and regulator have the rank-zero conventions.
\end{proof}

\subsection{Closure at every depth}\label{subsec:AllPrimeTwoDepths}

\begin{theorem}[Uniform closure of the two-primary column]\label{thm:PrimeTwistAllTwoPowerShadows}
For every prime $q\equiv3\pmod8$ and every integer $e\ge0$, one has
\[
  \mathscr{X}_{2,e}(E_q)=0
  \qquad\text{and}\qquad
  \mathscr{X}_{2,e}(F_q)=0.
\]
Under the arithmetic interpretation, this says
\[
  \Sha(E_q/\Q)[2^e]=0
  \qquad\text{and}\qquad
  \Sha(F_q/\Q)[2^e]=0.
\]
\end{theorem}

\begin{proof}
The shadow at $e=0$ is trivial by definition. Theorem~\ref{thm:PrimeTwistShaTwoVanishing} gives the assertion at $e=1$. Corollary~\ref{cor:FirstShadowClosure} then proves the claim at every depth.

The induction may also be written directly. Suppose that the shadow at depth $e$ is trivial, and take $x$ in the shadow at depth $e+1$. The element $2x$ is killed by $2^e$, so torsion compatibility places it in the image of the depth-$e$ shadow, whence $2x=0$. Torsion compatibility then places $x$ in the image of the depth-$1$ shadow, which is trivial.
\end{proof}

\begin{corollary}[An explicit bounded realization]\label{cor:PrimeTwistBoundedRealization}
For every prime $q\equiv3\pmod8$, put
\[
  P_{q,2}=\{\ast\}.
\]
For every integer $e\ge0$, the map
\[
  \rho_{q,e}\colon
  \mathscr{X}_{2,e}(E_q)
  \longrightarrow
  P_{q,2}
  \qquad
  x\longmapsto\ast
\]
is injective. The same construction applies to $F_q$.

Thus the prime-$2$ restriction of Proposition~\ref{prop:BoundedShaRealization} holds uniformly in the parameter and in the descent depth, with
\[
  \#P_{q,2}=1.
\]
\end{corollary}

\begin{proof}
Theorem~\ref{thm:PrimeTwistAllTwoPowerShadows} proves that every source group has precisely one element. The displayed map is therefore injective.
\end{proof}

\begin{remark}[The scope of the bounded target]\label{rem:PrimeTwistPrimaryBoundScope}
The bounded target applies to the subcategory of levels
\[
  1,2,2^2,2^3,\ldots.
\]
Proposition~\ref{prop:BoundedShaRealization} for the whole Shafarevich presentation calls for a bound at every positive integer level. The present theorem supplies no bound for the shadows at the odd primes.

Within the two-axis terminology, the prime-$2$ column closes at depth $e_2=0$. The horizontal closure assertion of Definition~\ref{def:TwoAxisClosureDatum} remains outside this calculation.
\end{remark}

\subsection{The quantitative interpretation}\label{subsec:PrimeTwistQuantitativeInterpretation}

\begin{proposition}[Vanishing index costs on the closed column]\label{prop:PrimeTwistZeroIndexCosts}
Restrict the category of Definition~\ref{def:ShaIndexCategory} to the levels $2^e$ for $E_q$. For every pair $e\le f$, one has
\[
  I(2^e,2^f)=1
  \qquad
  c(2^e,2^f)=0
  \qquad
  H(2^e)=0.
\]
The same identities hold for $F_q$.
\end{proposition}

\begin{proof}
Every finite shadow on the column has order $1$ by Theorem~\ref{thm:PrimeTwistAllTwoPowerShadows}. Substitute these orders into the definitions of the index, the cost, and the height.
\end{proof}

\begin{remark}[Arithmetic origin of the zero cost]\label{rem:PrimeTwistArithmeticZeroCost}
The zero cost follows from the arithmetic exhaustion of the isogeny Selmer groups. The local exclusions and the rational witnesses prove that the Shafarevich quotient has one element, and the quantitative category then expresses this equality of orders along every transition.

The eight candidate square classes, their local tests, and their rational witnesses remain available as finite derivation data. A vanishing index cost measures the absence of growth in the quotient. It places no requirement upon the length of the computation producing the certificate.
\end{remark}

\begin{corollary}[Odd-level reduction and the two-primary BSD factor]\label{cor:PrimeTwistBSDTwoFactor}
Let $q\equiv3\pmod8$ be prime. For $n=2^e m$ with $m$ odd, the transition maps give
\[
  X_m(E_q)\xrightarrow{\sim}X_n(E_q)
  \qquad\text{and}\qquad
  X_m(F_q)\xrightarrow{\sim}X_n(F_q).
\]
For every odd $m$, the isogeny induces an isomorphism
\[
  \phi_q\colon X_m(E_q)\xrightarrow{\sim}X_m(F_q).
\]
Thus the two-primary Shafarevich factor is $1$, with logarithmic contribution $0$, and the two curves have equal shadow orders at every level. In classical arithmetic, both Tate--Shafarevich groups are finite, with the same odd square order.
\end{corollary}

\begin{proof}
The primary decomposition of $X_n$ separates its subgroup killed by $2^e$ from its subgroup killed by $m$. Torsion compatibility identifies the first with $X_{2^e}$, which vanishes by Theorem~\ref{thm:PrimeTwistAllTwoPowerShadows}, and the second with the image of $X_m$. This proves the first two isomorphisms.

Choose an integer $u$ with $2u\equiv1\pmod m$. The composition identities show that $[u]\widehat\phi_q$ is a two-sided inverse to $\phi_q$ on the $m$-torsion shadows. These isomorphisms commute with the transition maps, so they identify the coded Shafarevich presentations after reduction to odd levels. By Proposition~\ref{prop:PrimeTwistMWPackage} and Theorem~\ref{thm:PrimeTwistShaTwoVanishing}, the $2$-Selmer group of $E_q$ is generated by rational $2$-torsion. Smith's theorem gives analytic rank zero and classical finiteness of $\Sha(E_q/\Q)$~\cite[Theorem~1.2 and Corollary~1.3]{Smith2016Congruent}. The shadow identifications transfer finiteness to $F_q$ with the same order. This order is odd by two-primary vanishing and is a square by Corollary~\ref{cor:ShaSquareOrder}.
\end{proof}

\subsection{The finite certificate}\label{subsec:PrimeTwistFiniteCertificate}

\begin{definition}[A prime-twist closure certificate]\label{def:PrimeTwistClosureCertificate}
For a prime $q\equiv3\pmod8$, a \emph{prime-twist closure certificate} consists of the following finite data:

\begin{enumerate}
\item A primality certificate for $q$ and the integer equality
\[
  q=8k+3
\]
for some $k\ge0$.

\item The ordered list
\[
  \mathcal{D}_q
  =
  (1,-1,2,-2,q,-q,2q,-2q).
\]

\item The four rational witnesses of Lemma~\ref{lem:MinusRationalWitnesses} and the rational witness for the plus class $1$.

\item For the even minus classes, the primitive congruence obstruction modulo $16$.

\item For the negative plus classes, the real sign obstruction.

\item For the plus classes $q$ and $2q$, the primitive congruence obstructions modulo $q^2$.

\item For the plus class $2$, the primitive congruence obstruction modulo $32$.

\item The finite isogeny formulas, the Kummer labels, and the instantiated kernel and composition certificates from Definition~\ref{def:PrimeTwoIsogenyInterface}.
\end{enumerate}
\end{definition}

\begin{proposition}[Finite verification]\label{prop:PrimeTwistCertificateVerification}
The numerical part of a prime-twist closure certificate is verified by terminating integer procedures. These procedures can use exhaustive finite residue tables, and their correctness is uniform in $q$.
\end{proposition}

\begin{proof}
Primality can be checked by trial division up to $q-1$, or by verifying a supplied primality certificate. The congruence $q\equiv3\pmod8$ is an integer division test. Substitution checks each rational witness.

For the even minus classes, enumerate the triples $(U,V,W)$ modulo $16$ with at least one of $U,V$ odd. Lemma~\ref{lem:MinusEvenClassObstruction} proves that no triple satisfies the relevant equation.

For the plus class $2$, enumerate the primitive triples modulo $32$ in
\[
  W^2=2(U^4+q^2V^4).
\]
Lemma~\ref{lem:PlusTwoAdicObstruction} proves that none survives.

For the plus classes $q$ and $2q$, enumerate the triples modulo $q^2$ with $(U,V)$ non-zero modulo $q$. Lemma~\ref{lem:PlusParameterPrimeObstruction} excludes every such triple.

The real obstruction is an identity concerning the signs of integer coefficients and even powers. The support argument of Lemma~\ref{lem:PrimeTwistSquareclassSupport} is a uniform valuation proof covering every prime outside $\{2,q\}$, and it removes the need to test further local places.

Every enumeration has an explicit finite bound. The symbolic proofs permit a shorter verification, while the finite tables provide a direct alternative.
\end{proof}

\begin{example}[The parameter three]\label{ex:PrimeTwistParameterThree}
Take $q=3$. Then
\[
  E_3\colon y^2=x^3-9x
  \qquad\text{and}\qquad
  F_3\colon Y^2=X^3+36X.
\]
The minus Selmer group is
\[
  \mathcal{S}_3^-=\{1,-1,3,-3\}
\]
and the plus Selmer group is
\[
  \mathcal{S}_3^+=\{1\}.
\]

For the plus classes $3$ and $6$, the obstruction is checked modulo $9$. For the plus class $2$, it is checked modulo $32$. The four even minus classes are excluded modulo $16$. The rational witnesses are the entries of Lemma~\ref{lem:MinusRationalWitnesses} with $q=3$.

These finite data prove
\[
  \Sha(E_3/\Q)[2^e]
  =
  \Sha(F_3/\Q)[2^e]
  =
  0
\]
for every integer $e\ge0$, through one application of the first-shadow calculation and an induction on the depth.
\end{example}

\subsection{Proof-theoretic scope and the remaining primes}\label{subsec:PrimeTwistFormalScope}

\begin{proposition}[Intuitionistic derivation relative to descent]\label{prop:PrimeTwistIntuitionisticScope}
Over $\mathsf{T}^{\mathrm{int}}_q$, the arguments of this section derive the bounded realization of Corollary~\ref{cor:PrimeTwistBoundedRealization} for every certified prime $q\equiv3\pmod8$.

The numerical exclusions and the witness checks are formalizable in Heyting arithmetic. The transfer to the finite Shafarevich shadows uses the specified two-isogeny descent correctness assumptions. The classical package construction of Proposition~\ref{prop:PrimeTwistMWPackage} and the classical finiteness conclusion of Corollary~\ref{cor:PrimeTwistBSDTwoFactor} use additional theorems and are outside this derivation.
\end{proposition}

\begin{proof}
The local exclusions use divisibility, parity, residue arithmetic, and exponentiation in a finite field. Fermat's identity in a prime residue field is proved by finite arithmetic. The rational witnesses are polynomial identities in $q$.

The exhaustion of the eight candidates is a finite case analysis. The two isogeny quotients are obtained through the declared exact sequences. Their vanishing implies the vanishing of the first two-primary shadows by the composition identities. Induction and torsion compatibility then give every greater depth.

All the case distinctions concern decidable integer relations or finite group equality. Every search has a specified finite domain. No step invokes excluded middle for arbitrary propositions, a choice principle, compactness, or a comprehension forming the entire Tate--Shafarevich group.
\end{proof}

\begin{remark}[Uniformity in the parameter]\label{rem:PrimeTwistParameterUniformity}
The certificate is produced from the finite input $q$ and its primality verification. The proof applies to every input satisfying the stated congruence, and no distribution theorem for primes is used.

The bound on the primary shadows is the constant $1$, independent of $q$ and of the descent depth. The size of an exhaustive verification table may depend upon $q$. These are different bounds, and the construction keeps both visible.
\end{remark}

\begin{remark}[Classical finiteness and finite closure data]\label{rem:PrimeTwistRemainingPrimes}
Classical finiteness for this family follows from the additional theorem used in Corollary~\ref{cor:PrimeTwistBSDTwoFactor}. The explicit descent calculation supplies a uniform certificate for the two-primary column. A full two-axis datum inside the declared arithmetic theory would also require finite certificates for the odd-primary columns and a proof of a horizontal cutoff in that theory. The cited classical theorem is not included among its assumptions.

A further arithmetic example with a non-zero first shadow and a vanishing next increment would exhibit closure at positive depth. Such an example calls for certified higher descent or pairing data in addition to the first Selmer calculation.
\end{remark}

\bibliographystyle{smfalpha}
\bibliography{intensional_semantics}

\end{document}